%% file: main.tex
\documentclass[11pt]{amsart}

\input{settings.tex}

\author[Bell]{Mark Bell}
\address[Mark Bell]{
  Independent\\
  UK.}
\email{mark00bell@googlemail.com}

\author[Delecroix]{Vincent Delecroix}
\address[Vincent Delecroix]{
 CNRS - Universitsé de Bordeaux\\
 351, cours de la Libération\\
 33400 Talence. \href{https://orcid.org/0000-0002-9608-782X}{ORCID: 0000-0002-9608-782X}.
}
\email{vincent.delecroix@u-bordeaux.fr}

\author[Gadre]{Vaibhav Gadre}
\address[Vaibhav Gadre]{
  School of Mathematics and Statistics\\
  University of Glasgow\\
  University Place, Glasgow, G128QQ UK. \href{https://orcid.org/0000-0003-4222-9551}{ORCID: 0000-0003-4222-9551}.
}
\email{vaibhav.gadre@glasgow.ac.uk}

\author[Gutiérrez-Romo]{Rodolfo Gutiérrez-Romo}
\address[Rodolfo Gutiérrez-Romo]{
        Departamento de Ingeniería Matemática \& Centro de Modelamiento Matemático, CNRS-IRL 2807, Universidad de Chile, Beauchef 851, Santiago, Chile. \href{https://orcid.org/0000-0002-5446-6602}{ORCID: 0000-0002-5446-6602}.
        }
\email{g-r@rodol.fo}

\author[Schleimer]{Saul Schleimer}
\address[Saul Schleimer]{
    	Department of Mathematics\\
        University of Warwick\\
        Coventry, CV47AL UK. 
        \href{https://orcid.org/0000-0003-3362-7822}{ORCID: 0000-0003-3362-7822}.}
\email{s.schleimer@warwick.ac.uk}



\thanks{This work is in the public domain.}

\title{Diagonal flow detects topology of strata}

\begin{document}

\begin{abstract}
We study the interplay between the diagonal flow on, and the topology of, a stratum component of a space of rooted quadratic differentials.
We prove that the flow group -- the subgroup of the fundamental group generated by almost-flow loops -- equals the fundamental group.
As a corollary, we show that the plus and minus modular Rauzy--Veech groups are finite-index subgroups of their ambient modular monodromy groups. 
This partially answers a question of Yoccoz.

Using this, and recent advances on algebraic hulls and Zariski closures of symplectic monodromy groups, we prove that the Rauzy--Veech groups are Zariski dense in their ambient symplectic groups.  
Density, in turn, implies the simplicity of the plus and minus Lyapunov spectra of any component of any stratum of quadratic differentials.
We thus establish the Kontsevich--Zorich conjecture.
\end{abstract}

\maketitle

\section{Introduction}

\subsection*{Motivation}

A Riemann surface $S$ equipped with an (abelian or quadratic) differential $q$ has a canonical flat metric. 
This metric determines, and is determined by, the \emph{periods} of the differential: 
its integrals along appropriately chosen relative cycles.
Varying $q$, without changing the underlying combinatorics, gives rise to an ambient \emph{stratum component} $\calC$: a moduli space of differentials.  
(The possible components are completely classified; see~\cite{Kon-Zor03, Lan08, Che-Moe}.)
The periods endow the component with (orbifold) charts known as \emph{period coordinates}.
The usual identification of $\CC$ with $\RR^2$ gives an $\SL(2, \RR)$--action on the period coordinates. 

As a hint of deep importance of these spaces, and this action, suppose that $q$ is a differential and $\calN$ is its $\SL(2, \RR)$--orbit closure. 
Then $\calN$ is dynamically defined, yet is an algebraic variety~\cite{Fil16}.
Futhermore, in period coordinates $\calN$ is cut out by homogeneous linear equations, with (real) algebraic coefficients~\cite{Esk-Mir-Moh}. 
So, despite the essentially transcendental dependence of flat surfaces on their differentials, their dynamics surprisingly behave much like a Lie group acting on a homogeneous space. 

\subsection*{Monodromy}

Our goal is to understand the \emph{monodromy} associated to the component $\calC$. 
There is a forgetful map from $\calC$ to $\calM(S)$: 
the moduli space of Riemann surface structures on $S$.
Both $\calC$ and $\calM(S)$ are orbifolds; 
both have various manifold covers. 

The (orbifold) universal cover of $\calM(S)$ is the so-called Teichmüller space $\calT(S)$; 
this is homeomorphic to an open ball in $\RR^{6g-6}$ where $g = \genus(S)$.
The deck group is the mapping class group $\Mod(S)$.

To obtain finite manifold cover of $\calC$ we consider \emph{rooted} differentials. 
A \emph{root} for $q$ is a horizontal unit tangent vector at a singularity.  
The choice of root breaks all symmetries of $q$
and so unwraps the orbifold locus.
The resulting finite cover is a manifold. 
We fix a connected component of this space and denote it by $\calC_\root$.  

The above maps of spaces induce the following sequence of homomorphisms:
\[
\uppi_1(\calC_\root) \to \uppi_1^\orb(\calC) \to \uppi_1^\orb(\calM(S)) \isom \Mod(S) \xrightarrow{\rho} \Aut(\upH_1(S; \ZZ)) \cong \Sp(2g, \ZZ)
\]
The third map, $\rho$, is the symplectic representation of the mapping class group.
We call the image of $\uppi_1(\calC_\root)$, inside the mapping class group, the \emph{modular monodromy group}. 
The image of the modular monodromy group under $\rho$ is known as the \emph{symplectic monodromy group}.

Monodromy groups are ``topological offspring'' of the stratum component $\calC_\root$.
In order to relate the topology and dynamics of $\calC_\root$ we ask the following.

\begin{question}
	\label{q:local_global_vague}
	Is $\uppi_1(\calC_\root)$ ``detected'' by the diagonal flow? 
\end{question}

To make this precise, suppose that $U \subset \calC_\root$ is a small contractible open set.
Suppose that $q_0 \in U$ is a base-point.
Consider the diagonal flow trajectories that start and end in $U$.
For each, we connect its endpoints to $q_0$ inside of $U$ to get a loop based at $q_0$.
As $U$ is contractible, the resulting based homotopy class is independent of the choices made inside of $U$. 
We call these based homotopy classes \emph{almost-flow loops}. 
The \emph{flow group} of $\calC_\root$ associated with the pair $(U, q_0)$ is the subgroup of $\uppi_1(\calC_\root, q_0)$ generated by all such loops. 
(A semigroup version, and some of its applications, are discussed by Hamenstädt~\cite[Section 4.2]{Ham}.)
So, \Cref{q:local_global_vague} can be restated as follows.

\begin{question}
	\label{q:local_global}
	Suppose that $\calC_\root$ is a stratum component of the moduli space of (rooted) abelian or quadratic differentials.  
	Is the flow group of $\calC_\root$ equal to $\uppi_1(\calC_\root)$?
\end{question}

\noindent
This is a version of a ``question'' of Yoccoz~\cite[Remark~9.3]{Yoc}.\footnote{Yoccoz asks if the fundamental group of the Rauzy diagram surjects $\Mod(S)$.
	However, what is \emph{meant} is the monodromy of $\calC$ inside of the ``obvious'' subgroup of $\Mod(S)$~\cite{MatPer}.}

One of our main results is a positive answer to \Cref{q:local_global}, in the setting of rooted differentials.

\begin{restate}{Theorem}{t:flow-general}
	Suppose that $\calC_\root$ is any component of a stratum of the moduli space of rooted abelian or quadratic differentials. 
	Suppose that $U$ is any contractible open set and $q_0 \in U$ is a choice of base-point. 
	Then the flow group of $\calC_\root$, associated with the pair $(U, q_0)$, equals $\uppi_1(\calC_\root, q_0)$.
\end{restate}

\noindent
We also define the \emph{strict flow group} associated with the pair $(U, q_0)$.
This is the subgroup of the flow group of $\calC_\root$ generated by those elements which are \emph{closed} diagonal flow trajectories. 
We then directly show that the strict flow group and the flow group of $\calC_\root$ coincide.
This implies the following, suggested by Forni:

\begin{restate}{Corollary}{c:strict-flow}
	Suppose that $\calC_\root$ is any component of a stratum of the moduli space of rooted abelian or quadratic differentials. 
	Suppose that $U$ is any contractible open set and $q_0 \in U$ is a choice of base-point. 
	Then the strict flow group of $\calC_\root$, associated with the pair $(U, q_0)$, equals $\uppi_1(\calC_\root, q_0)$.
\end{restate}

\noindent 
Through the zippered rectangles construction, the diagonal flow on $\calC_\root$ can be coded combinatorially by a Rauzy diagram.
In the course of the proof of \Cref{t:flow-general}, we prove the following. 
\begin{restate}{Theorem}{t:pi1-surjective}
	Let $\calD_\root$ be the Rauzy diagram for $\calC_\root$. 
	Then the natural homomorphism  $\uppi_1(\calD_\root) \to \uppi_1(\calC_\root)$ is surjective.
\end{restate}

For strata of abelian differentials, previous work by Calderon and Calderon--Salter also allows us to explicitly compute the image of the flow group inside of $\Mod(S)$ and of $\Aut(\upH_1(S; \ZZ))$ (or some larger group, such as $\Aut(\upH_1(S, Z; \ZZ))$), up to finite index \cite{Cal,Cal-Sal21a,Cal-Sal21b,Cal-Sal23}.

\subsection*{Cocycles}

Fix $\calC$, a stratum component. 
Given a bundle over $\calC$, the diagonal flow gives us a natural cocycle. 
The most studied of these is the \emph{Kontsevich--Zorich cocycle}.
This can be lifted to a connected component $\calT\calC$ of the Teichmüller space of abelian or quadratic differentials (the choice of $\calT\calC$ is, in general, not unique \cite{Cal,Cal-Sal23}). 

In more detail, we define a vector bundle over $\calT\calC$ with a suitable fibre. 
In the abelian case, this fibre is the first cohomology of the underlying topological surface; 
in the quadratic case, it is the first cohomology of the orientation double cover. 
By Poincaré-duality, it is also possible to use the corresponding homology groups as the fibre.

The $\SL(2, \RR)$--action induces a trivial dynamical cocycle on this vector bundle. 
By modding out by the mapping class group, the vector bundle descends to a bundle over $\calC$ known as the \emph{Hodge bundle}; 
similarly the cocycle descends to the \emph{Kontsevich--Zorich cocycle} \cite{Kon-Zor97,Kon}. 
In the quadratic case, the cocycle naturally splits into two distinct symplectically orthogonal blocks, usually referred to as the \emph{plus} (or \emph{invariant}) and \emph{minus} (or \emph{anti-invariant}) pieces.

Moduli spaces of abelian or quadratic differentials carry a natural $\SL(2,\RR)$--invariant measure called the Masur--Smillie--Veech measure. 
The cocycle is measurable and log-integrable; by Oseledets' theorem, it has Lyapunov exponents. 

Many interesting dynamical properties of abelian or quadratic differentials can be written in terms of the Lyapunov exponents of the Kontsevich--Zorich cocycle.
An important example are the deviations of ergodic averages of the linear flow on almost every abelian or quadratic differential \cite{Zor97,Esk-Kon-Zor}. 
In fact, when the Lyapunov spectrum of the Kontsevich--Zorich cocycle is simple, these deviations can be precisely described. 

Kontsevich and Zorich conjectured that the Lyapunov spectrum is simple for all abelian stratum components~\cites[Conjecture~2]{Zor97}[page~1499]{Zor99}.  
Their conjecture extends naturally to the quadratic case as follows. 
We form the branched orientation double cover. 
The homology of the cover splits into the plus and minus eigenspace for the involution; 
the $\SL(2, \RR)$ action preserves this splitting. 
Simplicity is conjectured in both pieces~\cite{ZorPer}.

In the abelian case, Forni proved positivity of the exponents and simplicity in genus two \cite{For}. 
Simplicity for all abelian strata was later established in the famous work by Avila--Viana \cite{Avi-Via07b}.
Similar techniques were used by Matheus--Möller--Yoccoz \cite{Mat-Moe-Yoc} to prove simplicity for certain loci of square-tiled surfaces.
A coding-free proof of this was then given by Eskin--Matheus \cite{Esk-Mat}.

Simplicity in the quadratic case was shown for many stratum components \cite{Tre,Gut17}.  
Our paper establishes simplicity for all abelian and quadratic stratum components; as discussed below, our proof relies on certain machinery of these previous authors, but is independent of their theorems.  

\subsection*{Rauzy--Veech groups}

The Rauzy--Veech groups are subgroups of the symplectic group generated by the matrices (in a preferred basis) induced by evaluating these cocycles over based loops in Rauzy diagrams. 
It follows from \Cref{t:pi1-surjective} that Rauzy--Veech groups have finite index in the corresponding symplectic monodromy groups. We use this fact to prove the following:

\begin{restate}{Theorem}{t:zariski}
	For every component of every abelian stratum, the Rauzy--Veech group is Zariski dense in its ambient symplectic group.
	Furthermore, for every component of every quadratic stratum, the plus and minus Rauzy--Veech groups are Zariski dense in their corresponding ambient symplectic groups.
\end{restate}

The groups of \Cref{t:zariski} that arise, by splitting singularities, from abelian strata are known to be finite index inside the ambient symplectic groups (over $\ZZ$) and hence Zariski dense. 
This was shown by Avila--Matheus--Yoccoz \cite{Avi-Mat-Yoc} for all abelian hyperelliptic components.
It was extended to all components (abelian or quadratic) that arise from minimal abelian strata by splitting singularites by the fourth author \cite{Gut19, Gut17}. 

With \Cref{t:flow-general} in hand, we can compute the Kontsevich--Zorich cocycle over any loop in $\calC_\root$ and not just along the diagonal flow.
It is this additional flexibility that allows us, for the question of Zariski density, to consider a symplectic monodromy group instead of a Rauzy--Veech group.

For the symplectic monodromy groups of abelian differentials, and also for the symplectic monodromy groups induced by the minus piece of the cocycle for quadratic differentials, we directly apply some of Filip's results to obtain Zariski density \cite[Corollary 1.7]{Fil17}. For the symplectic monodromy groups induced by the plus piece of the cocycle, we need to discuss \emph{algebraic hulls}.

The algebraic hull of the Kontsevich--Zorich cocycle restricted to a linear invariant suborbifold $\calN$ can be thought of as the smallest algebraic group into which the cocycle over $\calN$ can be measurably conjugated.
As such, the hull is both an algebro-geometric and an ergodic-theoretic object.
Eskin--Filip--Wright showed that the algebraic hull is as large as it can be, namely it equals the stabiliser of the tautological plane (that is, the cohomology classes spanned by the real and imaginary parts of the differential) in the Zariski closure of the symplectic monodromy group \cite[Theorem 1.1]{Esk-Fil-Wri}. 
We remark that, just as the theorem by Eskin--Filip--Wright shows that the algebraic hull is as large as it can be, \Cref{t:flow-general} shows that the flow group is also as large as it can be. Thus, for stratum components, \Cref{t:flow-general} can be considered as a dynamical analogue of the result by Eskin--Filip--Wright.

The plus piece of the Kontsevich--Zorich cocycle does not meet the tautological plane. The stabiliser then equals the Zariski closure of the symplectic monodromy group, and hence so does the algebraic hull. This result, together with Filip's classification of the possible Lie algebra representations of algebraic hulls \cite[Theorem 1.2]{Fil17}, enables us to show that the Zariski closure of the symplectic monodromy group of the plus piece is $\Sp(2g, \RR)$ by a simple dimension count.

\subsection*{Simplicity} By the work of Benoist \cite{Ben97}, Zariski density of an appropriate Rauzy--Veech group implies that the monoids associated with the Kontsevich--Zorich cocycles are ``rich'' in the sense of the simplicity criterion of Avila--Viana \cite{Avi-Via07a, Avi-Via07b}. 
As a consequence of \Cref{t:zariski}, we can apply the Avila--Viana criterion to prove the Kontsevich--Zorich simplicity conjecture.

As mentioned before, simplicity was known for all abelian \cite{Avi-Via07b} and some quadratic stratum components \cite{Gut17}.
It is also known for the principal stratum of quadratic differentials by different methods through the recently announced solution by Eskin--Mirzakhani--Rafi of the Furstenberg problem for random walks on the mapping class group. However, we have claimed the known results as our proof is self-contained and is uniform across all stratum components.

\begin{restate}{Theorem}{t:KZ}
	The Kontsevich--Zorich cocycle has a simple spectrum for all components of all strata of abelian differentials. The plus and minus Kontsevich--Zorich cocycles also have a simple spectrum for all components of all strata of quadratic differentials. 
\end{restate}

\subsection*{Acknowledgements}
The authors are immensely grateful to Carlos Matheus for countless illuminating conversations. 
We also thank Giovanni Forni, Maxime Fortier Bourque, Erwan Lanneau, and Alex Wright for their helpful comments on an earlier version of this article.
We thank Diane Maclagan for many helpful conversations about convex geometry.
We thank Cameron Wilson for sharing with us his elegant proof of \Cref{r:strong}.

The fourth author is grateful to the ANID AFB-170001, the FONDECYT Iniciación 11190034, and the MATHAMSUD 21-MATH-07 grants.
The third author was supported by LMS Scheme 4 during the final revision of the paper.

\section{Strategies}

We outline the key steps and ideas in our proofs.

\subsection*{Rooted differentials}

The dynamical issues we consider are stable under passing to a finite cover of the given stratum component $\calC$.
Accordingly, in \Cref{s:rooted} we pass to the space $\calC_\root$ of \emph{rooted differentials}:
differentials decorated with a choice of horizontal unit tangent vector at a singularity.
We do this for two reasons.
Most importantly, a generic rooted differential admits a canonical decomposition into \emph{zippered rectangles}. 
Also, while $\calC$ has a complicated orbifold structure, the cover $\calC_\root$ is a manifold; 
this simplifies various transversality and fundamental group arguments.

\subsection*{Zippered rectangles and the based-loop theorem} 

The zippered rectangle construction is due to Veech \cite{Vee82} in the abelian case and due to Boissy--Lanneau \cite{Boi-Lan} in the quadratic case. We discuss them in depth in \Cref{s:zippered_rectangles}.
Parameter spaces of zippered rectangles, where the length of base-arc is chosen in a specific way (that we call \emph{distinguished}), define contractible open sets in $\calC_\root$ which we call \emph{polytopes of differentials}.
See \Cref{s:polytopes}.
The union of these polytopes is dense in $\calC_\root$. 
However, the complement of their union is complicated; 
in particular the polytopes do not give $\calC_\root$ a CW-complex structure.
For instance, there are compact arcs in $\calC_\root$ that intersect, transversely, the faces of the polytopes infinitely many times.
See \Cref{a:non-polytopal} for an example and relevant discussions.
As a result, our \emph{based-loop theorem} (\Cref{t:based-loops}) cannot be deduced from naïve transversality arguments.

Fortunately, as discussed by Yoccoz~\cite[Proposition in Section 9.3]{Yoc}, the subset of rooted differentials that do \emph{not} admit any decomposition into zippered rectangles is a countable union of codimension-two subsets. See \Cref{d:saddled} and \Cref{l:base-arc-infinite}.
Thus, any based loop $\gamma \from [0, 1] \to \calC_\root$ can be homotoped to be disjoint from such differentials.

After this homotopy, we cover the image of $\gamma$ by finitely many charts with good properties. 
We arrange matters so that the boundaries of these charts are codimension-one embedded submanifolds. 
A further homotopy makes $\gamma$ transverse to these boundaries while still being covered by the charts. 
Unfortunately, our charts may not be contained in any of the polytopes defined above.
That is, the base-arcs along $\gamma$ may not be the distinguished base-arcs. 
If it is not, we flow; we flow forwards if the base-arc is too short and backwards if the base-arc is too long. 

We apply the flow, with direction determined as above, to a sufficiently small subsegment of $\gamma$, contained in a chart.
We then replace the given subsegment of $\gamma$ by two segments contained in the flow and one segment contained in the interior of a polytope.
This gives a homotopy of $\gamma$. 
Doing this finitely many times, we homotope from $\gamma$ to a concatenation of segments which alternate between being diagonal flow segments (forward or backward) or lying inside of polytopes.
This is \Cref{t:based-loops}, our \emph{based-loop theorem}.

\subsection*{Rauzy--Veech induction and the diagonal flow}

A zippered rectangle decomposition gives combinatorics in the form of an \emph{unlabelled irreducible generalised permutation}~\cite[Definition~3.1]{Boi-Lan}.
The decomposition, and thus the rooted differential, can be recovered from the permutation and various associated parameters such as the widths and heights of the rectangles as well as the heights of the zippers. 

If we apply the forward diagonal flow, the base-arc grows (exponentially) until it is no longer the distinguished base-arc.
At this point we pass to a new, unique, shorter base-arc which is again a distinguished base-arc.
In this way we obtain a new irreducible generalised permutation as well as new parameters.
We call one such operation a \emph{Rauzy--Veech move}.

The collection of these moves gives a renormalisation procedure known as the \emph{Rauzy--Veech induction}. See \Cref{s:RV-induction}.

It was originally defined by Rauzy and Veech for abelian differentials~\cite{Rau, Vee82} and by Boissy and Lanneau~\cite{Boi-Lan} for quadratic differentials.
Applying the diagonal flow to a single (generic) rooted differential, we obtain a sequence of generalised permutations.
This is the \emph{Rauzy--Veech coding} for the given differential.

The generalised permutations and the Rauzy--Veech moves give an automaton (a directed graph), called the \emph{Rauzy diagram}, as follows.
Two permutations $\pi$ and $\pi'$ are equivalent if we can precompose with a permutation $\sigma$ to obtain $\pi \circ \sigma = \pi'$.
The vertices of the automaton are equivalence classes of irreducible generalised permutations arising from differentials in $\calC_\root$.
There is a directed edge from $[\pi]$ to $[\rho]$ if some representative of the latter arises from a single Rauzy--Veech move applied to some representative of the former.
Furthermore, the Rauzy diagram is strongly connected.
This follows from work of Masur~\cite{Mas} and Veech~\cite{Vee82,Vee86} showing that the diagonal flow is ergodic.

We use the Rauzy--Veech coding to derive simplicity, as explained below.

\subsection*{Flow groups and the fundamental group} 

Suppose that $\pi$ is a vertex in the Rauzy diagram. 
Suppose that $q$ lies in $\calC(\pi)$, the polytope of differentials with combinatorics $\pi$.
We define a homomorphism from the fundamental group of the Rauzy diagram (as an undirected graph, based at $\pi$) to $\uppi_1(\calC_\root, q)$. 
That is, for any directed loop in the Rauzy diagram, based at $\pi$, we choose a diagonal flow segment whose coding is the loop, and cone its endpoints to $q$. 
Then strong connectivity of the Rauzy diagram extends this to all based loops.
With this done, we use the based-loop theorem to show that the homomorphism is surjective.
See \Cref{t:pi1-surjective}.
This answers a question of Yoccoz \cite[Remark in Section 9.3]{Yoc}.

Suppose that $U$ is a contractible open set in $\calC_\root$ with a base-point $q$.
From any diagonal flow segment with endpoints in $U$, we obtain an almost-flow loop by coning the endpoints to $q$.
Since shrinking $U$ makes the flow group smaller (in principle), we may assume that $U$ is contained in a polytope $\calC(\pi)$. 
We use the based-loop theorem, and surjectivity, to show that the flow group equals the fundamental group of $\calC_\root$.
See \Cref{t:flow-general}.
In other words, at the level of the fundamental group, the diagonal flow captures the topology of $\calC_\root$ and hence the topology of $\calC$ (up to finite index).

\subsection*{Simplicity}

By a criterion of Avila--Viana~\cite{Avi-Via07a,Avi-Via07b}, simplicity of log-integrable cocycles, such as the Kontsevich--Zorich cocycle, can be deduced from the existence of a coding for the flow that has an ``approximate product structure'' and a notion of ``richness'' for the cocycle. A (weaker) version of this coding is stated in \Cref{c:simplicity}.
As we indicated earlier, a coding with the required product structure
can be achieved by accelerating the Rauzy--Veech induction. 
This was done by Avila--Gouëzel--Yoccoz~\cite{Avi-Gou-Yoc} for abelian differentials and by Avila--Resende~\cite{Avi-Res} for quadratic differentials. 
See \Cref{s:dynamics} for more details.
The remaining task, and the crux of the problem, is to obtain the ``richness'' of the cocycle. 
The required richness was established by Avila--Viana~\cite{Avi-Via07b} for abelian stratum components by a direct computation. 

By the work of Benoist~\cite{Ben97}, Zariski density in the symplectic group of an associated group implies richness of the cocycle. 
In fact, Zariski density is strictly stronger~\cite[Appendix~A]{Avi-Mat-Yoc}).

For the Kontsevich--Zorich cocycle, the associated group is the Rauzy--Veech group. 
Its Zariski density for hyperelliptic components was proved by Avila--Matheus--Yoccoz. 
In fact, their result is stronger; they prove that the Rauzy--Veech group is a certain finite-index subgroup of the ambient symplectic group~\cite[Theorem 1.1]{Avi-Mat-Yoc}. 
This finite index result was extended by the fourth author to all components (abelian or quadratic) that arise by splitting zeroes from minima abelian components~\cites[Theorem 1.1]{Gut19}[Theorem 1.1]{Gut17}.

Our result on flow groups (\Cref{t:flow-general}) is a key step in our plan to prove Zariski density of the Rauzy--Veech group, of \emph{any} stratum component.
This is because we can use the cocycle along any loop in $\calC_\root$; 
we are not restricted to almost-flow loops.

In recent work~\cite{Fil17}, Filip classifies the situations in which zero Lyapunov exponents can arise, in terms of the Zariski closure of a symplectic monodromy group. See \Cref{s:algebraic_hulls_Zariski_closures}. From this description, he also derives the fact that, when restricted to the symplectic block that contains the tautological plane, the Zariski closure of this group is the full symplectic group for this block \cite[Corollary 1.7]{Fil17}.
Combined with this fact, our \Cref{t:flow-general} directly yields simplicity for abelian components.

A quadratic component lifts to a linear invariant suborbifold of its orientation double-cover. 
Hence Filip's result again applies. 
The involution on the orientation double-cover splits the Kontsevich--Zorich cocycle into two symplectically orthogonal blocks, usually referred to as the \emph{plus} (or \emph{invariant}) cocycle and the \emph{minus} (or \emph{anti-invariant}) cocycle. 
We will refer to the induced subspaces, inside the relative homology of a fixed flat surface, as the \emph{plus} and \emph{minus piece}, respectively.
The minus piece contains the tautological plane.
Again by Filip's corollary, the Zariski closure for the minus cocycle is the full symplectic group.
Simplicity of the minus cocycle follows directly from combining this with \Cref{t:flow-general}. See \Cref{s:minus}.

It remains to tackle the plus cocycle. Filip also classifies the possible algebraic hulls of any linear invariant suborbifold at the level of Lie algebra represntations \cite[Theorem 1.2]{Fil17}. Moreover, Eskin--Filip--Wright showed that \cite[Theorem 1.1]{Esk-Fil-Wri} the algebraic hull and the Zariski closure of a piece of the symplectic monodromy group not containing the tautological place coincide. Hence, this applies for the plus piece.

Now, we exploit our extra flexibility to build a dimension argument that eliminates all but the full symplectic group as the Zariski closure.
We carry out the dimension argument first for components of minimal strata (\Cref{s:minimal strata}) and hyperelliptic components with two zeros (\Cref{s:hyperelliptic with two}) to conclude Zariski density for the symplectic monodromy groups of these components. 
This implies the Zariski density of their Rauzy--Veech groups, as they are finite index in the symplectic monodromy groups (a consequence of \Cref{t:pi1-surjective}). 
We then deal with a few remaining low genera components by using a well-known criterion for Zariski density \cite{Pra-Rap} in \Cref{s:sporadic}. 
Finally, we extend the density to Rauzy--Veech groups of all quadratic components by standard techniques of surgery/splitting zeroes (\Cref{s:extending_from_basic}). 
The density allows us to apply the Avila--Viana criterion to conclude the proof of the Kontsevich--Zorich conjecture in full generality.

\section{Preliminaries}

\subsection{Differentials and stratum components}
\label{s:differentials_components}





Suppose that $P$ is a \emph{polygon}: 
a compact, convex region in $\CC$ with boundary a finite union of line segments.
Suppose that $Z = Z(P)$ is a finite subset of $P$.  
We call $Z$ the \emph{marked points} of $P$.  
We call the pair $(P, Z)$ a \emph{marked polygon}.

A \emph{(completely marked) flat surface} $(S, q)$ is a finite collection $\{(P, Z(P))\}_P$ of marked polygons, 
together with \emph{side pairings} as follows.
\begin{enumerate}
	\item 
	A side pairing is a (half-)translation which reverse the induced orientation of the paired boundary segments.
	\item
	A side pairing between $(P, Z(P))$ and $(Q, Z(Q))$ must send points of $Z(P)$ to points of $Z(Q)$.
	\item 
	Every boundary segment is part of exactly one side pairing.
\end{enumerate} 
The \emph{underlying surface} $S$ is the quotient of the disjoint union $\bigsqcup_P P$ by the side pairings.
We deduce that $S$ is compact, without boundary, and (canonically) oriented.
We will now abuse notation and use simply $q$ to denote our flat surface.
The set of \emph{singularities} $Z(q)$ is the image, under the quotient map, of $\bigsqcup_P Z(P)$.  

Topological and geometric properties in the complex plane $\CC$, preserved by (half-)translations, descend to flat surfaces.
This applies to orientations, to the real (horizontal) and imaginary (vertical) foliations, to geodesics (line segments), and to angles.
These give orientations, \emph{horizontal} and \emph{vertical} foliations, \emph{flat geodesics}, and \emph{cone angles} to flat surfaces.

A point of $q$ with cone angle $k\pi$ is a \emph{pole}, a \emph{regular point}, or a \emph{zero} if $k$ equals one, two, or more, respectively.
We require that $Z(q)$ contain all poles and zeros of $q$.
Note that $Z(q)$ may contain (finitely many) regular points -- these are the \emph{marked regular points}.
Note that at a regular point the tangent space to $q$ is a plane;
at poles and zeros the tangent space is a cone of the corresponding angle.

\begin{definition}
	Suppose that $q$ is a flat surface.
	Then $K(q)$ is the multi-set of cone angles (divided by $\pi$) appearing at the singularities (that is, points of $Z(q)$). 
	We call $K(q)$ the \emph{cone angle data} for $q$. 
\end{definition}

Note that the sum $\sum_{k \in K(q)} (k - 2) = 4g - 4$ recovers the genus of the underlying surface $S$.
For example, if $Z(q)$ contains only marked regular points then $S$ is the two-torus.

A (completely marked) flat surface $q$ comes equipped with a decomposition into polygons.
Thus a pair of flat surfaces may be isometric, with identical regular marked points and identical vertical and horizontal foliations, yet not be ``the same''.
To deal with this, we say that a \emph{differential} is a scissors congruence class of flat surfaces.
We again abuse notation and use simply $q$ to denote the differential represented by $(S, q)$. 

\begin{definition}
	\label{d:stratum-component}
	The set of differentials sharing the same cone angle data $K$ is called a \emph{stratum}.
	We equip a stratum with the quotient topology (by scissors congruence) of the subspace topology (enforcing gluings by (half-)translations) of the product topology (coming from the vertices and marked points of the given polygons).
	A connected component $\calC$ of a stratum is called a \emph{(stratum) component}.
\end{definition}

In general, a stratum component $\calC$ is an \emph{orbifold}. 
We refer to the book by Boileau--Maillot--Porti~\cite{Boi-Mai-Por} for background on orbifolds and their fundamental groups.
Here are the combinatorial and algebraic invariants shared by all differentials $q$ in a fixed stratum component $\calC$.

\begin{enumerate}
	\item 
	\emph{Abelian}
	or \emph{quadratic}:
	whether or not the vertical foliation of $q$ is orientable.  
	
	\item
	\emph{Singularity data}:
	the multi-set
	\[
	\kappa(q) = 
	\begin{cases}
		\left\{k - 2 \st k \in K(q)\right\}, & \mbox{if $q$ is quadratic} \\
		\left\{(k - 2)/2 \st k \in K(q)\right\}, & \mbox{if $q$ is abelian}
	\end{cases}
	\]
	Each element of $\kappa(q)$ is an integer -- in the abelian case this follows from orientability of the vertical foliation.
	
	\item
	\emph{Hyperelliptic}: if every $q$ lying in $\calC$ admits an involution with $2g + 2$ fixed points~\cite{Lan04}.
	
	\item
	\emph{Spin}: (only for abelian components where all elments of $\kappa(q)$ are even) defined as the Arf invariant of a specific quadratic form~\cites{Joh}[Appendix~C]{Zor08}.
	
	\item
	\emph{Regular} or \emph{irregular}: (when possible) distinguished by the dimension of a cohomology group corresponding to a specific divisor \cite{Che-Moe}.
\end{enumerate}

For the convenience of the reader, we also list the complete classification of abelian and quadratic stratum components in \Cref{s:classification}.

\subsection{Saddle connections, triangulations, and period coordinates} 

Suppose that $q$ is a differential.  A \emph{saddle connection} for $q$ is a flat geodesic that meets the singularities $Z(q)$ exactly in its endpoints.
(Note that a saddle connection may be a loop.)
By~\cite[Section~4]{Mas-Smi}, when $Z(q)$ is non-empty there is a triangulation of $q$ where the vertices are exactly the points of $Z(q)$ and where the edges are saddle connections. 
If we choose an order on the edges of the triangulation, then their complex lengths give a vector which we call the \emph{(redundant) period coordinates} of $q$.
Note that these coordinates are functions on (small) manifold charts about orbifold points of $\calC$. 
If $q$ is quadratic then there are various ambiguities of sign.  
We deal with these in~\Cref{s:orientation-covers}.
Also, we give a less redundant version of period coordinates in \Cref{s:singularity_parameters}.


\subsection{\texorpdfstring{$\SL(2,\RR)$}{SL(2,R)}-action} 

The usual action of $\SL(2,\RR)$ on $\RR^2 \isom \CC$ preserves side pairings of polygons and also scissors congruence.
As a result, it descends to an action on differentials, preserving stratum components. 
The diagonal part of the $\SL(2,\RR)$--action gives the \emph{diagonal flow} on $\calC$. 
This is also called the \emph{Teichmüller flow} on $\calC$.  

By the famous work of Eskin--Mirzakhani--Mohammadi~\cite{Esk-Mir-Moh}, closures of $\SL(2,\RR)$--orbits inside of (a certain manifold cover of) $\calC$ are submanifolds cut out by linear equations (with real coefficients and no constant terms) in period coordinates. 
Such an orbit closure is called a \emph{linear invariant submanifold}. 

\subsection{Rooted differentials} 
\label{s:rooted}

We now give a manifold cover (of finite degree) of $\calC$ (essentially following~\cite[Section~6]{Vee82}). 

\begin{definition}
	Suppose that $q \in \calC$ is a differential.  
	Suppose that $z$ lies in $Z(q)$. 
	Suppose that $v$ is a unit tangent vector, at $z$, pointing along some leaf of the horizontal foliation.  
	We call the pair $(q, v)$ a \emph{rooted} differential.  
\end{definition}

Recalling the difference between the order of a point and the total angle at a point gives a naive count of $4g - 4 + 2|Z(q)|$ for the number of rootings of $q$.
However, some rootings of $q$ may be equivalent to others when $q$ has a symmetry.

Rooted differentials are intended to reproduce the notion of a \emph{marked translation surface} that is widely used in the literature: 
see~\cite[Section 6.10]{Yoc} and~\cite[Section 3]{Boi1}.
We use $\calC_\root$ to denote the space of rooted differentials.  

\begin{lemma}
	\label{Lem:Rooted}
	Suppose that $\calC$ is a stratum component with non-empty singularity data.
	Then $\calC_\root$ is a manifold. 
	Furthermore, the map $\calC_\root \to \calC$ forgetting the root is an orbifold covering map of finite degree. \qed
\end{lemma}


The manifold $\calC_\root$ may not be connected.  
This happens, for example, when $q$ has singularities of different orders.
We fix any one component; 
in a slight abuse of notation, from now on we will call this component $\calC_\root$.


\begin{remark}
	One drawback of $\calC_\root$ is that the $\SL(2, \RR)$--action on $\calC$ does not lift. 
	However, the diagonal action does lift, and this is all we will need below. 
	We finally remark that there is a natural action of the universal cover $\cover{\SL}(2, \RR)$ on $\calC_\root$.
\end{remark}

\section{Zippered rectangles}
\label{s:zippered_rectangles}

Here we pass from the flat geometry of rooted differentials to combinatorics with parameters.
We do this, following Veech~\cite{Vee82} (the abelian case) and Boissy--Lanneau~\cite{Boi-Lan} (the quadratic case), using the \emph{zippered rectangle} construction.  
Usually, this is used to investigate the dynamics of the diagonal flow on $\calC_\root$. 
However, we are also interested in the topology of $\calC_\root$; 
so we present the full details of the zippered rectangle construction and draw particular attention to the aspects we will need. 

The \emph{singularity parameters} we use are due to Yoccoz~\cite[Section~4.3]{Yoc} for abelian strata.  
These are closely related to the \emph{zipper parameters} introduced by Veech~\cite[Section~6]{Vee82}.
In this section we define both parameterisations and discuss how to move between them.  

\subsection{Base-arcs}

We begin with a version of Keane's property~\cite[Section~2]{Keane75}.  
Fix a stratum component $\calC$.

\begin{definition}
	\label{d:saddled}
	Suppose that $q$ is a differential in $\calC$. 
	We say that $q$ is \emph{saddled} if it has either a horizontal \emph{or} a vertical saddle connection.
	We say that $q$ is \emph{bi-saddled} if it has both a horizontal \emph{and} a vertical saddle connection. 
	We denote the sets of saddled and bi-saddled differentials in $\calC$ by $\calV$ and $\calW$, respectively. 
\end{definition}

The sets $\calV$, $\calW$, and their complements lift to $\calC_\root$. 
In a small abuse of notation we will not notationally distinguish the lifts from the originals.

\begin{remark}
	\label{r:countable-union}
	The set $\calV$ is a countable union of codimension-one loci,
	and $\calW$ is a countable union of codimension-two loci, in $\calC_\root$. 
\end{remark}

Suppose that $(q, v)$ is a rooted quadratic differential. 
Let $I_v$ be the horizontal separatrix emanating from the root. 
(Note that $I_v$ may be a horizontal saddle connection, and thus compact.)
We orient $I_v$ away from the base-point of $v$.  
(In our diagrams $I_v$ is always oriented to the right.)
A segment of a leaf of the vertical foliation, emanating from a point of $I_v$, is \emph{upwards} or \emph{downwards} as it makes an angle of $\pi/2$ or $-\pi/2$ with $I_v$.

Suppose that $r$ lies in $I_v$. 
Let $I(r)$ be the subarc of $I_v$ from the base of the root to $r$.
We call $r$ the \emph{right endpoint} of $I(r)$.

\begin{definition}
	\label{d:base-arc}
	We say that $I(r)$ is a \emph{base-arc} for $q$ if it satisfies the following. 
	\begin{enumerate}
		\item
		\label{i:base-arc every saddle connection}
		The interior of $I(r)$ meets every vertical saddle connection of $q$. 
		\item 
		\label{i:base-arc right endpoint}
		If the right endpoint $r$ does not lie in $Z(q)$ then the upward or downward ray emanating from $r$ hits a point of $Z(q)$ before hitting $I(r)$ a second time.
		\qedhere
	\end{enumerate} 
\end{definition}


If $I(r)$ is a base-arc for $q$, then~\cite[Corollary~5.5]{Yoc} implies that the interior of $I(r)$ meets every leaf of the vertical foliation.

\begin{definition}
	\label{d:base-arc-endpoints}
	Suppose that $(q, v)$ is a rooted differential.
	We define $E(q, v)$ to be the set of points $r \in I_v$ so that $I(r)$ is a base-arc.
\end{definition}

The proof of the following lemma is analogous to the one given by Yoccoz in the abelian case~\cite[Proposition~5.6]{Yoc}.

\begin{lemma} 
	\label{l:base-arc-infinite}
	Suppose that  $(q, v)$ is a rooted differential. 
	If $(q,v)$ has no vertical saddle connection then $E(q,v)$ accumulates at zero.
	If $(q,v)$ has no horizontal saddle connection then $E(q,v)$ is unbounded.
\end{lemma}

\begin{proof}
	Suppose that $q$ has no vertical saddle connection.  
	As in~\cite[Corollary~5.5]{Yoc} the vertical foliation for $q$ is minimal (as otherwise the closure of a vertical leaf would be a subsurface with boundary, containing vertical saddle connections~\cite[Section~2]{Keane75}).
	So every (non-trivial) initial segment of $I_v$ meets all leaves of the vertical foliation and so satisfies condition \eqref{i:base-arc every saddle connection} in \Cref{d:base-arc}.
	We pass to a subarc to obtain condition \eqref{i:base-arc right endpoint} in \Cref{d:base-arc}. 
	Thus every non-trivial initial segment of $I_v$ contains a base-arc.
	
	Suppose instead that $q$ has no horizontal saddle connection. 
	So the horizontal foliation for $q$ is minimal.  
	Thus every sufficiently long initial segment of $I_v$ meets the interior of all vertical saddle connections and so satisfies condition \eqref{i:base-arc every saddle connection} in \Cref{d:base-arc}. 
	Extending further gives condition \eqref{i:base-arc right endpoint} in \Cref{d:base-arc}. 
	Thus every initial segment of $I_v$ is contained in a base-arc.
\end{proof}

\begin{definition}
	\label{d:breakpoint}
	Suppose that $I(r)$ is a base-arc for $q$.
	A point $p$ in the interior of $I(r)$ is a \emph{top breakpoint} if the upwards ray emanating from $p$ hits a singularity before it hits the interior of $I(r)$. 
	We define the \emph{bottom breakpoints} similarly. 
\end{definition}

Note that, if some $p$ in $I(r)$ is both a top and bottom breakpoint, then $q$ has a vertical saddle connection.

\begin{definition}
	\label{d:interval-exchange}
	The connected components of (the interior of) $I(r)$, minus the top breakpoints, are called the \emph{top intervals}.  
	The set of these is denoted $\It = \It(r)$.  
	We define the bottom intervals, and the set $\Ib = \Ib(r)$, similarly. 
	We define $U \from \bigsqcup \It \to I(r)$ by setting $U(p)$ equal to the first (interior) intersection between the upward ray emanating from $p$ and $I(r)$.
	We define $D \from \bigsqcup \Ib \to I(r)$ similarly. 
	We finally define $T = (U, D)$:
	the \emph{(non-classical) interval exchange transformation} (also called the \emph{linear involution} \cite{Dan-Nog88,Dan-Nog90}) induced by $I(r)$.
	We write $\bdy T_\Rtop$ and $\bdy T_\Rbot$ for the sets of top and bottom breakpoints of $T$, respectively. 
\end{definition}

Note that, if $q$ has enough vertical saddle connections, it can happen that a subinterval of $I(r)$ lies in both $\It$ and $\Ib$.

\begin{remark}
	\label{r:breakpoints-bounded}
	The number of top and bottom breakpoints $|\bdy T_\Rtop| + |\bdy T_\Rbot|$ equals the number of vertical separatrices emanating from the singularities.  
	In the abelian case this is proved in~\cite[Section 3.1]{Yoc}.   
\end{remark}

The next lemma is folklore. 

\begin{lemma}
	\label{l:base-arc-endpoints}
	Suppose that $(q, v)$ be a rooted differential.
	Then, considered as a subset of $I_v$, the set $E(q, v)$ has at most one accumulation point. 
	Furthermore, if there is an accumulation point then it is the infimum of $E(q, v)$ (and, in this case, the infimum is not an element of $E(q, v)$).
\end{lemma}




\begin{proof}
	Suppose, for a contradiction, that there is an increasing sequence $(r_n)_{n}$ in $E(q, v)$ converging to some point $x$ in $I_v$.  
	Thus, none of the $r_n$ are singularities (that is, lie in $Z(q)$). 
	By definition, either the upward or downward ray emanating from $r_n$ hits a singularity before it hits the interior of $I(r_n)$. 
	Breaking symmetry (and passing to a subsequence and reindexing as needed) we may assume that the \emph{upward} ray from $r_n$ hits a singularity (and so does not meet $I(r_n)$ on its interior).
	
	Suppose that there is a subsequence of the indices $(n_k)_k$ so that the points $s_k = r_{n_k}$ have the following property:
	\begin{itemize}
		\item the upward rays from $s_k$ hit a singularity before hitting the interior of $I(x) - I(s_k)$.
	\end{itemize} 
	Then it follows that there is no uniform upper bound on the number of top breakpoints for the interval exchange transformations induced on the base-arcs $I(s_k)$.
	This contradicts \Cref{r:breakpoints-bounded}.  
	After reindexing to eliminate finitely many of the $r_n$ we deduce:
	\begin{itemize}
		\item 
		for all $y$ in $[r_0, x)$ the upward ray from $y$ hits $I(x) - I(y)$ before hitting a singularity.
	\end{itemize}
	Thus the upward flow takes the interval $[r_0, x)$ isometrically, and strictly, inside of itself. 
	This is a contradiction. 
	
	Suppose instead that there is an decreasing sequence $(r_n)_{n}$ in $I_v$ converging to $x$.  
	Breaking symmetry, we may assume that there is a subsequence of indices $(n_k)_k$ so that the points $s_k = r_{n_k}$ have the following property:
	\[
	s_{k + 1} <  U_k(s_{k + 1}) < s_k
	\]
	Here $T_k = (U_k, D_k)$ is the interval exchange transformation induced on $I(s_k)$.  
	We deduce that $U_k(s_{k + 1}) - s_{k + 1}$ converges to zero.
	Let $\alpha_k$ be the upwards segment from $s_{k+1}$ to $U_k(s_{k + 1})$.
	From the discreteness of the saddle spectrum we deduce that the length of $\alpha_k$ tends to infinity with $k$.
	
	If $x$ (the limit of the $s_k$) is the base-point of $v$ then there is nothing to prove. 
	Suppose instead that the interval $I(x)$ is non-trivial.  
	The upwards segment $\alpha_k$ does not intersect $I(x)$.  
	Thus the Hausdorff limit in $q$ of $(\alpha_k)_k$ is a closed subsurface, which we denote by $X \subset S$.  
	We note that $X$ is disjoint from the interior of $I(x)$. 
	Thus $X$ is not all of $S$. 
	So $\bdy X$ is non-empty; 
	we deduce that $\bdy X$ is a union of vertical saddle connections. 
	Since none of these intersect the interior of $I(x)$ we deduce that $I(x)$ is not a base-arc. 
\end{proof}


\subsection{Labelled generalised permutations}
\label{s:combinatorics}

Suppose that $(q, v)$ is a rooted differential.
Recall that $I_v$ is oriented by $v$, giving upwards and downwards rays in the vertical foliation. 
Suppose that $r$ lies in $E(q, v)$. 
Thus $I(r) \subset I_v$ is a base-arc. 
Let $T = (U, D)$ be the (non-classical) interval exchange transformation induced by $I(r)$.
Note that $T$ gives a fixed-point free involution on $\It \sqcup \Ib$ as follows:
\[
T(J) = 
\begin{cases}
	U(J), &\mbox{if $J$ is a top interval} \\
	D(J), &\mbox{if $J$ is a bottom interval}
\end{cases}
\]
We deduce that $|\It| + |\Ib|$ is even.
Set $2d = |\It| + |\Ib|$.

We now code the action of $T$ on the top and bottom intervals. 
Let $\calA$ be a set of $d$ \emph{letters}.  
Let $\ell = |\It|$ and $m = |\Ib|$;
so $2d = \ell + m$.
We index the intervals of $\It$ (and of $\Ib$) according to their order along the base-arc $I = I(r)$. 
So $\It = (J_i)_{i = 1}^\ell$ and $\Ib= (J_i)_{i = \ell + 1}^{\ell + m}$. 
Now let $\pi = \pi_T \from \{ 1, 2, \ldots, 2d\} \to \calA$ be any two-to-one map with the following property: 
for all $a \in \calA$, if $\{i, j\} = \pi^{-1}(a)$ then $T(J_i) = J_j$.
Thus $\pi$ induces a fixed-point free involution $\sigma$ of $\{ 1, 2, \ldots, 2d\}$; 
here $\sigma(i) = j$ if and only if $\pi(i) = \pi(j)$ (and $i \neq j$).

We call $(\pi, \ell, m)$ a \emph{labelled generalised permutation} of $\calA$.
We say that $(\pi, \ell, m)$ \emph{models} the (non-classical) interval exchange transformation $T$.  
We often just write $\pi$, suppressing $\ell$ and $m$ from the notation.
Also, we often omit $T$ and simply say that $\pi$ \emph{models} the data $(q, v, I(r))$.

Generalised permutations were first considered by Danthony--Nogueira~\cites[Definition on page 471]{Dan-Nog88}[Definition on page 409]{Dan-Nog90} and then by Boissy--Lannaeu~\cite[Definition~2.4]{Boi-Lan}. 
We adopt the notation and language of the latter.

\begin{definition}
	\label{d:equivalent}
	Suppose that $(\pi, \ell, m)$ and $(\pi', \ell', m')$ are generalised permutations on $\calA$.  
	We say they are \emph{equivalent} if $m = m'$, if $\ell = \ell'$, and if there is a reindexing $s \in \Sym(\calA)$ so that $\pi' = s \circ \pi$.
	We call the equivalence class $[\pi]$ an \emph{unlabelled} generalised permutation.
\end{definition}

The equivalence classes of \Cref{d:equivalent} are sometimes referred to as \emph{reduced permutations} (for example by Boissy~\cite{Boi1} and~\cite{Boi2}). 

For labelled generalised permutations, Boissy--Lanneau~\cite[Definition~3.1]{Boi-Lan} introduce the notion of \emph{combinatorially irreducible}.
We do not reproduce the definition here.  
Instead we note one of their main results~\cite[Theorem 3.2]{Boi-Lan} which is crucial in our work. 

\begin{theorem}
	\label{t:irreducible}
	A labelled generalised permutation is combinatorially irreducible if and only if it models a rooted differential with a choice of base-arc. \qed
\end{theorem}

Let $\calR_\root$ be the set of unlabelled generalised permutations arising from rooted differentials (equipped with base-arcs) in $\calC_\root$.  
We call $\calR_\root$ the \emph{Rauzy class} of $\calC_\root$.
(This is sometimes called the \emph{reduced} Rauzy class.)
We call the set $\calR_{\lab}$  
of labelled permutations the \emph{labelled Rauzy class}. 
By \Cref{t:irreducible}, as we vary over all stratum components and all choices of rootings, all labelled irreducible generalised permutations arise from the above construction~\cite[Theorem~3.2]{Boi-Lan}.

Suppose that $(\pi, \ell, m)$ is a labelled generalised permutation of $\calA$.
The letters $\pi(1), \ldots, \pi(\ell)$ are the \emph{top letters} for $\pi$.
Similarly, the letters $\pi(\ell+1), \ldots, \pi(\ell + m)$ are the \emph{bottom letters}.  
Any letter that is both a top letter and a bottom letter is called a \emph{translation letter}.  
Any letter that is only a top letter (or only a bottom letter) is called a \emph{flip letter}. 
We explain the terminology below. 

We say that $\pi$ is a \emph{abelian permutation} if it has no flip letters.  
We say that $\pi$ is a \emph{quadratic permutation} it has (at least one) top flip letter and (at least one) bottom flip letter.
All generalised permutations that arising in \Cref{t:irreducible} are of one of these two types.

From now on, we will not use the terminology ``generalised permutation''.
Instead we collectively refer to abelian and quadratic permutations simply as \emph{permutations}. 

\subsection{The rectangles}
\label{s:rectangles}

Suppose that $(q, v)$ is a rooted differential. 
Suppose that $I(r)$ is a base-arc for $q$. 
Suppose that the labelled permutation $\pi$ models the data $(q, v, I(r))$.

Suppose, as above, that $\It = (J_i)_{i = 1}^\ell$ are the top intervals and $\Ib = (J_i)_{i = \ell + 1}^{\ell + m}$ are the bottom intervals. 
Fix $i \leq 2d$; 
let $j = \sigma(i)$ and let $\alpha = \pi(i) = \pi(j)$. 
If $i \leq \ell$ then $T(J_i) = U(J_i) = J_j$: that is, the upwards flow restricted to $J_i$ has first return (to $I$) equal to $J_j$.
We deduce that the return time is constant on $J_i$.
On the other hand, if $i > \ell$ then the points of $J_i$ flow downwards to land simultaneously in $J_j$.
In a slight abuse of notation we call the closure of the resulting vertical strip a \emph{rectangle}; 
we denote it by $R_\alpha$. 
As an additional abuse of notation we use $\bdy R_\alpha$ to denote $J_i \cup J_j$ together with the correctly chosen vertical segments connecting the endpoints of $J_i$ to those of $J_j$.
While $R_\alpha$ may not be an embedded closed rectangle, the set $R_\alpha - \bdy R_\alpha$ is an embedded open rectangle.  

\begin{definition}
	\label{Def:Cardinal}
	We lay out the base-arc $I(r)$ in the plane, placing the root at the origin and $I(r)$ itself on the positive real axis. 
	Suppose that $\pi(i) = \pi(j) = \alpha$ with $i \neq j$.
	There are two resulting layouts of $R_\alpha$ which we denote by $R_i$ and $R_j$. 
	The sides of these layouts receive \emph{cardinal directions}: south, east, north and west. 
	If $\alpha$ is a translation letter then the translation taking $R_i$ to $R_j$ preserves the cardinal directions.  
	If $\alpha$ is a flip letter then the resulting half-translation exchanges south and north as well as east and west. 
\end{definition}

\begin{corollary}
	\label{c:AtMostOne}
	The closure of the west side of any rectangle $R_i$ contains at most one singularity. 
	The same holds for the east side of any $R_i$. \qed
\end{corollary}


The vertical sides of rectangles are glued across certain sub-arcs called \emph{zippers}, which we now define.

\begin{figure}
	\centering
	\includegraphics{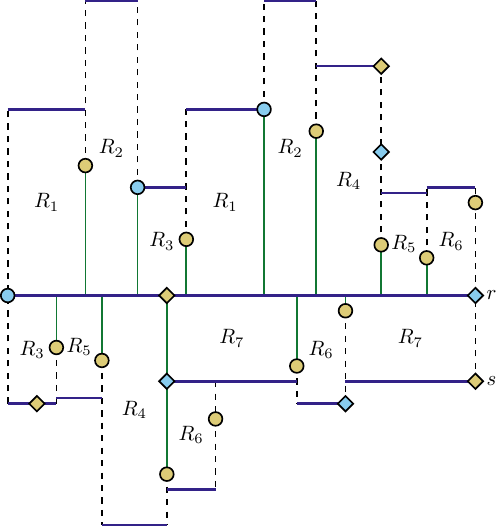}
	\caption{A zippered rectangle construction belonging to the permutation $\big( \begin{smallmatrix} 1 & 2 & 3 & 1 & 2 & 4 & 5 & 6 \\ & 3 & 5 & 4 & 7 & 6 & 7 \end{smallmatrix} \big)$.  
		The singularities are a regular point (blue dot) and a zero of cone angle $10\pi$ (yellow dot).
		Note that every zipper appears as a solid (green) arc exactly once.  
		Also, all rectangles appear twice \emph{except} $R_6$ which appears three times.  
		This is because the zipper through the right endpoint of $I$ is taller than $R_7$.}
	\label{f:zippered-rectangles-no-polygon}
\end{figure}

\subsection{The zippers} 

Suppose that $(q, v)$ is a rooted differential. 
Suppose that $I(r)$ is a base-arc for $q$. 
Let $T = (U, D)$ be the (non-classical) interval exchange transformation induced on $I(r)$.
Let $\pi$ be a labelled permutation that models $T$.

Suppose that $p$ is an top breakpoint for $T$; 
that is, the upward ray emanating from $p$ hits a singularity before it returns to the interior of $I$.
We denote this upward segment by $\zip_{\Rtop}(p)$ and we call it a \emph{top zipper}. 
By construction $\zip_{\Rtop}(p)$ meets the interior of $I$ only in the point $p$.  
Also by construction the boundary of $\zip_{\Rtop}(p)$ is $p$ itself and some singularity $z_{\Rtop}(p)$ in $Z(q)$.
We make similar definitions when $p$ lies in $\bdy T_\Rbot$ to obtain $\zip_{\Rbot}(p)$ and $z_{\Rbot}(p)$.

\begin{lemma}
	\label{l:zipper-neighbourhoods}
	Suppose that $p$ is a (top or bottom) breakpoint. 
	Suppose that $\zip(p)$ is the zipper for $p$. 
	Let $z(p)$ be the other endpoint of $\zip(p)$.
	\begin{itemize}
		\item 
		Suppose that the right endpoint of $I(r)$ does not lie in the interior of $\zip(p)$.
		Then the interior of $\zip(p)$ meets two sides of distinct rectangles $R_i$ and $R_j$ (of the layout) west and east of $\zip(p)$, respectively.
		Note that $\pi(i) = \pi(j)$ if and only if $z(p)$ is a pole.
		\item 
		Suppose that the right endpoint $r$ lies in the interior of $\zip(p)$.
		Then the interior of $\zip(p)$ meets (in $q$) three sides of three rectangles.
		Two of these rectangles are in the layout and are to the west of $\zip(p)$; 
		the third is not in the layout.
		In this case $r$ is not a singularity and $\zip(p)$ is the unique zipper that contains $r$.
	\end{itemize}
\end{lemma}

\begin{proof}
	Suppose that the right endpoint $r$ does not lie in the interior of $\zip(p)$.
	Then the rectangle heights (to the left and right of $\zip(p)$) are greater than or equal to the zipper height. 
	If the rectangle on the left has the same label as the rectangle on the right then $\zip(p)$ lies on the unique vertical separatrix through $z(p)$; 
	hence $z(p)$ is a pole. 
	
	Suppose that the right endpoint $r$ lies in the interior of $\zip(p)$.
	It follows from the definition of zippers that $r$ is not a singularity.
	Let $R$ and $R'$ be the rectangles meeting $p$.
	Let $R$ be the rectangle with height equal to the distance (along the zipper) between $p$ and $r$.
	In this case $R'$ has height greater than or equal to the zipper height, and there is a rectangle $R''$, stacked on top of $R$ (see $R_4$, $R_6$ and $R_7$ below the base-arc in \Cref{f:zippered-rectangles-no-polygon}), so that the height of $R$ plus the height of $R''$ is greater than or equal to the zipper height.
	Conversely, if the interior of a zipper meets three rectangles, consider the side where it meets two rectangles.
	Then the horizontal boundary shared by the two rectangles is a sub-arc of the base-arc, hence it must meet the zipper in the point $r$.
\end{proof}



\begin{corollary}
	\label{c:AtLeastOne}
	Suppose that $\pi(i) = \pi(j) = \alpha$ with $i \neq j$.
	Then at least one of the rectangles $R_i$ and $R_j$ (or perhaps both) has a singularity in the closure of its west side. 
	\qed
\end{corollary}


\subsection{Singularity parameters}
\label{s:singularity_parameters}

Suppose that $(q, v)$ is a rooted differential;
suppose that $I(r)$ is a base-arc for $(q, v)$. 
Suppose that $\pi$ models the data $(q, v, I(r))$.
Let $\calA$ denote the finite alphabet for $\pi$.

We now define the \emph{singularity parameters} $(x_q, y_q) \in \RR^\calA \cross \RR^\calA$ from the given zippered rectangle decomposition of $q$.
At the same time we define certain arcs $\gamma_i$, typically contained in the interior of their rectangle $R_i$.

Fix a letter $\alpha \in \calA$.
Suppose that $\pi(i) = \pi(j) = \alpha$ with $i \neq j$.
Applying \Cref{c:AtLeastOne}, and swapping $i$ and $j$ if needed, we may suppose that $R_i$ has a singularity $z_{\mathrm{w}}$ in the closure of its west side. 
Let $K_i$ be the eastward horizontal spanning arc of $R_i$, emanating from $z_{\mathrm{w}}$.
So $K_i$ ends in the east side of $R_i$. 

\begin{definition}
	\label{d:widths}
	We define the \emph{width} of $\alpha$ to be
	\[
	x_\alpha = |K_i| \qedhere
	\]
\end{definition}
\noindent
We deduce that $x_\alpha$ is the width of $R_i$ and thus the width of $R_\alpha$.
For future use, we note that the widths of rectangles are positive. 
That is,
\begin{equation}
	\label{e:widths}
	x_\alpha > 0
\end{equation}

Let $p_i$ be the endpoint of $K_i$ lying in the east side of $R_i$.
\begin{itemize}
	\item 
	If there is a singularity in the closure of the east side of $R_i$ (by \Cref{c:AtMostOne} there is at most one) then we denote it by $z_{\mathrm{e}}$.
	See \Cref{f:singularity left,f:singularity right}.
	\item
	Suppose instead that there is no singularity in the closure of the east side of $R_i$.
	In this case we extend the east side of $R_i$ away from $J_i \subset I(r)$ until it meets a singularity, which we denote $z_{\mathrm{e}}$. 
	(Note that this happens before the extension meets the interior of $I(r)$, since $I(r)$ is a base-arc.) 
	See \Cref{f:singularity no saddle connection}.
\end{itemize}
In either case we define $L_i$ to be the oriented vertical segment between $p_i$ and $z_{\mathrm{e}}$, oriented away from $p_i$. 
We define $\sgn(L_i)$ to be plus one if $L_i$ points north and minus one if $L_i$ points south. 

\begin{definition}
	We define the \emph{height} of $\alpha$ to be 
	\[
	y_\alpha = \sgn(L_i) \cdot |L_i|
	\]
	So $y_\alpha \in \RR$ is the signed height of $z_\mathrm{e}$ relative to $z_\mathrm{w}$.
\end{definition}

With $R_i$ and $R_j$ as above, we now consider the possibility that $R_j$ is the (unique) rectangle which does \emph{not} have a singularity in the closure of its west side.  
In this case $\alpha$ is a flip letter and $R_i$ was a rightmost rectangle in our layout.  
Let $f_\alpha \from \CC \to \CC$ be the half-translation taking $R_i$ to $R_j$.
We define $K_j$ and $L_j$ to be $f_\alpha(K_i)$ and $f_\alpha(L_i)$, respectively, reversing orientation in both cases.

\begin{definition}
	The pair $(x_\alpha, y_\alpha)$ is the \emph{singularity parameter} of the letter $\alpha \in \calA$.
\end{definition}

\noindent
Note that the singularity parameter is well-defined because the two rectangles $R_i$ and $R_j$ (and their various singularities) differ by a (half-)translation.

\begin{definition}
	\label{d:rectilinear}
	We define the oriented arc $\gamma_i = K_i \cup L_i$ to be the \emph{rectilinear arc} associated with the rectangle $R_i$.
\end{definition}

\begin{figure}
	\centering
	\begin{subfigure}[b]{0.31\textwidth}
		\centering
		\includegraphics[width=\textwidth]{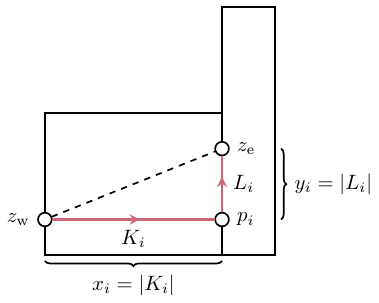}
		\caption{The curve $\gamma_i$ turns left at $p_i$ and can be tightened to a saddle connection.}
		\label{f:singularity left}
	\end{subfigure}
	\hfill
	\begin{subfigure}[b]{0.31\textwidth}
		\centering
		\includegraphics[width=\textwidth]{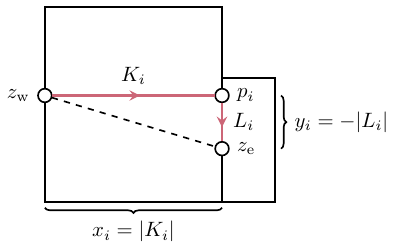}
		\caption{The curve $\gamma_i$ turns right at $p_i$ and can be tightened to a saddle connection.}
		\label{f:singularity right}
	\end{subfigure}
	\hfill
	\begin{subfigure}[b]{0.31\textwidth}
		\centering
		\includegraphics[width=\textwidth]{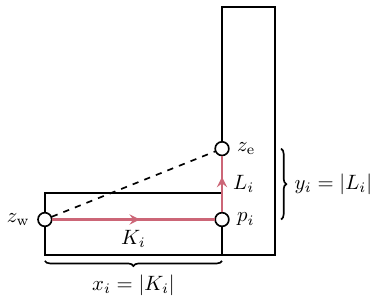}
		\caption{The curve $\gamma_i$ may or may not tighten to a saddle connection.}
		\label{f:singularity no saddle connection}
	\end{subfigure}
	\caption{Three cases of singularity parameters.}
	\label{f:singularity coordinates}
\end{figure}

\begin{remark}
	Suppose that $\gamma_i$ is contained in the closure of $R_i$.
	Then $\gamma_j$ is contained in the closure of $R_j$.
	In this case $\gamma_i$ is, up to a translation, 
	identical to $\gamma_j$ (contained in $R_j$).
	It follows that, again in this case, $\gamma_i$ can be straightened to give a saddle connection in $R_i$.
	See \Cref{f:singularity left,f:singularity right}.
	The parameter $x_\alpha + \mathrm{i} y_\alpha$ is then the period of $\gamma_i$.
	
	In general there is at most one letter (in both the abelian and quadratic setting) where $\gamma_i$ (and thus $\gamma_j$) may not be contained in the closure of $R_i$.
	In this case $\gamma_i$ may or may not be homotopic (relative to its endpoints) to a saddle connection.
	See \Cref{f:singularity no saddle connection} for the general situation and \Cref{f:zippered-rectangles-no-polygon} for a particular example.
	This explains the delicacy of the definitions of the rectilinear arcs $\gamma_i$ and of the parameters $y_\alpha$.
\end{remark}

\begin{remark}
	\label{r:homotopic}
	We further deduce that $\gamma_i$ is homotopic, relative to $Z(q)$, to $\gamma_j$ when $\alpha$ is a translation letter and is homotopic to the orientation-reverse of $\gamma_j$ when $\alpha$ is a flip letter.
\end{remark}

\begin{definition}
	Set $x_q = (x_\alpha)_{\alpha \in \calA}$ and $y_q = (y_\alpha)_{\alpha \in \calA}$.
	We call $(x_q, y_q)$ the \emph{singularity parameters} induced by the data $(q, v, I(r))$.
\end{definition}

\subsection{Width and height equalities}
\label{s:widths_and_heights}

We now form two larger rectilinear arcs in the plane: 
\[
\gamma_\Rtop = \bigcup_{i = 0}^{\ell} \gamma_i 
\qquad\qquad
\gamma_\Rbot = \bigcup_{i = \ell+1}^{\ell + m} \gamma_i 
\]
Both start at the root and end at the eastmost singularity (on the same vertical leaf as $r$).  
Since they have the same beginning and end, we deduce that their periods in $q$ are equal.
We call this the \emph{period equality}.
Taking real and imaginary parts, we deduce the following: 
\[
\sum_{i \leq \ell} x_{\pi(i)} = \sum_{i > \ell} x_{\pi(i)},
\qquad\qquad
\sum_{i \leq \ell} y_{\pi(i)} = \sum_{i > \ell} y_{\pi(i)}
\]
Since translation letters contribute to both sides, the period equality gives the \emph{width} and \emph{height equalities}:
\begin{equation}
	\label{e:e.x-y}
	\sum_{\alpha} x_{\alpha} = \sum_{\beta} x_\beta,
	\qquad\qquad
	\sum_{\alpha} y_{\alpha} = \sum_{\beta} y_\beta
\end{equation}
where $\alpha$ ranges over the top flip letters and $\beta$ ranges over the bottom flip letters.

\subsection{Masur polygon}

Suppose that a labelled permutation $\pi$ models the data $(q, v, I(r))$. 
We use the notion of rectilinear arcs as given in \Cref{d:rectilinear}.
Suppose that we are in the special situation where every rectilinear arc $\gamma_i$, when transported to $q$, tightens to give a saddle connection. 
In this case we may do the following:
\begin{itemize}
	\item 
	straighten all of the $\gamma_i$ (in $\CC$) to obtain line segments $\gamma^*_i$,
	\item 
	cut $\CC$ along the $\gamma^*_i$, and 
	\item 
	obtain a polygon $M = M(q, v, I(r))$. 
\end{itemize}
We call $M$ the \emph{Masur polygon} for $q$, induced by $I(r)$.  
When $q$ is abelian, the periods of the saddle connections $\gamma^*_i$ give the \emph{period coordinates} for $q$.
When $q$ is quadratic, the choice of root $v$ determines the signs of the periods of the $\gamma^*_i$.
In the quadratic case these periods satisfy the single complex relation given by \Cref{e:e.x-y}.

When some $\gamma_i$ cannot be straightened, the Masur polygon $M$, as induced by $I(r)$, does not exist.
For a discussion and further examples see \Cref{a:Masur}.

\subsection{Zipper parameters}
\label{s:zippers}

Again with notation as above, we break symmetry and pick $p \in \bdy T_\Rtop$.
We assume that $p \neq r \times \{\Rtop\}$ if $r \in \bdy T_\Rtop$.
Let $\zip(p)$ be a top zipper based at $p$.  
By a slight abuse of notation, think of $p$ as a point in $I$.
Let $R_{\pi(i)}$ for $i \leq \ell$ be the rectangle to the left of $\zip(p)$. 
Then the horizontal coordinate of $p$ (that is, the distance of $p$ from the left-endpoint of $I$) is given by
\[
x(p)= \sum_{j = 1}^i x_{\pi(j)}.
\]
The height of $\zip(p)$ is given by 
\begin{equation}
	\label{e:height_top_zipper}
	h(\zip(p)) = \sum_{j=1}^i y_{\pi(j)}.
\end{equation}
and we require this to be positive. 
This gives us the \emph{top zipper inequalities}
\begin{equation}
	\label{e:top-zippers} 
	\sum_{j=1}^i y_{\pi(j)} > 0
\end{equation}
for all $i < \ell$.

Similarly, if $\zip(p)$ for $p \in \bdy T_\Rbot \times \{\Rbot\}$ and $p \neq r \times \{\Rbot\}$ if $ r\in \bdy T_\Rbot$ is a bottom zipper and $R_{\pi(i)}$ for $i \geq \ell+1$ then the horizontal coordinate is
\[
x(p)= \sum_{j = \ell+1}^i x_{\pi(j)} 
\]
and the height is
\begin{equation}
	\label{e:height_bottom_zipper}
	h(\zip(p)) = \sum_{j = \ell+1}^i y_{\pi(j)}.
\end{equation}
Since it is a bottom zipper, we require that $h(\zip(p))$ is negative.  
This gives us the \emph{bottom zipper inequalities}
\begin{equation}
	\label{e:bottom-zippers}
	\sum_{j = \ell+1}^i y_{\pi(j)} < 0
\end{equation}
for all $\ell+1 \leq i < \ell+m$.

It remains to consider the right endpoint $r$. The zipper height of $\zip(r)$ gives us a linear relation in the $y$ parameters.
Note that the above equalities express the height $h(\zip(r))$ in two ways; namely 
\[
h(\zip(r)) = \sum_{j=1}^\ell y_{\pi(j)}
\]
and
\[
h(\zip(r)) = \sum_{j = \ell+1}^{\ell+m} y_{\pi(j)}.
\]
We thus recover the \emph{height} equality
\begin{equation}
	\sum_{k = 1}^\ell y_{\pi(k)} = \sum_{k = \ell + 1}^{\ell + m} y_{\pi(k)}
\end{equation}
which is equivalent to $\sum y_\alpha = \sum y_\beta$, 
where $\alpha$ ranges over the top flip letters and $\beta$ ranges over the bottom flip letters.

The height and width equalities are identical in form. 
Thus, the dimensions of the space of $x$ and $y$ parameters are equal; 
they are $|\calA|$ in the abelian case and $|\calA| - 1$ in the quadratic case.

\subsection{Rectangle parameters} 

For all rectangles $R = R_\alpha$, at least one of the points $z_{\mathrm{e}}$ and $z_{\mathrm{w}}$ lie in its east and west sides, respectively. 
Breaking symmetry, suppose that $\alpha$ is a top letter and $z_{\mathrm{e}}$ lies in its east side.
Let $\zip(p)$ for $p \in \bdy T_\Rtop \times \{\Rtop\}$ be the zipper with end point $z_{\mathrm{e}}$.
If $\alpha$ is a translation letter then there is a zipper $\zip(p')$ for $p' \in \bdy T_\Rbot \times \{\Rbot\}$ with endpoint $z_{\mathrm{e}}$ such that the union $\zip(p) \cup \zip(p')$ is the east side of $R_\alpha$. 
Recall that the heights of bottom zippers are negative. 
Hence, the height $h(R_\alpha)$ satisfies
\begin{equation}
	\label{e:rectangle height translation letter}
	h(R_\alpha) = h(\zip(p)) - h(\zip(p')).
\end{equation}

If $\alpha$ is a flip letter instead then there is a zipper $\zip(p')$ for $p' \in \bdy T_\Rtop \times \{\Rtop\}$ with end point $z_{\mathrm{e}}$ such that the union $\zip(p) \cup \zip(p')$ is the east side of $R_\alpha$. 
The height $h(R_\alpha)$ is then 
\begin{equation}\label{e:rectangle height flip letter}
	h(R_\alpha) = h(\zip(p)) + h(\zip(p')).
\end{equation}

A similar discussion follows if $\alpha$ is a bottom letter.

\subsection{Polytopes of differentials and polytopes of parameters}
\label{s:polytopes}

\begin{definition}
	\label{d:distinguished}
	Suppose that $(q, v)$ is a rooted differential.
	Suppose that 
	\begin{itemize}
		\item
		$1 \notin E(q, v)$ and
		\item
		there exist $s, t \in E(q, v)$ with $s < 1 < t$.
	\end{itemize}
	In this case \Cref{l:base-arc-endpoints} tells us that $r = \min (E(q, v) \cap (1, \infty))$ is well-defined.
	We call $I(r)$ the \emph{distinguished base-arc} for $(q, v)$.
\end{definition}

The choice of $1$ in the above definition is simply a choice, and can be replaced by any other positive real number. 

\begin{definition}
	\label{d:polytope_of_differentials}
	Suppose that $\pi$ is a labelled permutation.  
	We define $\calC(\pi)$, the \emph{polytope of differentials} modelled by $\pi$, as follows. 
	Suppose that $(q, v)$ is a rooted differential. 
	Suppose that $(q, v)$ has a distinguished base-arc $I(r)$. 
	Suppose that $(q, v, I(r))$ is modelled by $\pi$. 
	Then we place $(q, v)$ into the set $\calC(\pi)$.
\end{definition}

We deduce the following. 

\begin{lemma}
	\label{l:equivalent-implies-identicalpolytopes}
	Suppose that $\pi$ and $\pi'$ are labelled permutations.
	If $\pi$ and $\pi'$ are equivalent then $\calC(\pi) = \calC(\pi')$. 
	If $\pi$ and $\pi'$ are not equivalent then $\calC(\pi)$ and $\calC(\pi')$ are disjoint. \qed
\end{lemma}

\begin{lemma}
	\label{l:polytopes-open}
	Suppose that $\pi$ is a labelled permutation.
	Then the polytope $\calC(\pi)$ is open.
\end{lemma}

\begin{proof}
	Fix $(q, v)$ in $\calC(\pi)$. 
	Suppose that $I(r)$ is the given distinguished base-arc for $(q, v)$.
	We use $I(r)$ to give a layout of $q$ in the plane, made of the rectangles found in \Cref{s:rectangles}.
	Let $I(0, 1)$ be the open initial segment of $I(r)$ of length one.
	Let $z_{\mathrm{e}}$ be the easternmost singularity in this layout.
	Let $\zip(p)$ and $\zip(p')$ be the zippers contained in the northern and southern separatrices emanating from $z_{\mathrm{e}}$ (as in the layout). 
	By \Cref{l:zipper-neighbourhoods} there are rectangles $R$ and $R'$ \emph{not} in the layout whose sides contain $\zip(p)$ and $\zip(p')$ respectively.
	
	Note that the breakpoints (as in \Cref{d:breakpoint}) lie in $I(0, 1)$. 
	The top breakpoints are all distinct as are the bottom breakpoints. 
	Note also that the zippers have positive heights.
	We deduce, in a sufficiently small neighbourhood of $(q, v)$, that
	\begin{itemize}
		\item no pair top breakpoints collide,
		\item no pair of bottom breakpoints collide, and
		\item no singularity (except its westernmost endpoint) meets $I(0, 1)$.
	\end{itemize}
	In particular, the widths of the rectangles (including $R$ and $R'$) are bounded away from zero. 
	We deduce that all differentials near $(q, v)$ also have distinguished base-arcs and that these also induce the labelled permutation $\pi$.
\end{proof}

\begin{lemma}
	\label{l:polytopes-dense}
	The union of the polytopes $\calC(\pi)$ is dense in $\calC_\root$.
\end{lemma}

\begin{proof}
	By \Cref{r:countable-union} the set $\calV$ (of saddled differentials) is a countable union of codimension-one subsets.
	Let $\calO$ be the subset of $\calC_\root$ of differentials $(q, v)$ where $1$ lies in $E(q, v)$.  
	So $\calO$ is also a countable union of codimension-one subsets.
	Thus $\calC_\root - (\calV \cup \calO)$ is dense.
	
	Suppose that $(q,v)$ is a rooted differential in $\calC_\root - (\calV \cup \calO)$.
	By \Cref{l:base-arc-infinite}, the $E(q,v)$ accumulates at zero and is also unbounded. 
	Thus $E(q,v)$ contains values both smaller and larger than one, but does not contain one. 
	Thus $(q,v)$ has a distinguished base-arc and so lies in some $\calC(\pi)$, as desired. 
\end{proof}

\begin{remark}
	In \Cref{a:measures}, we recall the construction of the Masur--Smillie--Veech measure on $\calC_\root$. The subset $\calV \cup \calO$ has measure zero, and hence the union of polytopes has full measure.
\end{remark} 

We now discuss the \emph{polytope of parameters}.

\begin{definition} 
	\label{d:polytope_of_parameters}
	Suppose that $\pi$ is a labelled permutation. 
	The \emph{polytope of parameters} $P(\pi)$ is the set of pairs $(x,y) \in \RR^\calA \times \RR^\calA$ satisfying the following:
	\begin{enumerate}
		\item 
		the positivity inequalities \eqref{e:widths},
		\item 
		the width and height equalities \eqref{e:e.x-y},
		\item 
		the zipper inequalities \eqref{e:top-zippers} and \eqref{e:bottom-zippers}, and
		\item 
		\label{i:distinguished}
		the \emph{distinguished base-arc inequalities}
		\begin{equation}
			\label{e:distinguished}
			1 < \sum_{i = 1}^{\ell} x_{\pi(i)} < 1 + \min\{ x_{\pi(\ell)}, x_{\pi(\ell + m)} \}
		\end{equation}
	\end{enumerate}
	We say that parameters $(x, y) \in P(\pi)$ are \emph{admissible} for $\pi$.
\end{definition}

\begin{lemma}
	\label{l:distinguished_implies_admissible}
	Suppose that $\pi$ is a labelled permutation. 
	Suppose that $(q, v)$ lies in $\calC(\pi)$. 
	Suppose that $I(r)$ is the distinguished base-arc for $(q, v)$.
	Let $(x_q, y_q)$ be the singularity parameters induced by the data $(q, v, I(r))$.
	Then $(x_q, y_q)$ lies in $P(\pi)$.
\end{lemma}

\begin{proof}
	Since $\pi$ models the data $(q, v, I(r)$, the parameters $(x,y)$ satisfy 
	the positivity inequalities \eqref{e:widths} (shown in \Cref{s:singularity_parameters}), 
	the width and height equalities \eqref{e:e.x-y} (shown in \Cref{s:widths_and_heights}), and 
	the zipper inequalities \eqref{e:top-zippers} and \eqref{e:bottom-zippers} (shown in \Cref{s:zippers}). 
	
	Define $r'$ to be the maximum of $E(q, v) \cap (0, 1)$.  
	Note that $r'$ is well-defined because $(q, v)$ has a distinguished base-arc and by \Cref{l:base-arc-endpoints}.
	Recall that $r$ is the minimum of $E(q,v) \cap (1, \infty)$.
	Thus $r' < 1 < r$. 
	Suppose that $\pi$ is of the form $(\pi, \ell, m)$; 
	that is, there are $\ell$ letters on the top and $m$ letters on the bottom. 
	Since $r'$ and $r$ are consecutive points in $E(q, v)$, we have that
	\[
	r = r' + \min \{ x_{\pi(\ell)}, x_{\pi(\ell + m)} \} < 1 + \min \{ x_{\pi(\ell)}, x_{\pi(\ell + m)}\}
	\]
	as required.
\end{proof}

\begin{definition}
	\label{d:differentials-to-parameters}
	Suppose that $\pi$ is a labelled permutation. 
	Suppose that $(q, v)$ lies in $\calC(\pi)$. 
	Let $I(r)$ be the resulting distinguished base-arc.
	We define $p_\pi \from \calC(\pi) \to P(\pi)$ by setting $p_\pi(q, v) = (x_q, y_q)$.
\end{definition}

We now give a function in the opposite direction. 

\begin{definition}
	\label{d:parameters-to-differentials}
	Suppose that $\pi$ is a labelled permutation. 
	Suppose that $(x, y)$ lies in $P(\pi)$ (so is admissible for $\pi$).
	We now layout the resulting rectangles along the non-negative real axis in $\CC$ (as in \Cref{s:rectangles}). 
	We glue along the zippers to obtain the differential $q_\pi(x, y)$. 
\end{definition}

It is a delicate result of Boissy--Lanneau~\cite[Lemma 2.12]{Boi-Lan} that the function $q_\pi$ is well-defined.
Since the parameters are admissible, $q_\pi(x, y)$ lies in $\calC(\pi)$. 


From the definitions, the functions $q_\pi$ and $p_\pi$ are inverses of each other.
Thus both are bijections.

\begin{lemma}
	\label{l:homeomorphsims}
	Suppose that $\pi$ is a labelled permutation.
	Then the maps $q_\pi$ and $p_\pi$ are continuous. 
	Thus both are homeomorphisms.
\end{lemma}

\begin{proof}
	Fix parameters $(x, y)$ in $P(\pi)$.
	A small change $(x, y)$, remaining in $P(\pi)$, gives a small motion of the associated rectangles and zippers; 
	this leads to a small motion of the polygons, and thus by \Cref{d:stratum-component} remains in a small open set in $\calC(\pi)$.
	Thus $q_\pi$ is continuous.
	
	Fix instead a differential $(q, v)$ in $\calC(\pi)$. 
	Since $\calC(\pi)$ is open (\Cref{l:polytopes-open}) a sufficiently small deformation of $(q, v)$ remains in $\calC(\pi)$.
	Thus we may ignore the quotient (by scissors congruence) appearing in \Cref{d:stratum-component}. 
	Thus a sufficiently small deformation of $(q, v)$ only moves the vertices (of the polygons) a small amount.  
	Since we remain in $\calC(\pi)$ the shapes of the rectangles and zippers change continuously; 
	thus the parameters change continuously.
	So $p_\pi$ is continuous, as desired.
\end{proof}

\subsection{Based loops in \texorpdfstring{$\calC_\root$}{C\textasciicircum{}root}}

We now state our key tool for relating ``topology'' to ``dynamics''.  
Roughly, every (based) loop in $\calC_\root$ can (almost) be homotoped to be a concatenation of diagonal flow segments.
In the statement and proof we suppress the notation $v$ for the root.

\begin{theorem}
	\label{t:based-loops}
	Suppose that $\calC_\root$ is a stratum component of rooted quadratic differentials.
	Suppose that $q_0$ is a base-point in $\calC_\root$. 
	Suppose that $\gamma \from [0, 1] \to \calC_\root$ is a loop based at $q_0$.
	Then, up to a homotopy relative to the base-point, the loop $\gamma$ can be written as a finite concatenation of paths that are either
	\begin{itemize}
		\item 
		geodesic segments for the diagonal flow or
		\item 
		contained inside some polytope of differentials.
	\end{itemize}
\end{theorem}

\begin{proof}
	Recall that the set $\calW \subseteq \calC_\root$ of bi-saddled rooted differentials is a countable union of codimension-two loci. 
	For convenience, we assume that our basepoint $q_0$ is contained in some polytope but does not lie in $\calW$. 
	
	Let $q$ be another differential in $\calC_\root - \calW$. 
	By \Cref{l:base-arc-infinite}, we may choose a base-arc $I$ for $q$. 
	With $I$ chosen, we obtain a (generalised) permutation $\pi$ and singularity parameters $(x, y)$.
	Appealing to~\cite[Theorem~3.2]{Boi-Lan}, the permutation $\pi$ is irreducible (in the sense of~\cite[Definition~3.1]{Boi-Lan}). 
	
	By \Cref{l:base-arc-endpoints}, the induced singularity parameters  give coordinates in a neighbourhood of $q$. 
	So there exists an open set in $\calC_\root$ containing $q$ given by the zippered rectangles construction with underlying permutation $\pi$.
	If $q$ lies in $\calC(\pi)$ then there is an (even smaller) neighbourhood $U(q)$ of $q$ with the following properties:
	\begin{enumerate}
		\item $U(q)$ lies in $\calC(\pi)$ and
		\item $q^{-1}_\pi (U(q))$ is a product neighbourhood in $P(\pi)$.
	\end{enumerate}
	We call $U(q)$ an \emph{admissible box neighbourhood}.
	
	Suppose instead that $q$ does not lie in $\calC(\pi)$. 
	Since $\pi$ is irreducible we deduce that the base-arc $I$ for $q$ is not the distinguished base-arc.  
	That is, we violate either the upper or lower base-arc inequality
	\[
	1 < \sum_{k = 1}^\ell x_{\pi(k)} < 1 + \min\{ x_{\pi(\ell)}, x_{\pi(\ell + m)} \}.
	\]
	We remedy this by applying the diagonal flow. 
	That is, there is some (non-unique) time $t(q) \in \RR$ such that $g_{t(q)} q$ lies in $\calC(\pi)$. 
	
	Making choices, we have some admissible box neighbourhood $U(g_{t(q)} q)$.  
	In a slight abuse of notation, we take $U(q) = g_{-t(q)} U(g_{t(q)} q)$.
	We call $U(q)$ a \emph{pre-admissible box neighbourhood}.  
	Therefore, $\bdy U(q)$ is a union of finitely many codimension-one embedded submanifolds (with boundary) in $\calC_\root$.
	
	The locus $\calV$ of saddled rooted differentials can be covered by countably many relatively open codimension-one charts.
	Similarly, the locus $\calW$ of bi-saddled rooted differentials can be covered by countably many relatively open codimension-two charts.
	Hence, we may apply a homotopy (relative to $q_0$) to arrange that $\gamma$ is transverse to $\calV$ and to $\calW$~\cite[Theorem~2.5, page~78]{Hir}.
	In particular, after this, the loop $\gamma$ is disjoint from $\calW$.  
	The boxes $(U({\gamma(s)}))_s$ cover $\gamma$. 
	By compactness, there exists a finite collection $s_1, \dotsc, s_n \in [0, 1]$ so that $(U({\gamma(s_j)}))_j$ covers $\gamma$. 
	Let $U_j = U(\gamma(s_j))$. 
	We add to the collection $(U_j)$ an admissible box neighbourhood $U_0$ of $\gamma(s_0) = q_0$.
	
	Again appealing to~\cite[Theorem~2.5, page~78]{Hir} we perform a further homotopy (again relative to $q_0$) supported in the union of the boxes; 
	this makes $\gamma$ transverse to the sides of the boxes $U_j$ while keeping $\gamma$ transverse to $\calV$ (and disjoint from $\calW$).
	
	This done, $\gamma$ intersects $\bdy U_j$ only finitely many times.
	Let $J_0 = [0, s) \cup (s', 1]$ be the union of the components of $\gamma^{-1}(U_0)$ containing the boundary of $I$.
	Furthermore, for each $j > 0$, the preimage $\gamma^{-1}(U_j)$ is a finite union of intervals in $[0, 1]$. 
	All such intervals are (relatively) open.
	We now select a minimal subcollection $J_1, \dotsc, J_{m-1}$ of these intervals that covers $[s, s']$; 
	we choose the indexing so that $J_k$ is left of $J_{k+1}$.
	We use $V_k$ to denote the box $U_j$ containing $\gamma(J_k)$.
	(Note that the list $V_0, \dotsc, V_{m-1}$ may contain repetitions.)
	We now treat indices modulo $m$.
	Thus, for all $k$, we have that $J_k \cap J_{k+1}$ is a non-empty (relatively) open interval.
	(There is a special case when $m = 2$.)
	
	Since $\gamma$ is transverse to $\calV$, we have that $\gamma(s)$ lies in $\calV$ for at most countably many $s \in [0, 1]$.
	Thus, there exists a rooted differential $q_{k+1}$, without horizontal or vertical saddle connections, in the image of each set $\gamma(J_k \cap J_{k+1})$. 
	Hence, we obtain a sequence of times $0 = s_0 \leq s_1 \leq \dotsb \leq s_m = 1$  such that the closed intervals $[s_k, s_{k+1}] \subseteq [0, 1]$ cover $[0, 1]$ and $\gamma(s_k) = q_k$.
	
	\begin{figure}
		\centering
		\includegraphics{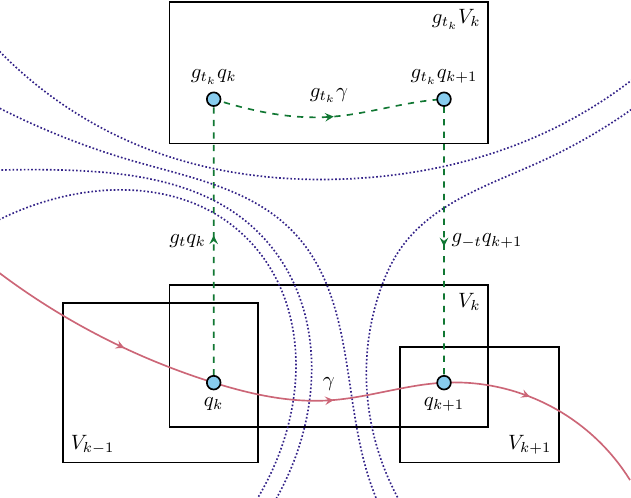}
		\caption{Illustration of the proof of \Cref{t:based-loops}. 
			Part of the loop $\gamma$ is depicted as a solid curve. 
			The dotted lines represent the boundaries of the polytopes. 
			Unlike the boxes $V_{k-1}$ and $V_{k+1}$, the box $V_k$ is not contained inside a polytope, so the diagonal flow must be applied to it. 
			The resulting segment $\delta_k$ is shown as dashed curve.}
		\label{f:based-loops}
	\end{figure}
	
	Since $V_k$ is a pre-admissible box neighbourhood there exists a real number $t_k$ such that $g_{t_k} V_k$ is completely contained inside some polytope in $\calC_\root$. 
	We now define a path $\delta_k$ starting at $q_k$ and ending at $q_{k+1}$.  
	If $t_k \geq 0$ then $\delta_k$ is the concatenation of
	\begin{itemize}
		\item $g_t q_k$ for $t \in [0, t_k]$;
		\item $g_{t_k} \gamma(s)$ for $s \in [s_k, s_{k+1}]$; and
		\item $g_{-t} q_{k+1}$ for $t \in [-t_k, 0]$.
	\end{itemize}
	If $t_k < 0$ then $\delta_k$ is the concatenation of
	\begin{itemize}
		\item $g_{-t} q_k$ for $t \in [0, -t_k]$;
		\item $g_{t_k} \gamma(s)$ for $s \in [s_k, s_{k+1}]$; and
		\item $g_{t} q_{k+1}$ for $t \in [t_k, 0]$.
	\end{itemize} 
	See \Cref{f:based-loops}.
	
	The union of $\delta_k$ and $\gamma_k = \gamma|[s_k, s_{k+1}]$ bounds a disc in $\calC_\root$; 
	this disc is foliated by the arcs $g_t \gamma_k$ where $t \in [0, t_k]$.
	In particular, $\delta_k$ and $\gamma_k$ are homotopic, relative to their common endpoints.
	
	Let $\delta$ be the concatenation of the paths $\delta_k$.
	By construction, $\delta$ is a closed curve homotopic to $\gamma$, fixing $q_0$.  
	Moreover, the pieces $g_{t_k} \gamma(s)$ for $s \in [s_k, s_{k+1}]$ in the concatenation are the only paths that are not (forward or backward) diagonal flow segments. 
	This concludes the proof of the theorem.
\end{proof}

\begin{remark}
	We do not attempt to make optimal choices to reduce the length of geodesic pieces in the concatenation for $\delta$.
	A simple way to reduce these lengths is to choose the admissible zippered rectangle construction whenever $q$ admits one.
	Thus, if $q \in \calC_\root - \calW$ is contained in some polytope, then we can choose the box $U(q)$ to be contained inside the same polytope and set $t(q) = 0$ for such boxes.
	On the other hand, when $q$ does not lie in any polytope we can choose the length of the base-arc to be as close to $1$ as possible, so $t(q)$ is as small as possible.
	See \Cref{s:crossing faces} for a concrete example of the construction.
\end{remark}

\subsection{Rauzy--Veech induction}
\label{s:RV-induction}

We now define Rauzy--Veech induction on zippered rectangles.
Let $I(r)$ be the base-arc of a zippered rectangle decomposition of a rooted differential $(q, v)$.
Suppose that the underlying labelled permutation is $\pi$. 
The induction is defined by passing to the longest base-arc $I(r')$ which is strictly contained in $I(r)$. 
The base-arc $I(r')$ has length
\[
|I(r')| = |I(r)| - \min \{ x_{\pi(\ell)}, x_{\pi(\ell+m)} \}
\]

We now say how the combinatorics and geometry of the decomposition change under a single Rauzy--Veech move.
Suppose that $\pi$ is a labelled irreducible permutation.
Suppose that $(x, y) \in P(\pi)$ are singularity parameters.

Let $\alpha = \pi(\ell)$ and $\beta = \pi(\ell+m)$ be the final letter on the top and bottom, respectively.
Since $\pi$ is irreducible, we deduce that $\alpha \neq \beta$.  
Suppose that $x_\alpha > x_\beta$.
Thus we say that the \emph{top letter wins}. 
The new parameters are 
\begin{align} \label{e:parameter-change}
	\begin{split}
		x'_\alpha &= x_\alpha - x_\beta \\
		y'_\alpha &= y_\alpha - y_\beta
	\end{split}
\end{align}
Also $x'_\rho = x_\rho$ for all $\rho \neq \alpha$
and $y'_\rho = y_\rho$ for all $\rho \neq \beta$.

We encode the parameter transformations as a matrix. 
Let $E = (e_{r s})$ be the $\calA \times \calA$ elementary matrix with ones along the diagonal, with $e_{\alpha \beta} = 1$ and with all other entries zero. 

\begin{convention}
	\label{conv:column-vectors}
	We write horizontal and vertical parameters as \emph{column} vectors. 
	Thus $E x' = x$ and $E y' = y$.
\end{convention}

Observe that our convention is \emph{opposite} to the one used in Yoccoz's lecture notes~\cite[Section~3.4]{Yoc}; 
it allows us to avoid taking various transposes.

To define the new permutation we recall that $\alpha$ is either a translation \emph{or} flip letter.
Suppose $\alpha$ is a translation letter.
Let $\pi(j) = \alpha$ for some $\ell+1 \leq j < \ell + m$. 
We then set 
\begin{itemize}
	\item $\pi'(i) = \pi(i)$ for all $i \leq j$, 
	\item $\pi'(j+1) = \beta$, and
	\item $\pi'(i) = \pi(i-1)$ for all $i > j+1$.
\end{itemize}  

Suppose instead that $\alpha$ is a flip letter and let $\pi(j) = \alpha$ for some $1 \leq j < \ell$. 
In $\pi'$ the top indices to range from $1$ to $\ell + 1$; 
the bottom indices to range from $\ell + 2$ to $\ell + m$.
We then set
\begin{itemize}
	\item $\pi'(i) = \pi(i)$ for all $i < j$, 
	\item $\pi'(j) = \beta$, 
	\item $\pi'(i) = \pi(i-1)$ for all $i > j$.
\end{itemize}  

With the above definitions, we set $R_{\Rtop} (\pi, x, y) = (\pi', x', y')$.
If $x_\alpha < x_\beta$ instead, then we say that the \emph{bottom letter wins}. 
The definition of $R_{\Rbot}$ is very similar.

\begin{remark}
	The induction is undefined on the codimension-one locus $x_\alpha = x_\beta$, which is contained in $\calV$.
\end{remark}

The Rauzy--Veech moves give the labelled Rauzy class $\calR_{\lab}$ the structure of a directed graph.
We denote this graph by $\calD_\lab$. 
The vertices of $\calD_\lab$ are irreducible labelled permutations in $\calR_{\lab}$. 
We have an arrow from a permutation $\pi$ to a permutation $\pi'$ if $\pi' = R_{\Rtop} (\pi)$ or $R_{\Rbot}(\pi)$.
A component of $\calD_\lab$ is called a \emph{labelled Rauzy diagram}.

Suppose $\pi$ and $\sigma$ are permutations equivalent by the reindexing $s \in \Sym(\calA)$.
Then the permutations $\pi'= R_{\Rtop} (\pi)$ and $\sigma' = R_{\Rtop} (\sigma)$ (or $R_{\Rbot}(\pi)$ and $R_{\Rbot}(\sigma)$ respectively) are also equivalent by $s$. 
Thus, the Rauzy class $\calR_\root$ inherits a directed graph structure from $\calD_\lab$. 
This graph is denoted $\calD_\root$: the \emph{Rauzy diagram}.

\begin{lemma}
	\label{l:regular_cover}
	(Any component of) $\calD_\lab$ is a regular cover of $\calD_\root$.
\end{lemma}


\begin{proof}
	Suppose that $\pi$ and $\pi'$ are labelled generalised permutations in (a component of) $\calD_\lab$.  
	Suppose that $\pi'$ lies in the fibre of $[\pi]$.  
	That is, there is some reindexing $s \in \Sym(\calA)$ so that $\pi' = s \circ \pi$.  
	Since $s$ gives a deck transformation of the covering $\calD_\lab \to \calD_\root$ (and so preserves components), we are done.
\end{proof}   

We now explain the coding of the diagonal flow using admissible parameters and Rauzy--Veech induction.
Suppose that $\pi$ is a labelled irreducible generalised permutation. 
Let $\calC(\pi) = q_\pi (P(\pi))$ be the corresponding polytope.
Let $\alpha = \pi(\ell)$ and $\beta = \pi(\ell + m)$ be the final letters on the top and the bottom.
Suppose that $(x, y)$ lies in $P(\pi)$. 
Thus $q = q_\pi(x, y)$ lies in $\calC(\pi)$. 
Suppose that $I = I(r)$ is the distinguished base-arc for $q$. 
Let $(t_-, t_+) = (t_-(q), t_+(q))$ be the largest interval so that
for all $t \in (t_-, t_+)$ we have $g_t(x, y)$ lying in $P(\pi)$. 
Note that $t_- < 0 < t_+$ because $I$ is the distinguished base-arc.

The parameters $g_{t_-}(x,y)$ and $g_{t_+}(x,y)$ satisfy all conditions of \Cref{d:polytope_of_parameters} \emph{except} the lower or upper bounds in the base-arc inequalities, respectively.
We define $q_- = g_{t_-} q$ and $q_+= g_{t_+} q$.
By~\cite[Lemma~2.12]{Boi-Lan}, the parameters $g_{t_-}(x, y)$ and $g_{t_+}(x, y)$ give zippered rectangles decompositions of $q_-$ and $q_+$ respectively. 
It also follows that
\begin{itemize}
	\item 
	$t_- = - \log |I|$ and $I$ gives a base-arc $I_-$ in $q_-$ with width one;
	\item 
	$t_+ = - \log (|I| - \min \{x_\alpha, x_\beta \})$ and $I$ gives a base-arc $I_+$ in $q_+$ whose width equals the upper bound in the base-arc inequalities.
\end{itemize}
We call $g_{t_-}(x,y)$ a \emph{backwards almost admissible parameter} for $\pi$ and we call $q_-$ a \emph{backwards almost admissible differential} for $\pi$.
We use similar language in the forwards direction.

\begin{definition}
	\label{d:backwards-flow-face}
	The subset $\bdy^- P(\pi)$ of backwards almost admissible parameters is the \emph{backwards flow face} of $P(\pi)$.
	We make similar definitions for $\bdy^- \calC(\pi)$ as well as $\bdy^+ P(\pi)$ and $\bdy^+ \calC(\pi)$. 
\end{definition}

\begin{notation}
	We take $\bdy^- P = \bigsqcup_\pi \bdy^- P(\pi)$ and $\bdy^- \calC = \bigsqcup_\pi \bdy^- \calC(\pi)$.
\end{notation}

\begin{lemma}
	\label{l:base-arc-discrete-at-lower-boundary}
	Suppose that the rooted differential $(q, v)$ lies in $\bdy^- \calC(\pi)$.
	Then there is some $\delta > 0$ so that: 
	\[
	(1 - \delta, 1 + \delta) \cap E(q, v) = \{1\}
	\]
\end{lemma}

\begin{proof}
	By definition, $1$ lies in $E(q,v)$ and there is some $\epsilon > 0$ (depending on $(q,v)$) so that 
	\begin{itemize}
		\item
		$g_\epsilon q$ is contained in $\calC(\pi)$ and
		\item 
		the geodesic segment $[q, g_\epsilon q]$ meets $\bdy^- \calC$ in exactly one point (that is, at $q$). 
	\end{itemize}
	By \Cref{d:polytope_of_differentials}, the set $E(g_\epsilon q, v)$ contains some $s < 1$. 
	Hence, $E(q, v)$ contains $e^{-\epsilon} s$.
	We conclude the proof by \Cref{l:base-arc-endpoints}.
\end{proof}

The length of the base-arc $I$ is the sum of the widths.
Therefore, $\bdy^- P(\pi)$ has codimension one in $P(\pi)$. 
Similarly, $\bdy^+ P(\pi)$ is contained in the union of two codimension-one loci.

Note that it is possible for $\bdy^+ \calC(\pi)$ and $\bdy^- \calC(\pi)$ to intersect. 
This happens if and only if there is a Rauzy--Veech move from $\pi$ to itself. 
In particular, if $R_{\Rtop} (\pi) = \pi$, then $\pi(\ell+m -1) = \alpha$.
Finally, if $R_{\Rbot} (\pi) = \pi$, then $\pi(\ell - 1) = \beta$.

The above discussion gives the following lemma.

\begin{lemma}
	\label{l:flow-face-parameters}
	The map $q_\pi \from P(\pi) \to \calC(\pi)$ extends continuously to 
	a homeomorphism 
	\[
	q_\pi^- \from P(\pi) \cup \bdy^- P(\pi) \to \calC(\pi) \cup \bdy^- \calC(\pi)
	\]
	Similarly, it extends continuously to a homeomorphism 
	\[
	q_\pi^+ \from P(\pi) \cup \bdy^+ P(\pi) \to \calC(\pi) \cup \bdy^+ \calC(\pi)
	\]
	These maps conjugate the diagonal flow on parameters to the diagonal flow on differentials. \qed
\end{lemma}

\begin{remark}
	There is also a continuous extension to $q_\pi^\pm \from P(\pi) \cup \bdy^\pm P(\pi)$.  
	However, this extension need not be a homeomorphism.
	See \Cref{s:crossing faces}.
\end{remark}

\begin{remark}
	Strictly speaking, we must define and give notation to the resulting quotient; that is, identifying points in $P \cup \bdy^\pm P$ having the same image under $\bigsqcup q_\pi^\pm$.
	As we shall see, the diagonal flow $g_t$ is well-defined on all but a measure zero subset of this quotient.  
	Also $\bigsqcup q_\pi$ is a measurable conjugacy between the diagonal flows on parameters and on differentials. 
	
	However, we will instead slightly abuse notation and use $P \cup \bdy^- P$ to denote this quotient.
\end{remark}

Let $\bdy^+ \calC_{\Rtop}(\pi)$ be the subset $\bdy^+ \calC(\pi) \cap q_\pi ( \{ (x,y) : x_\alpha > x_\beta \} )$. 
Similarly, let $\bdy^+ \calC_{\Rbot}(\pi)$ be the subset $\bdy^+ \calC(\pi) \cap q_\pi ( \{ (x,y) : x_\beta > x_\alpha \} )$.

\begin{lemma}
	\label{l:flow-face-identifications}
	Let $q$ be a rooted differential in $\bdy^+ \calC_{\Rtop}(\pi)$ (respectively $\bdy^+ \calC_{\Rbot}(\pi)$). 
	Then $q$ is contained in $\bdy^- \calC(R_{\Rtop}(\pi))$ (respectively $\bdy^- \calC(R_{\Rbot}(\pi))$).
\end{lemma}

\begin{proof}
	Suppose that $q$ lies in $\bdy^+ \calC(\pi)$.
	By \Cref{l:flow-face-parameters} there are parameters $(x, y)$ in the upper boundary of $P(\pi)$ so that $q = q_\pi (x, y)$. 
	Let $I$ be the associated base-arc for $q$. 
	Breaking symmetry, suppose that $x_\alpha > x_\beta$.
	(That is, $q$ lies in $\bdy^+ \calC_{\Rtop}(\pi)$.)
	Thus $|I| = 1 + x_\beta$.
	By \Cref{e:parameter-change} we can perform a Rauzy--Veech move and obtain $R_{\Rtop}(\pi, x, y) = (\pi', x', y')$.
	The width of the base-arc $I'$ after the move is $|I| - x_\beta = 1$. 
	Thus $q$ lies in $\bdy^- \calC(\pi') = \bdy^- \calC(R_{\Rtop}(\pi))$, as desired.
\end{proof}

\subsection{The diagonal flow and Rauzy-Veech induction}

Suppose that $q$ is contained in $\calC(\pi)$ with the distinguished base-arc $I$.
The rooted differential $q_+ = g_{t_+} q$ lies in $\bdy^+ \calC(\pi)$. 
Suppose that the parameters for $q$, hence for $q_+$, satisfy $x_\alpha  \neq x_\beta$.
By \Cref{l:flow-face-identifications}, the roote differential $q_+$ is contained in $\bdy^- \calC(\pi')$ where $\pi' = R_{\ast} (\pi)$ (here $\ast = \Rtop$ or $\ast = \Rbot$ depending on which of $x_\alpha$, $x_\beta$ is larger). 
We repeat the process by flowing $q_+$ forward to get a rooted differential $q'_+$ in $\bdy^+ \calC(\pi')$.
Let $\alpha', \beta' \in \calA$ be the top and bottom rightmost letters in $\pi'$.
If the $\alpha'$ and $\beta'$ widths in $q'_+$ are not equal then we may define $\pi'' = R_\ast (\pi')$, we see that $q'_+$ lies in $\bdy^- \calC(\pi'')$, and we may continue to iterate.
The iteration stops (after finitely many Rauzy--Veech moves) if and only if $q$ has a vertical saddle connection~\cite[Proposition~4.2]{Boi-Lan}.

Since $\calC(\pi) = \calC(\sigma)$ for equivalent permutations, the quotient graph, namely the Rauzy diagram $\calD_\root$, codes the diagonal flow on $\calC_\root$.

\begin{remark}
	\label{r:strong}
	As recalled in \Cref{a:measures}, the Masur--Smillie--Veech measure is a diagonal flow-invariant measure that is ergodic.
	By ergodicity, a set of positive measure in $\calC(\pi) - \calV$, visits every $\calC(\sigma)$ under the diagonal flow. 
	This implies that the Rauzy diagram $\calD_\root$ is \emph{strongly connected}: 
	there is a directed path between any pair of vertices. 
	Since each component of $\calD_\lab$ evenly covers $\calD_\root$, the components of $\calD_\lab$ are also strongly connected (as noted by Boissy--Lanneau \cite{Boi-Lan}).
\end{remark}

We fix a component of $\calD_\lab$ for the rest of the article.
With a slight abuse of notation, we (again) simply call this component $\calD_\lab$.

The lemmas below are standard facts in Rauzy--Veech theory and are often implicitly used. 
We include proof sketches for completeness.
Recall from \Cref{d:base-arc-endpoints} that $E(q, v)$ is the set of endpoints of base-arcs.

\begin{lemma}
	\label{l:finite-flow-faces}
	Suppose that $\pi$ is a labelled permutation in $\calD_\lab$.
	Suppose that $(q, v)$ is a rooted differential in the polytope $\calC(\pi)$. 
	Let $T_0 = - \log(\inf E(q, v))$. 
	(If the infimum is zero, we take $T_0 = \infty$.)
	Then, for any positive $T < T_0$, the diagonal flow segment $[q, g_T q]$ crosses only finitely many flow faces. 
\end{lemma} 

\begin{proof}
	Let $\bdy^- \calC_1, \bdy^- \calC_2, \dots$ be the sequence of backwards flow faces crossed by the segment $[q, g_T q]$.
	Let $0 < t_k \leq T$ be the monotonically increasing sequence of times such that $g_{t_k} q$ is contained in $\bdy^- \calC_k$.
	Let $I_v$ be the horizontal separatrix emanating from the root.
	It follows that the sequence of points $e^{-t_k}$ in $I_v$ give base-arcs for $q$. 
	Thus they are all contained in $E(q,v)$.
	
	Since $T < T_0$, by \Cref{l:base-arc-endpoints} the intersection $E(q, v) \cap [e^{-T}, 1)$ is finite. 
	Hence, the sequence $t_k$ is finite and we are done.
\end{proof}

\begin{remark}
	\label{r:deltas-finite-code}
	Recall that the differentials $q_k$ introduced in the proof of \Cref{t:based-loops} did not have any vertical or horizontal saddle connections.  
	Thus by \Cref{l:base-arc-endpoints} the geodesic segments $\delta_k$ (also introduced in the proof) satisfy the hypothesis of \Cref{l:finite-flow-faces}. 
	Thus the $\delta_k$ cross only finitely many flow faces.
\end{remark} 

\begin{definition}
	A \emph{Rauzy--Veech sequence} is a concatenation of Rauzy--Veech moves labelling a directed path in $\calD_\lab$.
\end{definition}

Let $\zeta$ be a finite Rauzy--Veech sequence that starts and ends at labelled permutations $\pi$ and $\pi'$.
Let $P(\zeta) \subseteq P(\pi)$ be the parameters whose Rauzy--Veech sequence begins with $\zeta$. 
By induction on the number of moves, we conclude that $P(\zeta)$ is a convex open subset of $P(\pi)$.
In particular, $\calC(\zeta) = q_\pi(P(\zeta))$ is path connected.

\begin{definition}
	\label{d:zeta-segment}
	Suppose that $q$ is a differential and $t$ is a time. 
	Suppose that $q$ lies in $\calC(\pi)$ and $q' = g_t q$ lies in $\calC(\pi')$.
	We say that the geodesic segment $[q, q']$ is a \emph{$\zeta$--segment} if the Rauzy--Veech sequence of $[q, g_t q]$ from $\pi$ is $\zeta$.
\end{definition}

It follows that $q \in \calC(\zeta)$.
We shall see (in \Cref{r:segments-non-empty}) that every finite Rauzy--Veech sequence $\zeta$ admits $\zeta$--segments. 

\begin{remark}
	\label{r:zeta-segment}
	Suppose that $q$ is a differential and $t$ is a time.
	Suppose that $q$ lies in $\calC(\pi) \cup \bdy^- \calC(\pi)$ and $q' = g_t q$ lies in $\calC(\pi') \cup \bdy^- \calC(\pi')$.
	Then, by \Cref{d:backwards-flow-face}, there is a small $\epsilon > 0$ (depending on $q$ and $q'$) so that 
	\begin{itemize}
		\item $g_\epsilon q$ lies in $\calC(\pi)$,
		\item $g_\epsilon q'$ lies in $\calC(\pi')$, and
		\item the interiors of the segments $[q, g_\epsilon q]$ and $[q', g_\epsilon q']$ do not cross flow faces.
	\end{itemize}
	We say that the geodesic segment $[q, q']$ \emph{$\zeta$--segment} if $[g_\epsilon q, g_\epsilon q']$ is.
\end{remark}

\begin{lemma}
	\label{l:segment-factor} 
	Suppose that $\zeta = \zeta' \zeta''$ is a concatenation of finite Rauzy--Veech sequences. 
	Suppose that $[q, g_t q]$ is a $\zeta$--segment. 
	Then there is a time $s \in [0, t]$ so that $[q, g_s q]$ is a $\zeta'$--segment and $[g_s q, g_t q]$ is a $\zeta''$--segment.
\end{lemma}

\begin{proof}
	As in \Cref{r:zeta-segment} we may (if needed) move the endpoints of the segment slightly to ensure that the endpoints lie in the interiors of polytopes.
	
	Note that, By definition, $P(\zeta)$ is a subset of $P(\zeta')$.
	Let $q = q_\pi(x,y)$.
	Then $(x, y)$ lies in $P(\zeta)$, hence in $P(\zeta')$.
	Thus there is a time $s \leqslant t$ such that $[q, g_s q]$ is a $\zeta'$--segment. 
	Suppose that $\zeta'$ ends at the labelled permutation $\pi'$.
	Since $[q, g_s q]$ can be extended by $\zeta''$ to a $\zeta$--segment, it follows that $[g_s q, g_t q]$ is a $\zeta''$--segment. 
\end{proof}

The segments associated with finite Rauzy-Veech sequences are ``unique up to homotopy'', as follows.

\begin{lemma}
	\label{l:segment-homotopy} 
	Suppose that $\zeta$ is a finite Rauzy--Veech sequence.
	Then any pair of $\zeta$--segments are isotopic in $\calC_\root$ through $\zeta$--segments.
\end{lemma}

\begin{proof}
	Suppose that $\zeta$ starts and ends at labelled permutations $\pi$ and $\pi'$. 
	Let $[p, g_t p]$ and $[q, g_{t'} q]$ be $\zeta$--segments.
	Applying \Cref{r:zeta-segment} we may assume that all endpoints lie in the interiors of polytopes. 
	Thus $p$ and $q$ lie in $\calC(\zeta)$, the image of $P(\zeta)$. 
	Suppose that $\rho \from [0, 1] \to P(\zeta)$ is a line segment connecting the parameters of $p$ to those of $q$.  
	Then for any $r \in [0, 1]$ we have that $q_\pi(\rho(r))$ lies in $\calC(\zeta)$.  
	Note that $q_\pi(\rho(0)) = p$ and $q_\pi(\rho(1)) = q$.
	Set $q_r = q_\pi(\rho(r))$.
	We choose times $t(r)$ continuously so that $t(0) = t$, so that $t(1) = t'$, and so that $[q_r, g_{t(r)} q_r]$ is a $\zeta$--segment. 
	This gives the required isotopy. 
\end{proof}

We give a full exposition of the \emph{width} and \emph{height cones} in \Cref{a:measures}, starting from \Cref{d:cones}.
Suppose that $\zeta$ is a finite Rauzy--Veech sequence from $\pi$ to $\pi'$.
We define $X(\zeta)$ to be the cone of widths in $X(\pi)$ whose Rauzy--Veech sequence is given by $\zeta$. 
Note that $X(\zeta)$ is the cone over the width parameters of points in $P(\zeta)$.
After recalling \Cref{e:parameter-change} and \Cref{conv:column-vectors}, we induct on the length of $\zeta$ to obtain a matrix $E_\zeta = E_1 \cdots E_n$. 
Here $E_i$ is the elementary matrix for the $i^{\text{th}}$ Rauzy--Veech move in $\zeta$.
If $(x, y)$ lies in $P(\zeta)$, then the new widths and heights $(x', y')$ lie in $P(\pi')$.  
The new widths and the old are related as follows.
\begin{equation}
	\label{e:RV-matrix}
	(E_\zeta x', E_\zeta y') = (x,y) 
\end{equation}
This extends \Cref{e:parameter-change}. 
In particular, we have 
\begin{equation}
	\label{e:X-cones}
	X(\zeta) = E_\zeta (X(\pi')) \subset X(\pi)
\end{equation}

\begin{remark}
	\label{r:segments-non-empty}
	Note that $X(\zeta)$ is nonempty.
	Thus any finite Rauzy--Veech sequence $\zeta$ admits $\zeta$--segments.
\end{remark}

\begin{lemma}
	\label{l:width-nesting}
	Suppose that $\pi$ is a labelled permutation.
	Suppose that $U$ is an open subset of $P(\pi)$. 
	Then there is a finite Rauzy--Veech sequence $\theta$, starting from $\pi$, so that 
	\[
	\closure{X(\theta)} - \{0\} \subset X(U)
	\]
	The same holds for any extension $\theta \eta$ of $\theta$.
\end{lemma}

We also give a version of this for the height cones in \Cref{l:height-nesting}.

\begin{proof}[Proof of {\Cref{l:width-nesting}}]
	By Masur~\cite{Mas} and Veech~\cite{Vee82} (more strongly
	by Kerckhoff, Masur, and Smillie \cite{Ker-Mas-Smi}, see also~\cite[Section~4.4, Theorem]{Yoc06}), a typical differential (with respect to the Masur--Smillie--Veech measure - see \Cref{a:measures}) has a uniquely ergodic vertical foliation. 
	Suppose that $q \in q_\pi(U)$ is one such. 
	Let $(x, y) \in P(\pi)$ be the parameters for $q$. 
	By definition, $x$ lies in $X(U)$.
	Also, $q$ admits an infinite Rauzy--Veech sequence $\zeta$. 
	Let $\zeta_n$ be its prefix of length $n$. 
	Let $E_n$ be the Rauzy matrix of $\zeta_n$. 
	Suppose that $\zeta_n$ ends at the labelled permutation $\pi_n$.
	Note that $X(\zeta_n) = E_n X(\pi_n)$ gives a sequence of nested cones in $X(\pi)$.
	By unique ergodicity, the intersection of the closures of these cones is exactly the ray through $x$. 
	(In the abelian case this is~\cite[Section~4.4, Proposition]{Yoc06}; 
	the same proof works in the quadratic case.)
	Since $U$ is open, for some sufficiently large $n$ we have that $\closure{X(\zeta_n)}- \{0\}$ is contained inside $X (U)$. 
	Setting $\theta$ to be $\zeta_n$, we are done.
\end{proof}

By \Cref{l:width-nesting}, it follows that if $q$ lies in $\calC(\theta)$ then the maximal geodesic segment in $\calC(\pi)$ containing $q$ intersects $q_\pi(U)$.

\section{The flow group is the fundamental group}

\subsection{Fundamental groups of directed graphs}

We will need the following lemma.

\begin{lemma}
	\label{l:directed-loops-suffice}
	Suppose that $\calD$ is a directed graph which is strongly connected.
	Suppose that $\pi$ is a vertex of $\calD$. 
	Then $\uppi_1(\calD, \pi)$ is generated by the homotopy classes of directed loops based at $\pi$. \qed
\end{lemma}

\begin{remark}
	If $\calD$ is finite then finitely many directed loops suffice. 
\end{remark}

\subsection{Rauzy diagrams and the flow group}
\label{s:Rauzy_diagrams_flow_group}

Suppose that $\calD_\root$ is a Rauzy diagram associated to stratum component of rooted differentials $\calC_\root$.
Let $\calD_\lab$ be (a component of) the labelled Rauzy diagram.

For every vertex $\pi$ in $\calD_\lab$ we choose a base-point $q^\pi$ in the interior of $\calC(\pi)$.  
For every directed edge $e = (\pi, \pi')$ in $\calD_\lab$ we choose a base-point $q^e$ in $\bdy^+ \calC(\pi) \cap \bdy^- \calC(\pi')$: the intersection of the forwards and backwards flow faces.
For example, we could produce such $q^e$ using an $e$--segment.

Suppose that $s \in \Sym(\calA)$ is a reindexing. 
By \Cref{l:equivalent-implies-identicalpolytopes}, we may choose base-points in polytopes to satisfy $q^\pi = q^{s \circ \pi}$.
By \Cref{l:regular_cover}, $(\pi, \pi')$ is an edge if and only if $(s \circ \pi, s \circ \pi') $ is an edge. 
Hence, we may choose base-points on flow faces to satisfy $q^e = q^{e'}$ if and only if $e$ and $e'$ are equivalent by reindexing.

By connecting the base-points according to the combinatorics of Rauzy--Veech moves, we build a graph $\calE_\root$ that is isomorphic to $\calD_\root$.
We choose a homeomorphism from $\calD_\root$ to $\calE_\root$ which (for all $\pi$ and $e$) sends $\pi$ to $q^\pi$ and sends the midpoint of $e$ to $q^e$. 

Fix a base-point $[\pi]$ for $\calD_\root$. 
The homeomorphism from $\calD_\root$ to $\calE_\root$ induces a homomorphism 
\begin{equation}
	\label{e:pi1-homomorphism}
	\Psi \from \uppi_1(\calD_\root, [\pi]) \to \uppi_1(\calC_\root, q^\pi)
\end{equation}

Here is our first consequence of \Cref{t:based-loops}.

\begin{theorem}
	\label{t:pi1-surjective}
	Suppose that $\calC_\root$ is a component of a stratum of rooted abelian or quadratic differentials. Suppose that $\pi$ a labelled permutation in $\calD_\lab$. 
	Then the homomorphism $\Psi$ (given in \Cref{e:pi1-homomorphism}) is surjective.
\end{theorem}

\begin{proof}
	Suppose that $\gamma$ is a loop in $\calC_\root$ based at $q^\pi$. 
	By \Cref{t:based-loops}, we can homotope $\gamma$ (relative to its base-point) to a finite concatenation of paths $\gamma_i$ where each $\gamma_i$ is either a (forwards or backwards) geodesic segment or is contained inside a polytope.   
	Breaking symmetry, we assume $\gamma$ is the concatenation $\gamma_1 \gamma_2 \dotsb \gamma_k$ where the odd indexed $\gamma_i$ are contained in a polytope and the even indexed $\gamma_i$ are (forwards or backwards) geodesic segments.
	
	Applying \Cref{r:deltas-finite-code}, we have finite Rauzy--Veech sequences $\zeta_{2i}$ in $\calD_\lab$ as follows: 
	\begin{itemize}
		\item traversed forward, each $\gamma_{2i}$ is a $\zeta_{2i}$--segment;
		\item $\zeta_2$ starts at $\pi$; and
		\item the end of $\zeta_{2i}$ is the beginning of $\zeta_{2i + 2}$.
	\end{itemize}
	Each $\zeta_{2i}$ descends to a path $\kappa_{2i}$ in $\calD_\root$. 
	By \Cref{l:segment-factor}, the geodesic $\gamma_{2i}$ visits (in the correct order) the polytopes coded by $\zeta_{2i}$.  
	Since the $\calC(\pi)$ are polytopes (and the same essentially holds for the flow faces) we can homotope $\gamma_{2i}$ to move its intersections with the flow faces to the points $q^e$.  
	This done we may homotope each $\gamma_{2i + 1}$ to the corresponding point $q^\pi$ and each $\gamma_{2i}$ (relative to the points $q^e$) to (the image in the graph $\calE_\root$ of) the corresponding path $\kappa_{2i}$.
\end{proof}

\begin{remark}
	\label{r:calC-lab}
	The components of $\calD_\lab$ evenly cover $\calD_\root$. 
	Thus, by \Cref{t:pi1-surjective}, corresponding to each component of $\calD_\lab$ there is a finite cover $\calC_\lab$ of $\calC_\root$. 
	This cover has an intrinsic description for abelian components (see~\cite[Proposition~1.2]{Boi15}).  
	It is an open problem to give an intrinsic description of $\calC_\lab$ in the quadratic case (see~\cite[Appendix~A]{Boi15}).
\end{remark}

Suppose that $q_0$ is a base-point in $\calC_\root$. 
Suppose that $U$ is a contractible open set around $q_0$.
For every $q \in U$, we choose a path $\eta_q$ from $q_0$ to $q$ inside of $U$.
As $U$ is contractible, any choice of $\eta_q$ is homotopic to any other relative to the end-points. 
By $\overline{\eta}_q$, we mean $\eta_q$ traversed in reverse.

Suppose that $\gamma$ is a geodesic segment whose endpoints $q$ and $q'$ lie in $U$.
The concatenation $\eta_q \gamma \overline{\eta}_{q'}$ is a loop in $\calC_\root$ based at $q_0$.
We call such loops \emph{almost-flow loops} based at $(U, q_0)$.

\begin{definition}
	The flow group $\Flow(U, q_0)$ is the subgroup of $\uppi_1(\calC_\root, q_0)$ generated by the almost-flow loops based at $(U, q_0)$. 
\end{definition}

Here is one of our main results. 

\begin{theorem}
	\label{t:flow-general}
	Suppose that $U$ is a contractible open set in $\calC_\root$. 
	Suppose that $q_0$ is a differential in $U$.
	Then 
	\[
	\Flow(U, q_0) = \uppi_1( \calC_\root, q_0)
	\]
\end{theorem}

\begin{proof}
	By \Cref{l:polytopes-dense}, the union of the polytopes $\calC(\pi)$ is dense in $\calC_\root$. 
	Thus, applying the change of base-point isomorphism for $\uppi_1$, we may assume that $q_0$ lies in some $\calC(\pi)$.
	Note that if $V \subset U$ also contains $q_0$ then $\Flow(V, q_0) \leq \Flow(U, q_0)$. 
	Thus we may assume that $U$ is contained in $\calC(\pi)$.
	Applying \Cref{l:directed-loops-suffice} and \Cref{t:pi1-surjective}, it suffices to show that for any directed based loop $\zeta$ in $\calD_\root$ there is a product of almost flow loops based at $(U, q_0)$ whose Rauzy--Veech sequences concatenate give the unique lift of $\zeta$ to $\calD_\lab$ starting at $\pi$.
	
	By \Cref{l:width-nesting}, there is a finite Rauzy--Veech sequence $\theta$ starting at $\pi$ so that 
	$\closure{X(\theta)} \subset X(q_\pi^{-1}(U))$. 
	It follows that the same is true for any finite extension of $\theta$. 
	So we may use strong connectivity of $\calD_\lab$ to extend $\theta$ until it ends at $\pi$.
	Thus $\theta$ is a loop in $\calD_\lab$.
	Thus $\theta$ also gives a loop in $\calD_\root$ -- in a small abuse of notation we use the same name for both.
	
	We now must force a geodesic segment to meet $U$ twice.
	Suppose that $q$ lies in the intersection of $\calC(\theta^2)$ and $U$.
	Let $[q, g_b q]$ be a $\theta^2$--segment. 
	By \Cref{l:segment-factor}, there is a time $a \in (0, b)$ as follows:
	\begin{itemize}
		\item 
		$[q, g_a q]$ and $[g_a q, g_b q]$ are both $\theta$--segments and
		\item 
		$g_a q$ is contained in $\calC(\theta) \cap U$.
	\end{itemize}
	Thus $\gamma_\theta = [q, g_a q]$ begins and ends in $U$. 
	The resulting almost-flow loop $\eta_q \gamma_\theta \overline{\eta}_{g_a q}$ lies in the based homotopy class $\Psi([\theta])$ (with $\Psi$ as defined in \Cref{e:pi1-homomorphism}).
	
	Suppose that $\zeta$ is any directed based loop in $\calD_\root$.
	The concatenation $\theta \zeta \theta$ is a directed based loop in $\calD_\root$. 
	Suppose that this lifts to the Rauzy--Veech sequence $\xi$ in $\calD_\lab$, starting from the labelled permutation $\pi$.
	
	Since $\xi$ extends $\theta$, there is a differential $q$ in the intersection of $\calC(\xi)$ and $U$.  
	Let $\gamma_\xi = [q, g_d q]$ be a $\xi$--segment. 
	By \Cref{l:segment-factor}, there is a time $c \in (0, d)$ as follows:
	\begin{itemize}
		\item 
		$[q, g_c q]$ is a $\theta \zeta$--segment,
		\item 
		$[g_c q, g_d q]$ is a $\theta$--segment, and
		\item 
		$g_c q$ is contained in the intersection $\calC(\theta) \cap U$.
	\end{itemize}
	Thus $\gamma_{\theta \zeta}= [q, g_c q]$ begins and ends in $U$.
	The resulting almost-flow loop $\eta_q \gamma_{\theta \zeta} \overline{\eta}_{g_c q}$ lies in the based homotopy class $\Psi([\theta \zeta])$. 
	
	Since $\Flow(U,q_0)$ is a subgroup we deduce that there is a product giving $\Psi([\zeta])$, as desired.
\end{proof}

We also obtain a ``flow loop'' version of \Cref{t:flow-general}. 
Suppose that $q_0$ is a base-point in $\calC_\root$. 
Suppose that $U$ is a contractible open set around $q_0$.

\begin{definition}
	The \emph{strict flow group} $\SFlow(U, q_0)$ is the subgroup of $\Flow(U, q_0)$ generated by loops of the form $\eta_q \gamma \overline{\eta}_q$ where 
	\begin{itemize}
		\item
		$\gamma$ is a closed geodesic that intersects $U$ and
		\item
		$q$ is any point of $U \cap \gamma$. \qedhere
	\end{itemize}
\end{definition}

Recall that $Y(\pi)$ is the open cone of height parameters for $\pi$ (\Cref{d:cones}). 
Recall that $P(\pi)$ is the polytope of parameters for $\pi$.
For any open set $U$ in $P(\pi)$, we define $Y(U)$ to be the cone over the height parameters of points in $U$.

Suppose that $\zeta$ is a finite Rauzy--Veech sequence from $\pi$ to $\pi'$.
We define 
\begin{equation}
	\label{e:Y-cones}
	Y(\zeta) = E_\zeta^{-1} (Y(\pi)) \subset Y(\pi')
\end{equation}
Recall that widths and heights are, respectively, unstable and stable for the flow. 
This is why when we go from \Cref{e:X-cones} to \Cref{e:Y-cones} we must replace $E_\zeta$ by $E_\zeta^{-1}$.

\begin{lemma}
	\label{l:height-nesting}
	Suppose that $\pi$ is a labelled permutation.
	Suppose that $U$ is an open subset of $P(\pi)$. 
	Then there is a finite Rauzy--Veech sequence $\theta$, ending at $\pi$, so that 
	\[
	\overline{Y(\theta)} - \{ 0 \} \subset Y(U)
	\]
	The same holds for any extension $\eta \theta$ of $\theta$.
\end{lemma}

\begin{proof}
	The proof is similar to that of \Cref{l:width-nesting} but instead using the genericity of uniquely ergodic horizontal foliations.    
\end{proof}

\begin{definition}
	\label{d:strong-nesting}
	Suppose that $U$ is an open set contained in a polytope $\calC(\pi)$.
	We call a based (at $\pi$) Rauzy--Veech loop $\theta$ \emph{strongly nested for $U$} if 
	\begin{itemize}
		\item $\closure{X(\theta)} - \{0\}$ is contained in $X(q_\pi^{-1}(U))$; and 
		\item 
		$\closure{Y(\theta)} - \{0\}$ is contained in $Y(q_\pi^{-1}(U))$.
	\end{itemize}    
\end{definition}

\begin{lemma}
	\label{l:strong_nesting_existence}
	Suppose that $U$ is an open set contained in $\calC(\pi)$.
	Then there is a Rauzy--Veech loop $\theta$, based at $\pi$, that is strongly nested for $U$.
\end{lemma}

\begin{proof}
	By \Cref{l:width-nesting}, there is a based (at $\pi$) Rauzy--Veech loop $\zeta$ such that $\closure{X(\zeta)} - \{0\}$ is contained in $X(q_\pi^{-1}(U))$. 
	By \Cref{l:height-nesting}, there is a based (also at $\pi$) Rauzy--Veech loop $\zeta'$ such that $\closure{Y(\zeta')} - \{0\}$ is contained in $Y(q_\pi^{-1}(U))$.     
	We deduce that the concatenation $\zeta \zeta'$ is strongly nested for $U$.
\end{proof}

When $X(q_\pi^{-1}(U)) = X(\pi)$ then positivity of the Rauzy matrix implies (but is not implied by) nesting in the widths.
For abelian strata, Marmi--Moussa--Yoccoz produce a sequence $\zeta$ so that the matrix $E_\zeta$ is positive;  see~\cite[Proposition~2]{Yoc06}. 
By extending it, we may assume $\zeta$ is a based loop.
The same proof works for quadratic strata. 
For related discussion see~\cite[Lemma~3.12]{Avi-Res}.
This gives the following.

\begin{corollary}
	\label{c:strong_nesting_by_positive_matrix}
	Suppose that $U \subset \calC(\pi)$ is an open subset. 
	Then there is a based loop $\theta$ which is strongly nested for $U$ and so that its Rauzy matrix $E_\theta$ is positive.
\end{corollary}

\begin{proof}
	Let $\zeta$ be a based loop so that the matrix $E_\zeta$ is positive. 
	By \Cref{l:strong_nesting_existence}, there is a Rauzy--Veech loop $\eta$, based at $\pi$, that is strongly nesting for $U$.
	The based loop $\theta = \eta \zeta \eta$ is thus strongly nesting for $U$ and has a positive Rauzy matrix.
\end{proof}

Strong nesting using a sequence with a positive Rauzy matrix simplifies the discussion in \Cref{s:dynamics}.

\begin{corollary}
	\label{c:strict-flow}
	Suppose that $q_0$ lies in $\calC_\root$. 
	Suppose that $U$ is a contractible open set containing $q_0$.
	Then
	\[
	\SFlow(U, q_0) = \uppi_1( \calC_\root, q_0).
	\]
\end{corollary}

\begin{proof}
	By \Cref{t:flow-general}, it suffices to prove that $\SFlow(U, q_0) = \Flow(U, q_0)$.  As before, we may (and do) assume that $U$ is contained in some polytope $\calC(\pi)$. 
	
	Suppose that $\gamma$ is a geodesic segment that starts at $q$ and ends at $q'$ in $U$.
	So $\eta_q \gamma \overline{\eta}_{q'}$ is an almost flow loop. 
	Since $\gamma$ returns to $\calC(\pi)$ we deduce that $\gamma$ is a $\zeta$--segment, for some based loop $\zeta$ in $\calD_\lab$ starting from $\pi$.
	
	Let $\theta$ be a Rauzy--Veech loop, based at $\pi$, that is strongly nested for $U$.
	Suppose that $\xi$ is any other Rauzy--Veech loop, based at $\pi$ (perhaps of length zero).  
	Then the extension $\theta \xi$ satisfies
	\begin{itemize}
		\item 
		the sequence $X((\theta \xi)^n)$ converges as $n \to \infty$ to a ray in $X (q_\pi^{-1}(U))$ and
		\item 
		the sequence $Y((\theta \xi)^n)$ converges as $n \to \infty$ to a ray in $Y (q_\pi^{-1}(U))$.
	\end{itemize}
	Let $x^{\theta \xi}$ and $y^{\theta \xi}$ be parameters that span the rays.
	Then, after appropriate scalings,
	we have that $q^{\theta \xi} = q_\pi (x^{\theta \xi}, y^{\theta \xi})$ lies in $U$. 
	Let $\gamma^{\theta}$ and $\gamma^{\theta\xi}$ be the orbits of $q^{\theta}$ and $\gamma^{\theta\xi}$, respectively, under the diagonal flow.
	Then:
	\begin{itemize}
		\item 
		$\gamma^{\theta}$ and $\gamma^{\theta\xi}$ are closed geodesics and 
		\item 
		the Rauzy--Veech sequences have, respectively, period $\theta$ and $\theta \xi$.
	\end{itemize}
	We use $p'$ to denote $q^\theta$ and $p''$ to denote $q^{\theta \zeta}$.
	Now consider the closed geodesics $\gamma' = \gamma^{\theta}$ and $\gamma'' = \gamma^{\theta\zeta}$.
	The based homotopy classes of $\eta_{p'} \gamma' \overline{\eta}_{p'}$ and $\eta_{p''}\gamma'' \overline{\eta}_{p''}$ are generators of $\SFlow(U, q_0)$.
	Also, \Cref{l:segment-homotopy}, the original almost flow loop $\eta_q \gamma \overline{\eta}_{q'}$
	is homotopic to 
	\[
	\overline { 
		\eta_{p'}\gamma' \overline{\eta}_{p'}
	} \cdot
	\eta_{p''} \gamma'' \overline{\eta}_{p''}
	\]
	Thus the class of $\eta_q \gamma \overline{\eta}_{q'}$ belongs to $\SFlow(U,q_0)$, as desired.
\end{proof}

\section{Dynamics of the diagonal flow}\label{s:dynamics}

\subsection{Coding formalism}

Let $\Pi$ be a finite or countable alphabet. 
We consider the symbolic space $\Sigma = \Pi^\ZZ$ endowed with the left shift map $\shift$. 
Suppose that $u \in \Pi^m$ and $v \in \Pi^n$ are finite words. 
We write $uv \in \Pi^{m+n}$ for their concatenation.
The (forward) cylinder $\Sigma(u)$ induced by $u$ is defined as 
\[
\Sigma(u) = \{ a \in \Sigma \st a_k = u_k \text{ for } 0 \leq k < m \}
\]

\begin{definition}
	\label{d:bounded-distortion}
	Suppose that $\mu$ is an $\shift$--invariant finite measure on $\Sigma$.
	We say that $\mu$ has \emph{bounded distortion} if there exists a constant $M \geq 1$ such that, for any finite words $u \in \Pi^r$ and $v \in \Pi^s$,
	\[
	\frac{1}{M} \mu(\Sigma(u v)) \leq \mu(\Sigma(u))\mu(\Sigma(v)) \leq M \mu(\Sigma(u v)) \qedhere
	\]
\end{definition}
\noindent
This generalises the property of being a \emph{Bernoulli shift}, where $M = 1$. 
For this reason, it is also called an \emph{approximate product structure}.

\subsection{A smaller section for the Rauzy--Veech renormalisation}

Suppose that $\calC_\root$ is a component of rooted differentials from a stratum component $\calC$. 
We take $\calC^{(1)}_\root$ to be the locus of rooted differentials of area one.
Note that this is preserved by the diagonal flow. 

\begin{notation}
	Suppose that $U$ is a subset of $\calC_\root$ closed under scaling: that is, $\RR \cdot U = U$. 
	Then we set $U^{(1)} = U \cap \calC^{(1)}_\root$.  
	We use similar notation in the parameter space.
\end{notation}

We now turn to the task of building a coding for the system $(\calC^{(1)}_\root, g_t)$.
The natural alphabet is in bijection with single moves in $\calD_\root$ (and so is finite).
The corresponding Poincaré section is 
\[
\bdy^- \calC^{(1)}_\root = \bigsqcup_{[\pi]} \bdy^- \calC^{(1)}(\pi)
\]
As discussed in \Cref{s:measures on backwards flow faces} (see in particular \Cref{l:RVinvariance}) this section carries the measure $\nu^{(1)}$ which is invariant for the Rauzy--Veech renormalisation.

However, there are two problems with this choice of alphabet and section.
\begin{itemize}
	\item
	Every backwards flow face $\bdy^- \calC^{(1)}(\pi)$ has infinite $\nu^{(1)}$--mass. See~\cite{Vee82} for a proof in the abelian case; 
	a similar calculation holds in the quadratic case.
	\item 
	The measure $\nu^{(1)}$ does not have bounded distortion for finite Rauzy--Veech sequences.
	This is a consequence of \cite[Proposition~5.2]{Vee78} in the abelian case and of~\cite[Equation~8.3 and Lemma~8.5]{Gad} in the quadratic case.
\end{itemize}
These imply, respectively, that the existence of the Lyapunov spectrum is unclear and that we cannot apply the simplicity criterion of Avila--Viana~\cite{Avi-Via07a}.

In the abelian case, Zorich~\cite{Zor97} gives an acceleration of the Rauzy--Veech renormalisation. 
The result is a larger (now countable) alphabet but a smaller Poincaré section with finite $\nu^{(1)}$--mass.  
This deals with the first problem, but not the second. 
To control the distortion, again in the abelian case, Avila--Gouëzel--Yoccoz give a further acceleration that results in a pre-compact (in $\bdy^- \calC^{(1)}$) Poincaré section. 

Here we extend their technique to deal with quadratic strata.
Suppose that $\pi$ is a labelled irreducible generalised permutation.
Define $X_1(\pi) = \{ x \in X(\pi) \st w_X(x) = 1 \}$
where $w_X$ is the sum of the widths.
We also take $\bdy^- P(\pi)$ to be the product $X_1(\pi) \cross Y(\pi)$.
Let $\bdy^- P^{(1)}(\pi)$ be the area-one locus in $X_1(\pi) \cross Y(\pi)$. 
(We give a more detailed discussion in \Cref{s:measures on backwards flow faces}.)

\begin{notation}
	We take $X_1 = \bigsqcup_\pi X_1(\pi)$ and $\bdy^- P^{(1)} = \bigsqcup_\pi \bdy^- P^{(1)}(\pi)$.
\end{notation}

Let $\xi$ be a Rauzy--Veech sequence starting at $\pi$ and ending at $\pi'$.
We define $\bdy^- P^{(1)}(\xi)$ to be the subset of $\bdy^- P^{(1)}(\pi)$ whose Rauzy--Veech expansion begins with $\xi$. 
Take $\bdy^- \calC^{(1)}(\xi) = q_\pi (\bdy^- P^{(1)}(\xi))$. 
Let $p^\pi$ be the projection to widths. 
We define $X_1(\xi) \subset X_1(\pi)$ to be the image of $\bdy^- P^{(1)}(\xi)$ under $p^\pi$.
By composing the Rauzy--Veech renormalisation maps $\RV^\bdy$ (defined in \Cref{s:measures on backwards flow faces}) for moves in $\xi$, we obtain a map $\RV^\bdy_\xi \from X_1(\pi') \cross Y(\xi) \to X_1(\xi) \cross Y(\pi)$.
Here $Y(\xi)$ is defined as in \Cref{e:Y-cones}. 

\begin{definition}
	\label{d:piece-of-lower-boundary}
	Suppose that $\zeta$ is a finite Rauzy--Veech sequence that ends at $\pi$.
	Suppose that $\eta$ is finite Rauzy--Veech sequence starting from $\pi$. 
	We define the $\calS_P(\zeta | \eta) = X_1(\eta) \cross Y(\zeta)$.
	We take $\calS^{(1)}_P(\zeta | \eta)$ to be the area-one locus inside of $\calS_P(\zeta | \eta)$.  
	Note that $\calS^{(1)}_P(\zeta | \eta)$ is a subset of $\bdy^- P^{(1)}(\pi)$.
	
	We finally define $\calS(\zeta | \eta)$ and $\calS^{(1)}(\zeta | \eta)$ to be  the images under $q_\pi$.
\end{definition}

In a slight abuse of notation we use $\pi$ to denote the Rauzy--Veech sequence of length zero starting (and ending) at $\pi$. 
So, for example, $\calS_P(\zeta | \pi) = X_1(\pi) \cross Y(\zeta)$.

Note that $\calS^{(1)}(\zeta | \eta)$ is the subset of differentials in $\bdy^- \calC^{(1)}(\pi)$ whose backwards and forwards Rauzy--Veech sequences begin with $\zeta$ and $\eta$ respectively.

We now fix a permutation $\pi$.
By \Cref{c:strong_nesting_by_positive_matrix}, we may also fix a directed loop $\theta$, based at $\pi$, which has a positive Rauzy matrix $E_\theta$ and is strongly nested for $P(\pi)$ (see \Cref{d:strong-nesting}). 
By extending $\theta$ even further (if necessary) we may also assume that $\theta$ is \emph{neat} in the sense of~\cite[Section 4.1.3]{Avi-Gou-Yoc}: 
that is, if $\theta = \zeta \eta$ and $\theta = \eta' \zeta$ then $\zeta = \theta$.
It follows that, in any Rauzy--Veech sequence, the occurrences of $\theta$ do not overlap. 

So, suppose that $\theta$ is strongly nested, positive, and neat. 
We use $\calS^{(1)}(\theta|\theta)$ as our Poincaré section.

Suppose that $s$ is a reindexing of $\calA$.
By \Cref{r:calC-lab}, the permutation $s$ defines an element of the deck group of the cover $\calD_\lab \to \calD_\root$.
Let $\theta_s$ be the translate of $\theta$ by $s$.
Since the cover is regular (\Cref{l:regular_cover}) $\theta_s$ is again a loop in $\calD_\lab$ based at $\pi_s = s \circ \pi$.  
Thus $\theta_s$ is again strongly nested, positive, and neat.

Note that all $\theta_s$ in $\calD_\lab$ project to $[\theta]$ in $\calD_\root$.
It follows that 
\[
q_{\pi_s} (\calS^{(1)}_P (\theta_s| \theta_s)) = \calS^{(1)}(\theta|\theta)
\]

It is often much more convenient to work in the parameter spaces rather than in $\calC_\root$. 
We use $\nu_P^{(1)}$ to denote the Rauzy--Veech invariant measure on $\bdy^- P^{(1)} = \bigsqcup_\pi \bdy^- P^{(1)}_\pi$ defined in \Cref{a:measures} (right after \Cref{r:coning}).
Suppose that $s$ is a reindexing.  
Then $s$ induces a map $s \from X(\pi) \cross Y(\pi) \to X(\pi_s) \cross Y(\pi_s)$ preserving all disintegrations discussed in \Cref{a:measures}.  
It follows that $\nu_P^{(1)}$ is invariant under these maps.
That is, for any measurable set $U \subseteq \bdy^- P^{(1)}_\pi$ we have
\begin{equation}
	\label{e:reindexings_have_equal_volumes}
	\nu_P^{(1)} (s(U)) = \nu_P^{(1)}(U)
\end{equation}
We now work with the section (in parameters) given by $\bigsqcup_s \,  \calS^{(1)}_P(\theta_s | \theta_s)$.
By \Cref{e:reindexings_have_equal_volumes} we have
\[
\nu_P^{(1)} (\calS^{(1)}_P(\theta_s | \theta_s)) = \nu_P^{(1)}(\calS^{(1)}_P(\theta | \theta))
\]
Applying this and \Cref{d:Veech_and_Rauzy--Veech}, we have
\[
\nu^{(1)}(\calS^{(1)}(\theta| \theta)) = \nu_P^{(1)}(\calS^{(1)}_P(\theta | \theta))
\]

\begin{lemma}
	\label{l:section-finite-volume}
	$\calS^{(1)}_P(\theta|\theta)$ has finite $\nu_P^{(1)}$--volume.
\end{lemma}

\begin{proof}
	Since $\calS^{(1)}_P(\theta|\theta)$ lies in $\bdy^- P^{(1)}(\pi)$ it suffices to prove that its $\nu_\pi^{(1)}$--volume is finite.
	By \Cref{s:projecting to widths}, the measure $\nu^{(1)}_\pi$ pushes forward under $p^\pi$ to the measure $\phi_{X_1(\pi)}$ on $X_1(\pi)$. 
	By \Cref{l:nu_X_one_finite}, the measure $\nu_{X_1(\pi)}$ on $X_1(\pi)$ is finite. 
	Also, by \Cref{e:density} we have
	\[
	\frac{\diff \phi_{X_1(\pi)}}{\diff \nu_{X_1(\pi)}} = \vol_\pi.
	\]
	for a smooth function $\vol_\pi \from X_1(\pi) \to \RR$.
	
	By strong nesting, the closure $\closure{X_1 (\theta)}$ is compact and lies in the interior of $X_1(\pi)$.
	Hence, $\vol_\pi$ restricted to $X_1(\theta)$ is bounded.
	Since $\calS^{(1)}_P(\theta|\theta)$ is a subset of $\bdy^- P^{(1)}(\theta)$, it follows that 
	\[
	\int_{\calS^{(1)}_P(\theta| \theta)} \diff \nu_\pi^{(1)} 
	< \int_{\bdy^- P^{(1)} (\theta)} \diff \nu_\pi^{(1)} 
	= \int_{X_1(\theta)} \vol_\pi(x)\, \diff \nu_{X_1(\pi)} (x) 
	< \infty
	\]
	Thus $\calS^{(1)}_P(\theta|\theta)$ has finite $\nu_\pi^{(1)}$--mass.
\end{proof}

\subsection{The pieces of the first return map}
\label{s:Pieces}

As usual (in order to have non-negative matrices) we work with the inverse of the first return map.
On $\bdy^- P$ the first return, by the diagonal flow, is given by 
the Rauzy--Veech renormalisation maps $\RV^\bdy$.
As noted in \Cref{t:all-ergodic} the measure $\nu^{(1)}_P$ is ergodic for $\RV^\bdy$ (restricted to $\bdy^- P^{(1)}$).
So for a $\nu^{(1)}_P$--typical point of $\bigsqcup_s \, \calS^{(1)}_P(\theta_s|\theta_s)$ the first return is a finite power of $\RV^\bdy$. 
The power required and the polytopes visited depend in piecewise continuous fashion on the given point.  
Now we analyse the pieces.

Let $\Lambda = \Lambda_\theta$ be the set of Rauzy--Veech sequences which 
\begin{itemize}
	\item start with $\theta_r$ for some reindexing $r$,
	\item end with $\theta_t$ for some reindexing $t$, and 
	\item do not contain $\theta_s^2$ for any reindexing $s$. 
\end{itemize}
In particular, $\theta_s$ (for any reindexing $s$) itself lies in $\Lambda$. 


\begin{lemma}
	\label{l:first-return}
	Suppose that $\xi$ is contained in $\Lambda$. 
	Suppose that $\xi$ is the concatenation $\xi' \xi''$ of non-empty Rauzy--Veech sequences. 
	Then $\calS_P(\xi' | \xi'')$ is disjoint from $\calS_P(\theta_s | \theta_s) $ for all reindexings $s$.
\end{lemma}

\begin{proof}
	For any reindexing $s$ the sequence $\theta_s$ is neat.
	Note that $\xi$ does not contain $\theta_s^2$. 
	The lemma follows.
\end{proof}

The next result follows from the definitions. See \Cref{f:domain-and-range}.
\begin{lemma}
	\label{l:domain-and-range}
	Suppose that $\xi$ is contained in $\Lambda$. 
	Suppose that $r$ and $t$ are reindexings so that $\xi$ starts with $\theta_r$ and ends with $\theta_t$.
	Then we have the following:
	\begin{itemize}
		\item $\calS_P(\theta_r \xi|\theta_t)$ is a subset of $\calS_P (\theta_t | \theta_t)$, 
		\item $\calS_P(\theta_r| \xi \theta_t)$ is a subset of $\calS_P(\theta_r | \theta_r)$ and
		\item $\RV^\bdy_\xi (\calS_P(\theta_r \xi|\theta_t))= \calS_P(\theta_r | \xi \theta_t )$.\qed
	\end{itemize}
\end{lemma}

\begin{figure}
	\centering
	\includegraphics{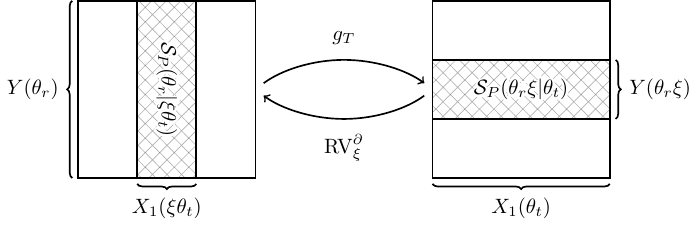}
	\caption{The domain and range of the piece $\RV^\bdy_\xi$.}
	\label{f:domain-and-range}
\end{figure}

By \Cref{l:first-return,l:domain-and-range}, we deduce that $\RV^\bdy_\xi \from \calS_P(\theta_r \xi|\theta_t) \to \calS_P(\theta_r | \xi \theta_t)$ is the inverse of a piece of the first return map to $\bigsqcup_s \, \calS_P(\theta_s | \theta_s)$. This also holds for the restriction to the area-one locus: namely, $\RV^\bdy_\xi \from \calS_P^{(1)}(\theta_r \xi|\theta_t) \to \calS_P^{(1)}(\theta_r |\xi\theta_t)$.

\begin{lemma}
	\label{l:first-returns-give-letters}
	Suppose that $(x, y)$ is a parameter in $\bigsqcup_s \, \calS_P(\theta_s | \theta_s)$.
	Suppose that there is a $T > 0$ so that $g_T(x, y)$ again lies in $\bigsqcup_s \, \calS_P(\theta_s | \theta_s)$.  
	Then there is a least such $T$.
	Furthermore, the flow segment $\gamma$ from $(x,y)$ to $g_T(x,y)$ is a $\xi$--segment for some $\xi$ in the alphabet $\Lambda$.
\end{lemma}

\begin{proof}
	By \Cref{l:base-arc-discrete-at-lower-boundary,l:finite-flow-faces}, the least positive $T$ such that $g_T(x,y)$ lies in $\bigsqcup_s \, \calS_P(\theta_s | \theta_s)$ exists.
	Since $\gamma$ returns to $\bigsqcup_s \, \calS_P(\theta_s | \theta_s)$, we have the following. 
	\begin{itemize}
		\item By \Cref{r:zeta-segment} the Rauzy--Veech sequence $\xi = \xi(\gamma)$ is well defined.
		\item By definition of the section, $\xi$ starts with some $\theta_r$ and ends with some $\theta_t$.
	\end{itemize} 
	Suppose that $\xi$ contains $\theta_s^2$. 
	We write $\xi$ as a finite concatenation $\xi_1 \xi_2 \dotsb \xi_r$ where each $\xi_i$ is in $\Lambda$.
	It follows that there exists a positive time $S < T$ such that
	\begin{itemize}
		\item the segment from $(x,y)$ to $g_S(x,y)$ is a $\xi_1$--segment, and
		\item $g_S(x,y)$ lies in $\bigsqcup_s \, \calS_P(\theta_s | \theta_s)$.
	\end{itemize}
	This contradicts the fact that $g_T(x,y)$ is the first return.
	It follows that $\xi$ is in $\Lambda$, as desired.
\end{proof}

With \Cref{l:first-returns-give-letters} in hand, the following is well-defined.

\begin{definition}
	\label{d:first-return-time}
	Suppose that $\xi$ is in $\Lambda$.
	Suppose that $(x, y)$ lies in $\calS_P(\theta_r | \xi \theta_t)$.
	Suppose that $\RV^\bdy_\xi(x', y') = (x, y)$, where $(x',y')$ lies in $\calS_P(\theta_r \xi| \theta_t)$.
	Then the first return of $g_T(x, y)$ to $\bigsqcup_s \, \calS_P(\theta_s | \theta_s)$ occurs at time
	\[
	\rho(x,y) = \log (w( E_\xi x'))
	\]
	Note that this is a $\nu_P$--measurable function on $\bigsqcup_s \, \calS_P(\theta_s | \theta_s)$.
\end{definition}

We now obtain the following.

\begin{lemma}
	\label{l:section-first-return-bounded-below}
	Suppose that $(x, y)$ is a point in $\bigsqcup_s \, \calS_P(\theta_s | \theta_s)$ which has a first return to $\bigsqcup_s \, \calS_P(\theta_s | \theta_s)$.
	Then $\rho(x, y) > \log 2$.
\end{lemma}


\begin{proof}
	Let $(x', y')$ be the first return of $g_t(x, y)$ to the section.
	By \Cref{l:first-returns-give-letters}, the flow segment from $(x,y)$ to $(x',y')$ is a $\xi$--segment for some $\xi$ in $\Pi$.
	Suppose that $E_\xi$ is its Rauzy matrix.
	Thus $x = (E_\xi x')/w(E_\xi x')$ and the return time is $\rho(x, y) = \log w(E_\xi  x')$.
	(Here we suppress the heights to simplify the notation.)
	Since $\xi$ is in $\Pi$, it can be written as a concatenation $\xi' \theta$, where $\xi'$ could be empty.
	Thus $E_\xi$ factors as $E_{\xi'} E_\theta$.
	Note that $E_{\xi'}$ is an integral, non-negative, invertible matrix.
	We now compute as follows.
	\[
	\log w(E_\xi x') = \log w(E_{\xi'} E_\theta  x') \geq  \log w(E_\theta x') >  \log |\calA| \geq \log 2 \qedhere
	\]
\end{proof}

By ergodicity of the map $\RV^\bdy$ (see \Cref{t:all-ergodic}), a $\nu^{(1)}_P$--typical point of $\bigsqcup_s \, \calS^{(1)}_P(\theta_s | \theta_s)$ returns to $\bigsqcup_s \, \calS^{(1)}_P(\theta_s | \theta_s)$ infinitely often (forwards and backwards). 
Hence, the disjoint union $\bigsqcup_{\xi \in \Lambda} \calS^{(1)}_P(\theta_r | \xi \theta_t)$ has full measure in $\bigsqcup_s \, \calS^{(1)}_P(\theta_s | \theta_s)$.

The measure $\nu^{(1)}_P$ is invariant under the Rauzy--Veech renormalisation (see \Cref{r:nu_P_invariant}).
Since each first return to $\bigsqcup_s \, \calS_P(\theta_s | \theta_s)$ is given by a finite Rauzy--Veech sequence, we conclude that the restriction of $\nu^{(1)}_P$ to $\bigsqcup_s \, \calS_P(\theta_s | \theta_s)$ is invariant under first returns.

\subsection{Distortion}
\label{s:distortion}

In order to control the measure distortion of the pieces $\RV^\bdy_\xi$ (for $\xi$ in $\Lambda$) we first consider the measure distortion of the Rauzy maps $\R^\bdy_\xi$.

We use the notation of \Cref{s:projecting to widths}. 
Let $\zeta$ be a finite Rauzy sequence that starts at $\pi$ and ends at $\pi'$.
By composing the renormalisation maps for each move in $\zeta$, we obtain the Rauzy renormalisation map $\R^\bdy_\zeta \from X_1(\pi') \to X_1(\zeta)$.
Again, we take $E_\zeta$ to be the associated Rauzy matrix.
Recall that $w_X \from X_1(\pi) \to \RR$ is the sum of the widths.
We deduce the following:
\[
\R^\bdy_\zeta (x') = \frac{E_\zeta x'}{w_X (E_\zeta x')}
\]

We use $\calJ (\R^\bdy_\zeta)$ to denote the Jacobian of $\R^\bdy_\zeta$; that is,
\[
\calJ(\R^\bdy_\zeta) (x') = 
\frac{\diff \nu_{X_1(\pi)}}{\diff [(\R^\bdy_\zeta)_* \nu_{X_1(\pi')}]} (\R^\bdy_\zeta (x'))
\]

\begin{definition}
	\label{d:loop-bounded-distortion}
	Let $K \geq 1$. 
	We say that a Rauzy--Veech sequence $\zeta$ from $\pi$ to $\pi'$ has \emph{$K$--bounded distortion} if 
	\[
	\frac{ \calJ (\R^\bdy_\zeta) (x) }
	{ \calJ (\R^\bdy_\zeta) (x')} \leq K
	\]
	for any pair of points $x, x' \in X_1(\pi')$.
\end{definition}

The next lemma says that probabilities, conditioned on a sequence having $K$--bounded distortion, are preserved (up to multiplication by $K^2$).

\begin{lemma}
	\label{l:relative-probability}
	Suppose that $\zeta$ is Rauzy--Veech sequence from $\pi$ to $\pi'$.
	Suppose that $\zeta$ has $K$--bounded distortion. 
	Suppose that $U'$ is any measurable subset of $X_1(\pi')$. 
	Then, taking $\nu = \nu_{X_1(\pi)}$ and $\nu' = \nu_{X_1(\pi')}$, we have
	\[
	\frac{1}{K^2} \cdot \frac{\nu'(U')}{\nu'(X_1(\pi'))} 
	\leq \frac{\nu( \R^\bdy_\zeta(U') )}{\nu (X_1(\zeta))} 
	\leq K^2 \cdot \frac{\nu'(U')}{\nu'(X_1(\pi'))}
	\]
\end{lemma}

\begin{proof}
	We fix a point $x_0$ in $X_1(\pi')$. 
	Since $\zeta$ has $K$--bounded distortion, it follows that for any $x \in X_1(\pi')$, 
	\[
	\frac{1}{K} \calJ (\R^\bdy_\zeta) (x_0) \leq \calJ (\R^\bdy_\zeta) (x) \leq K \calJ (\R^\bdy_\zeta) (x_0)
	\]
	
	Suppose that $V'$ is any measurable subset of $X_1(\pi')$ with positive measure.
	From the definitions and by the change of coordinates formula we have the following:
	\begin{align*}
		\nu (\R^\bdy_\zeta(V')) 
		&= \int_{\R^\bdy_\zeta(V')} \diff \nu \\
		&= \int_{\R^\bdy_\zeta(V')} \calJ(R^\bdy_\zeta) \diff (\R^\bdy_\zeta)_* \nu \\
		&= \int_{V'} \calJ (\R^\bdy_\zeta) \diff \nu'
	\end{align*}
	
	By \Cref{l:nu_X_one_finite}, all the integrals above are finite. 
	This and the preceding bounds on the Jacobian give the following:
	\[
	\frac{1}{K} \calJ(\R^\bdy_\zeta)(x_0)  
	\leq \frac{\nu (\R^\bdy_\zeta(V'))}{\nu' (V')} 
	\leq K \calJ(\R^\bdy_\zeta)(x_0) 
	\]
	
	We now suppose (as we may) that $U'$ has positive measure.
	We apply the above estimate twice: 
	once with $V' = U'$ and once with $V' = X_1(\pi')$ (noting that $\RV^\bdy_\zeta (X_1(\pi')) = X_1 (\zeta)$) to get
	\[
	\frac{1}{K} \calJ(\R^\bdy_\zeta)(x_0)  
	\leq \frac{\nu (\R^\bdy_\zeta(U'))}{\nu' (U')} 
	\leq K \calJ(\R^\bdy_\zeta)(x_0) 
	\]
	and
	\[
	\frac{1}{K} \calJ(\R^\bdy_\zeta)(x_0)  
	\leq \frac{\nu (X_1(\zeta))}{\nu' (X_1(\pi'))} 
	\leq K \calJ(\R^\bdy_\zeta)(x_0) 
	\]
	The conclusion follows.
\end{proof}

\begin{corollary}
	\label{c:relative-probability}
	Suppose that $\zeta$ is a Rauzy--Veech sequence from $\pi$ to $\pi'$ and $\zeta'$ a finite Rauzy--Veech sequence that starts from $\pi'$.
	Suppose that $\zeta$ has $K$--bounded distortion. Then, taking $\nu = \nu_{X_1(\pi)}$ and $\nu' = \nu_{X_1(\pi')}$, 
	we get
	\[
	\frac{1}{K^2} \cdot \frac{\nu' (X_1(\zeta'))}{\nu'(X_1(\pi'))}
	\leq \frac{\nu (X_1(\zeta \zeta'))}{\nu (X_1 (\zeta))}
	\leq K^2 \cdot \frac{\nu' (X_1(\zeta'))}{\nu'(X_1(\pi'))}
	\qedhere
	\]
\end{corollary}

\begin{proof}
	Take $U' = X_1(\zeta')$ and thus find $\R^\bdy_\zeta (U') = X_1(\zeta \zeta')$.  
	The corollary follows directly by \Cref{l:relative-probability}.
\end{proof}

Suppose that $\zeta$ is a Rauzy--Veech sequence from $\pi$ to $\pi'$ with Rauzy matrix $E_\zeta$.
In the abelian case, by \cite[Proposition 5.2]{Vee78}, we have for $x \in X_1(\pi')$ that
\begin{equation}
	\label{e:Jacobian-abelian}
	\calJ(\R^\bdy_\zeta) (x) = \frac{1}{w_X(E_\zeta x)^{|\calA|}}
\end{equation}
In the quadratic case, by \cite[Equation 8.3 and Lemma 8.1]{Gad}, there exists a constant $c_\zeta> 0$ such that 
\begin{equation}
	\label{e:Jacobian-quadratic}
	\calJ(\R^\bdy_\zeta) (x) = \frac{c_\zeta}{ w_X( E_\zeta x )^{|\calA|-1}}
\end{equation}
The appearance of the constant $c_\zeta$ arises from the width equality in \eqref{e:e.x-y}.

Suppose that $\zeta$ is a Rauzy--Veech sequence from $\pi$ to $\pi'$.
Suppose that $\alpha$ is a label in $\calA$. 
We define $\col_\alpha(E_\zeta)$ to be the $\alpha$--column of $E_\zeta$.
Suppose that $k \geq 1$ is a constant. 
The matrix $E_\zeta$ is said to be \emph{$k$--balanced} if 
\[
\frac{\|\col_\alpha(E_\zeta)\|_1} 
{\|\col_{\alpha'}(E_{\zeta})\|_1} \leq k
\]
for any $\alpha$ and $\alpha'$ in $\calA$.

\begin{lemma}
	\label{l:balanced_implies_bounded_distortion}  
	Suppose that $\zeta$ is a Rauzy--Veech sequence from $\pi$ to $\pi'$. 
	Suppose that the Rauzy matrix $E_\zeta$ is $k$--balanced.
	Then $\zeta$ has $K$--bounded distortion for $K = k^{|\calA|}$ in the abelian case and for $K = k^{|\calA|-1}$ in the quadratic case.
\end{lemma}

\begin{proof}
	Since $E_\zeta$ is non-negative, for every $x \in X_1(\pi')$ we have the following
	\[
	w_X(E_\zeta x) = \sum_{\alpha \in \calA} x_\alpha \| \col_\alpha(E_\zeta)\|_1
	\]
	Since $w_X(x) = 1$ for all $x \in X_1(\pi')$, it follows that
	\[
	\min_{\alpha \in \calA} \|\col_\alpha(E_\zeta)\|_1 \leq w_X(E_\zeta x) \leq \max_{\alpha \in \calA} \|\col_\alpha(E_\zeta)\|_1
	\]
	Therefore,
	\[
	\frac{w_X(E_\zeta x)}{w_X(E_\zeta x')} \leq k
	\]
	for every $x, x' \in X_1(\pi')$.
	
	In either case (abelian or quadratic), the lemma follows from \Cref{e:Jacobian-abelian} or \Cref{e:Jacobian-quadratic}, with $K = k^{|\calA|}$ or $K = k^{|\calA|-1}$, respectively.
\end{proof}

We now use our (strongly nested, positive, neat) sequence $\theta$. 

\begin{lemma}
	\label{l:balanced_when_ends_theta}    
	Suppose that $\zeta \theta_s$ (for some reindexing $s$) is a finite Rauzy--Veech sequence. Then there is $k_\theta \geq 1$, depending only on $\theta$, such that $E_{\zeta \theta_s}$ is $k_\theta$--balanced.
\end{lemma}

\begin{proof}
	Suppose that $M$ is a non-negative matrix with rows and columns indexed by $\calA$.
	We define $\|M\|_1 = \max_\alpha \|\col_\alpha(M)\|_1$.
	We denote the $(\alpha, \beta)$ entry of $M$ by $M(\alpha, \beta)$.
	We define $M_{\min} = \min_{\alpha, \beta} M (\alpha, \beta)$.
	
	Note that for any reindexing $s$, we have $\|E_{\theta_s} \|_1 = \| E_\theta \|_1$ and $(E_{\theta_s})_{\min}  = (E_\theta)_{\min}$.
	Since $E_\theta$ is a positive integer matrix, we have $(E_\theta)_{\min} \geq 1$.
	
	Now note that $E_{\zeta\theta_s} = E_\zeta E_{\theta_s}$, so
	\[
	\col_\beta(E_{\zeta \theta_s}) = \sum_{\alpha \in \calA} E_{\theta_s}(\alpha, \beta) \col_{\alpha}(E_\zeta)
	\]
	We deduce that, for any $\beta \in \calA$,
	\[
	\|E_\zeta\|_1  
	\leq (E_{\theta_s})_{\min} \cdot \|E_\zeta\|_1  
	\leq \| \col_\beta(E_{\zeta \theta_s}) \|_1 
	\leq \| E_{\theta_s} \|_1 \cdot \|E_\zeta\|_1
	\]
	It follows that $E_{\zeta \theta_s}$ is $k_\theta$--balanced for $k_\theta = \| E_\theta \|_1$.
\end{proof}

From this we deduce the following. 

\begin{corollary}
	\label{c:words_balanced}
	Suppose that $u = \xi_1 \cdots \xi_r$ is a finite Rauzy--Veech sequence obtained by concatenating $\xi_j$, all in $\Lambda$.
	Then there is $k_\theta \geq 1$, depending only on $\theta$, such that $E_u$ is $k_\theta$--balanced.    
\end{corollary}

\begin{proof}
	Note that $\xi_r$ ends with $\theta_t$ for some reindexing $t$.
	So we may write $u$ as $\zeta \theta_t$, for a (possibly empty) Rauzy sequence $\zeta$.
	The corollary follows from \Cref{l:balanced_when_ends_theta}.
\end{proof}

From this and \Cref{l:balanced_implies_bounded_distortion} we have the following.

\begin{corollary} \label{c:words_bounded_distortion}
	Suppose that $u = \xi_1 \cdots \xi_r$ where all $\xi_j$ are in $\Lambda$.
	Then, with $K$ as in \Cref{l:balanced_implies_bounded_distortion}, the sequence $u$ has $K$--bounded distortion. \qed
\end{corollary}

We now consider our section $\bigsqcup_s \, \calS^{(1)}_P(\theta_s|\theta_s)$. 
Recall that $\calS^{(1)}_P (\theta_s | \pi_s)$ is the area-one locus in $X_1(\pi_s) \cross Y(\theta_s)$. 
Let $Y(\theta_s, x)$ denote the preimage of $x \in X_1(\pi_s)$ under the (restriction of) projection to the first factor.
Note that $Y(\theta_s, x)$ is a subset of $Y(\pi_s, x)$,
as defined in \Cref{s:projecting to widths}. 
By \Cref{l:fibres-finite-vol}, we have
\[
\nu_{Y(\pi_s,x)} (Y(\theta_s, x)) < \nu_{Y(\pi_s,x)} (Y(\pi_s, x)) < \infty
\]
We define the function $\vol_{\theta_s} \from X_1(\pi_s) \to \RR$ by $\vol_{\theta_s} (x) = \nu_{Y(\pi_s,x)} (Y(\theta_s, x))$.
Note that $\vol_{\theta_s}$ is smooth.

\begin{lemma} 
	\label{l:bounds_vol_theta}
	There exist constants $L_\theta, U_\theta > 0$ depending only on $\theta$ 
	so that, for any reindexing $s$ and any $x \in X_1(\theta_s)$, we have
	\[
	L_\theta \leq \vol_{\theta_s}(x) \leq U_\theta
	\]
\end{lemma}

\begin{proof}
	Observe that, by \Cref{e:reindexings_have_equal_volumes}, we have $\vol_{\theta_s}(s(x)) = \vol_\theta(x)$ for every $x \in X_1(\pi)$.
	Note that $\vol_{\theta} (x) > 0$ for any $x \in X_1(\pi)$. 
	The lemma follows from the fact that $\closure{X_1(\theta)}$ is compact and contained in the interior of $X_1(\pi)$.
\end{proof}

We now consider the area-one locus in $X_1(\pi_s) \cross Y(\pi_s)$ with the measure $\nu^{(1)}_P$.
As in \Cref{s:projecting to widths} (immediately before \Cref{e:density}), the measure pushes forward (under projection to widths) to the measure $\phi_{X_1(\pi_s)}$.
By \Cref{e:density}, the measure $\phi_{X_1(\pi_s)}$ has density $\vol_{\pi_s}(x)$ with respect to $\nu_{X_1(\pi_s)}$.
It follows that when restricted to $\calS^{(1)}_P(\theta_s|\pi_s)$ the measure $\nu^{(1)}_P$ pushes forward to the measure $\vol_{\theta_s}(x) \cdot \nu_{X_1(\pi_s)}$. 

\begin{lemma}
	\label{l:measure-comparison}
	Suppose that $\zeta$ is a finite Rauzy--Veech sequence that starts at $\pi_s$. 
	Then we have:
	\[
	L_\theta \cdot  \nu_{X_1(\pi_s)} (X_1(\zeta)) \leq \nu^{(1)}_P (\calS^{(1)}_P(\theta_s| \zeta)) \leq U_\theta \cdot \nu_{X_1(\pi_s)} (X_1(\zeta))
	\]
\end{lemma}

\begin{proof}
	Note that $\calS^{(1)}_P (\theta_s | \zeta)$ is a subset of $\calS^{(1)}_P(\theta_s|\pi_s)$. 
	By definition of the pushforward measure we have the following:
	\[
	\int_{\calS^{(1)}_P (\theta_s | \zeta)} \diff \nu^{(1)}_P = \int_{X_1(\zeta)} \vol_{\theta_s} \diff \nu_{X_1(\pi_s)}
	\]
	The conclusion now follows from the bounds on $\vol_{\theta_s}$ in \Cref{l:bounds_vol_theta}.
\end{proof}

\subsection{The alphabet for the countable shift}
\label{s:countable_shift}
Recall the set $\Lambda = \Lambda_\theta$ defined at the beginning of \Cref{s:Pieces}.
Reindexing defines an equivalence relation on $\Lambda$. 
We denote the set of equivalence classes as $\Pi$.
Each equivalence class contains a unique representative that begins with $\theta$.
Hence, we may identify $\Pi$ with the set of Rauzy--Veech sequences which 
\begin{itemize}
	\item start with $\theta$,
	\item end with $\theta_t$ for some reindexing $t$, and
	\item do not contain $\theta^2_s$ for any reindexing $s$.
\end{itemize}
Suppose that $\xi$ and $\xi'$ are elements of $\Pi$. 
Suppose that $\xi$ ends with $\theta_t$. 
Let $\xi_t'$ be the reindexing of $\xi'$ by $t$. 
Then we define the concatenation $\xi \xi' \in \Pi^2$ to be the Rauzy-Veech sequence $\xi \xi'_t$.
We now recursively extend this to obtain finite words over $\Pi$.

As a piece of simplifying notation we define 
\[
\nu^{(1)}_P(\eta|\zeta) = \nu^{(1)}_P (\calS^{(1)}_P(\eta|\zeta))
\]

\begin{lemma}
	\label{l:almost-bernoulli}
	There exists a constant $M \geq 1$, depending only on $\theta$, with the following property. 
	Suppose that $u$ and $v$ are words in $\Pi^i$ and  $\Pi^j$ respectively. 
	Then we have the following:
	\[
	\frac{1}{M} 
	\leq \frac{\nu_P^{(1)}(\theta| uv )}{\nu^{(1)}_P(\theta|u) \cdot \nu_P^{(1)}(\theta|v)}
	\leq M
	\]
\end{lemma}

\begin{proof}
	We take $\nu = \nu_{X_1(\pi)}$. 
	Suppose that $u$ ends with $\theta_t$. 
	The concatenation $uv$ defines the Rauzy--Veech sequence $u v_t$.
	By \Cref{l:measure-comparison} applied to $u v_t$ we have:
	\[
	L_\theta \cdot \nu (X_1(u v_t)) \leq \nu_P^{(1)} (\theta| u v_t)  \leq U_\theta \cdot \nu (X_1(u v_t))
	\]
	We take $\nu_t = \nu_{X_1(\pi_t)}$.
	By \Cref{c:words_bounded_distortion,c:relative-probability} applied to the concatenated sequence $u v_t$ we get
	\[
	\frac{1}{m} \cdot  \nu(X_1(u)) \cdot \nu_t(X_1(v)) \leq \nu (X_1(u v_t)) \leq m \cdot \nu(X_1(u)) \cdot \nu_t(X_1(v)) 
	\]
	By \Cref{l:measure-comparison}, we have
	\[
	\frac{1}{U_\theta} \cdot \nu^{(1)}_P (\theta| u) \leq \nu(X_1(u))
	\]
	Similarly for $v_t$, we have
	\[
	\frac{1}{U_\theta} \cdot \nu^{(1)}_P(\theta_t| v_t) \leq \nu_t (X_1(v_t))
	\]
	Similarly, we have
	\[
	\nu(X_1(u))  \leq \frac{1}{L_\theta} \cdot \nu^{(1)}_P(\theta| u)
	\]
	and
	\[
	\nu_t (X_1(v_t))  \leq \frac{1}{L_\theta} \cdot \nu^{(1)}_P(\theta_t| v_t)
	\]
	Combining the above estimates, we get
	\[
	\frac{L_\theta}{m U_\theta^2} \cdot \nu^{(1)}_P (\theta| u) \cdot \nu^{(1)}_P (\theta| v) \leq \nu^{(1)}_P (\theta | u v_t) \leq  \frac{m U_\theta}{L_\theta^2} \cdot \nu^{(1)}_P (\theta| u) \cdot \nu^{(1)}_P (\theta| v)
	\]
	which concludes the proof by taking $M = \max\left\{ \frac{m U_\theta^2}{L_\theta}, \frac{m U_\theta}{L_\theta^2} \right\} \geq 1$.
\end{proof}

\subsection{Contraction}

The discussion above, restricted to abelian components, reproduces the work of~\cite[Sections~4.1 and~4.2]{Avi-Gou-Yoc}.
In this section and the next we give a (partial) dictionary between our language and theirs.

Recall (from \Cref{s:polytopal-measures-full}, immediately after \Cref{d:cones}) that the width and height cones do not contain a linear subspace. 
It follows that their Hilbert metric is well-defined. 
The pieces of the first return to $\bigsqcup_s \, \calS^{(1)}_P(\theta_s|\theta_s)$ depend only on the widths.
So it suffices to consider the Hilbert metric on $X_1(\pi_s)$ for each $\pi_s$.
For definitions and background, see~\cite[Section 4.2.1]{Avi-Gou-Yoc}.

By positivity of $E_{\theta_s}$, there is a positive constant $c_\theta < 1$ (that depends only on $\theta$) such that in the Hilbert metric on $X_1(\pi_s)$ we have that
\begin{itemize}
	\item the cone $X_1(\theta_s)$ has diameter bounded above by $c_\theta$ and
	\item for any $u$ in $\Pi^r$ (that ends in say $\theta_t$), the cone $X_1(u\theta_t)$ has diameter bounded above by $c_\theta^{r+1}$.  
\end{itemize}
We call the diameter bounds the \emph{uniform contraction property}.


\subsection{Summary of the coding}

We summarise our coding for the diagonal flow.
Again, for notational convenience, we work in parameter spaces. 
\begin{itemize}
	\item 
	Our section $\bigsqcup_s \, \calS^{(1)}_P(\theta_s|\theta_s)$ admits a (full measure) countable partition.  
	The pieces of this partition are the sets $\calS^{(1)}_P(\theta_r|\xi\theta_t)$ where $\xi$ in $\Lambda$ starts with $\theta_r$ and ends with $\theta_t$.
	\item The inverse of the map $\RV^\bdy_\xi$ restricted to $\calS^{(1)}_P(\theta_r|\xi \theta_t)$ gives a piece of the first return to $\bigsqcup_s \, \calS^{(1)}_P(\theta_s| \theta_s)$ (with image $\calS^{(1)}_P(\theta_r\xi |\theta_t)$).
	\item By the uniform contraction property, the pieces of first return give a \emph{uniformly expanding Markov map} (in the sense of \cite[Definition 2.2]{Avi-Gou-Yoc}).
	\item The maps $\RV^\bdy_\xi$ (where $\xi$ ranges over $\Lambda$) give a \emph{hyperbolic skew-product} (in the sense of \cite[Definition 2.5]{Avi-Gou-Yoc}) over the expanding Markov map. 
\end{itemize}

Taking the image under $q_P = \bigsqcup_\pi q_\pi$, we get the first return map for our section in differentials.  
We denote this by $F_\theta \from \calS^{(1)}(\theta|\theta) \to \calS^{(1)}(\theta|\theta)$. 
By \Cref{l:first-return}, \Cref{l:domain-and-range} and \Cref{l:first-returns-give-letters}, the pieces of $F_\theta$ are in bijection with $\Pi$.
(Note that this is only defined up to sets with $\nu^{(1)}$ measure zero).
Recall that $\Sigma = \Pi^{\ZZ}$ and $S$ is the shift on $\Sigma$.
Through the partition of $\calS^{(1)}(\theta| \theta)$, we obtain a conjugacy from $(\Sigma, \shift)$ to $(\calS^{(1)}(\theta| \theta), F_\theta)$.

\begin{remark}
	\label{r:dictionary}
	We pushforward (by the inverse of the bijection) the measure $\nu^{(1)}$ on $\calS^{(1)}(\theta|\theta)$ to obtain a measure $\mu$ on $\Sigma$.
	By \Cref{d:Veech_and_Rauzy--Veech} and \Cref{l:almost-bernoulli}, we deduce that $\mu$ has bounded distortion.
\end{remark}

\subsection{The first return time function} 
\label{s:first-return-time}

We now restrict to $\calS_P^{(1)}(\theta| \theta)$ (instead of $\bigsqcup_s \calS_P^{(1)} (\theta_s | \theta_s)$.
We now define a simpler alphabet. 
Let $\Pi'$ be the set of finite Rauzy--Veech sequences which 
\begin{itemize}
	\item 
	start with $\theta$,
	\item 
	end with $\theta$, and
	\item 
	do not contain $\theta^2$.
\end{itemize}
Following \Cref{l:first-return}, \Cref{l:domain-and-range}, and \Cref{l:first-returns-give-letters}, the pieces of the first return to $\calS_P^{(1)}(\theta| \theta)$ are given by the alphabet $\Pi'$.
By \Cref{c:words_bounded_distortion}, any sequence $\eta$ that lies in $\Pi'$ has $K$--bounded distortion where $K$ depends only on $\theta$.
Suppose that $(x,y)$ lies in $\calS^{(1)}_P (\theta | \eta \theta)$. 
Suppose that $\RV^\bdy_\eta (x',y') = (x,y)$. 
The return time for $(x,y)$, under diagonal flow, is given by
\begin{equation}
	\label{e:restricted_return_time}
	\rho' (x,y) = \log (w(E_\eta x'))
\end{equation}

Suppose that $\xi$ lies in $\Pi$. Suppose that $\xi$ ends with $\theta_t$ for some reindexing $t$.
We define $\Pi'_\xi$ to be those $\eta$ in $\Pi'$ that start with $\xi$.
The union $\bigsqcup_{\eta \in \Pi'_\xi} \calS^{(1)}_P (\theta | \eta \theta)$ has full $\nu^{(1)}_P$--measure in $\calS^{(1)}_P (\theta | \xi \theta_t)$.

Recall (from \Cref{d:first-return-time}) that we use $\rho$ to denote the first return time to $\bigsqcup_s \calS^{(1)}_P (\theta_s | \theta_s)$.
Suppose that $(x,y)$ lies in some $\calS^{(1)}_P (\theta | \eta \theta)$ for some $\eta$ in $\Pi'_\xi$.
It follows that $\rho(x, y) \leq \rho'(x,y)$.
We deduce that 
\[
\int_{\calS^{(1)}_P (\theta| \xi \theta) } \rho (x,y) \, d \nu^{(1)}_P \leq \sum_{\eta \in \Pi'_\xi} \int_{\calS^{(1)}_P (\theta | \eta \theta)} \rho'(x,y) \, d \nu^{(1)}_P
\]
So, to prove $\nu^{(1)}_P$--integrability of $\rho$, it suffices to prove $\nu^{(1)}_P$--integrability of $\rho'$. 
We do this in \Cref{s:measure-estimates} and \Cref{s:first-return-time-integrable}.

\subsection{The key measure estimate}
\label{s:measure-estimates}

We state the key estimate needed to show $\nu^{(1)}_P$--integrability of $\rho'$.
From this we will deduce the $\log$--integrability of the discrete cocycles in \Cref{s:integrability_discrete}.

We organise $\eta$ in $\Pi'$ by the size of $\| E_\eta \|_1$, on a logarithmic scale, as follows.
Given $S >1$ and $k \geq 1$, we denote
\[
\Pi'_k(S) = \{ \eta \in \Pi' \st  S^{k-1} \leqslant \| E_\eta \|_1 < S^k \}
\]

Applying the recurrence estimates of~\cite[Propositions~10.21 and~10.33]{Gad}, there exist constants $S > 1$, $C > 0$, and $0 < c < 1$ that depend only on $\theta$ such that, for any $k \geq 1$, we have the following:
\[    
\sum_{\eta \in \Pi'_k(S)} \nu_{X_1(\pi)} (X_1(\eta)) < C \cdot c^k
\]

By \Cref{l:bounds_vol_theta}, the density $\vol_\theta$ is bounded above and below on $X_1(\theta)$.
So, by \Cref{l:measure-comparison}, 
a similar estimate holds for $\nu^{(1)}_P$; namely
\begin{equation}
	\label{e:main-estimate}
	\sum_{\eta \in \Pi'_k(S)} \nu_P^{(1)} (\theta| \eta \theta) < U_\theta C \cdot c^k   
\end{equation}

\subsection{The first-return time is integrable} 
\label{s:first-return-time-integrable}

\begin{definition}
	Suppose that $(\Sigma, \mu)$ is a measure space. 
	Suppose that $\rho \from \Sigma \to \RR$ is non-negative and $\mu$--measurable. 
	We say that $\rho$ has \emph{exponential tails} if there exists $h > 0$ such that
	\[
	\int_\Sigma e^{h \rho}\, \diff \mu  < \infty \qedhere
	\]
\end{definition}

\begin{lemma}
	\label{l:first-return-time-integrable}
	The first return time $\rho'$ has exponential tails over $( \calS^{(1)}_P (\theta| \theta), \nu^{(1)}_P)$.
	In particular, $\rho'$ is integrable over $( \calS^{(1)}_P (\theta| \theta), \nu^{(1)}_P)$.
\end{lemma}

\begin{proof}
	
	Let $U_\theta$, $S$, $C$, $c$ be the constants in \Cref{e:main-estimate}.
	Suppose that $\eta$ is in $\Pi'_k(S)$. 
	Suppose that $(x, y)$ lies in $\calS_P^{(1)}(\theta|\eta\theta)$.
	Let $(x', y') \in \calS_P^{(1)}(\theta\eta|\theta)$ be the first-return of $(x, y)$.
	Then
	\[
	\rho'(x,y) =  \log (w( E_\eta x')) \leq \log \| E_\eta \|_1 < \log S^k = k \log S
	\]
	Given $h > 0$, the recurrence estimate in \Cref{e:main-estimate} shows that 
	\[
	\int_{\calS^{(1)}_P(\theta| \theta)} e^{h \rho'} \, \diff \nu^{(1)}_P = \sum_{k = 1}^\infty \, \sum_{\eta \in \Pi'_k(S)} \int_{\calS^{(1)}_P(\theta| \eta\theta)} e^{h \rho'} \, \diff \nu^{(1)}_P < \sum_{k=1}^\infty U_\theta C (S^{ h} c)^k
	\]
	The series on the right converges exactly when $S^h c < 1$.
\end{proof}

\section{Cocycles} 

\subsection{Hodge bundles}
\label{s:Hodge-bundle}

Suppose that $\calC$ is a stratum component.
Fix a differential $q$ in $\calC$; 
let $Z(q)$ be the set of singularities.
Fix also a surface $S$, a finite set $Z(S) \subset S$, and a homeomorphism $f_q \from (S, Z(S)) \to (q, Z(q))$.
Let $\calM = \calM(S, Z(S))$ be the moduli space of Riemann surfaces marked at $|Z(S)|$--many unlabelled points. 
Let $\calT = \calT(S, Z(S))$ be the corresponding Teichmüller space:
that is, points of $\calT$ are (equivalence classes of) Riemann surfaces $(X, Z(X))$ equipped with a marking homeomorphism $f_X \from (S, Z(S)) \to (X, Z(X))$.
Recall that $\calT$ is the universal (orbifold) cover of $\calM$.  
We use $\Mod(S, Z(S))$ to denote the deck group (acting on the right).
Thought of as a mapping class group, $\Mod(S, Z(S))$ acts transitively on the points of $Z(S)$.

The \emph{(relative) Hodge bundle} $\cover{\upH}$ over $\calT$ is the vector bundle which, at a point $[X, Z(X), f_X] \in \calT$, has fibre $\upH^1(X, Z(X); \RR)$. 
The markings give a trivialisation of $\cover{\upH}$:
namely, at $[X, Z(X), f_X] \in \calT$, the map $f_X \colon (S, Z(S)) \to (X, Z(X))$ induces an isomorphism 
\[
f_X^* \from \upH^1(X, Z(X); \RR) \to \upH^1(S, Z(S); \RR)
\]  
This, in turn, induces an isomorphism of vector bundles $\cover{\upH} \to \calT \times \upH^1(S, Z(S); \RR)$.
Note that $\Mod(S, Z(S))$ acts on the trivial bundle (again on the right), and thus also on $\cover{\upH}$. 
The quotient $\upH$ is an (orbifold) vector bundle over $\calM$.
We call $\upH$ the \emph{(relative) Hodge bundle} over $\calM$.

Suppose now that $\calL$ is an orbifold equipped with a map to $\calM$. 
We pullback $\upH$ to obtain an orbifold vector bundle $\upH_\calL$. 
In a small abuse of notation, we take $\upH_\calC$ to be the orbifold vector bundle obtained when $\calL = \calC_\root$.
We call $\upH_{\calC}$ the \emph{(relative) Hodge bundle} over $\calC_\root$.

\begin{lemma}
	\label{l:VectorBundle}
	The orbifold vector bundle $\upH_\calC$ is a vector bundle.
\end{lemma}

\begin{proof}
	We trace through the construction. 
	Let $\cover{\calC}_\root$ be the cover of $\calC_\root$ obtained by equipping differentials $((q, v), Z(q))$ with marking homeomorphisms $f_q$ and then taking equivalence classes.
	Thus there is a natural map from $\cover{\calC}_\root$ to $\calT$ that takes a marked differential $[(q, v), Z(q), f_q]$ to a marked Riemann surface $[X, Z(X), f_X]$.
	We form the \emph{(relative) Hodge bundle} $\cover{\upH}_\calC$ over $\cover{\calC}_\root$ by pulling back $\cover{\upH} \to \calT$. 
	Thus the fibre over $[(q, v), Z(q), f_q]$ is $\upH^1(X, Z(X); \RR)$.
	As before, we trivialise using the markings, and then quotient by the action of the mapping class group $\Mod(S, Z(S))$.
	Finally, the mapping class group acts without fixed points on $\cover{\calC}_\root$.  
	So $\upH_\calC$ is a vector bundle, as desired.
\end{proof}

\subsection{Orientation covers}
\label{s:orientation-covers}

Suppose that $\calC$ is a stratum component of quadratic differentials.
Suppose that $q$ lies in $\calC$. 
Then $q$ has a canonical \emph{orientation cover}, say $p$.
(The cover $p$ is branched over every pole, and over every zero of odd order, of $q$.)
We take $Z(p)$ to be the full preimage of $Z(q)$. 
(Note that this differs from the discussion in~\cite[Section~7.4.2]{Avi-Res}.)
Note that $p$ is an abelian differential.
Let $\iota_p$ be the deck transformation of the cover $p \to q$; 
note that $\iota_p$ is an involution.
We use $\pm \omega_p$ to denote the two cohomology classes whose periods give $p$.
For further details, see~\cite[Construction, page 519]{Lan04b}.

Fix a surface $R$, a finite set $Z(R) \subset R$, and a homeomorphism $f_p \from (R, Z(R)) \to (p, Z(p))$.
Let $\calT' = \calT(R, Z(R))$ be the corresponding Teichmüller space and let $\calM' = \calM(R, Z(R))$.
Let $\iota \from R \to R$ be the pullback of the involution $\iota_p$.
Note that the isotopy class of $\iota$ is independent of $p$; 
thus $\iota$ gives a well-defined mapping class $[\iota]$ in $\Mod(R, Z(R))$.

We now define $\calB$ to be the stratum component of abelian differentials that contains $p$.
It is unclear how to realise $\calC$ as (an isomorphic copy of) a locus in $\calB$.  
To address this we resort to \emph{bi-rooted} differentials. 

We define $\calB_\biroot$ to be the following stratum of differentials:
a point $(p, u, u')$ of $\calB_\biroot$ is an abelian differential $p$ in $\calB$ equipped with (distinct, unlabelled) roots $u$ and $u'$.
As usual the roots point along some horizontal separatrix. 
Furthermore, we require that $u$ points in the positive horizontal direction if and only if $u'$ points in the negative horizontal direction.

\begin{lemma}
	\label{l:Manifold}
	The stratum $\calB_\biroot$ is a manifold.
\end{lemma}

\begin{proof}
	Orbifold points in strata of differentials only arise from translation symmetries.
\end{proof}


The map from $\calC_\root$ to $\calB_\biroot$, taking $(q, v)$ to its orientation cover, is an isomorphism onto its image: see~\cite[Section~7.4.3]{Avi-Res}.
We use $\calC'_\root$ to denote the image.

Let $\upH'$ be the (relative) Hodge bundle over $\calM'$.
Note that $\calC'_\root$ has a forgetful map to $\calM'$.  
Let $\upH'_\calC$ be the pullback of $\upH'$.  
We again call $\upH'_\calC$ the \emph{(relative) Hodge bundle} over $\calC'_\root$.
As in \Cref{l:VectorBundle}, we have the following.

\begin{lemma}
	\label{l:VectorBundleTwo}
	The bundle $\upH'_\calC$ is a vector bundle. \qed
\end{lemma}

Recall that the fibre of $\upH'_\calC$ over $(p, u, u')$ is $\upH^1(p, Z(p); \RR)$ which is isomorphic to $\upH^1(R, Z(R); \RR)$.  
Since $\iota_p$ is an involution, so is $\iota_p^*$ acting on cohomology.
It thus splits each fibre of $\upH'_\calC$ into $+1$ and $-1$ eigenspaces: we call these the \emph{plus} and \emph{minus fibres}.
Since the maps $\iota_p$, for $p \in \calC'_\root$, induce a single isotopy class $[\iota]$ in $\Mod(R, Z(R))$, the plus and minus fibres induce a splitting of the vector bundle $\upH'_\calC$.
We define the \emph{plus piece} $\upH'_+$ and the \emph{minus piece} $\upH'_-$ to be the vector bundles induced by the plus and minus fibres respectively.

\subsection{Symplectic cocycles}
\label{s:symplectic_cocycles}
Suppose that $\calC$ is a space with a flow $g_t$.
A \emph{symplectic cocycle} for $g_t$ 
(acting on $\RR^{2m}$) is a map $\upC \from \RR \times \calC \mapsto \Sp(2m, \RR)$ such that for all $q \in \calC$ and for all $s, t \in \RR$ we have the following:
\begin{itemize} 
	\item $\upC(0, q) = \Id$ and 
	\item $\upC(t + s, q) = \upC(t, g_s q) \, \upC(s, q)$.
\end{itemize} 

The cocycle $\upC$ is \emph{$\log$--integrable} (with respect to a finite flow-invariant measure) if for all $t$ the functions $q \mapsto \log \| \upC(t, q) \|$ and $q \mapsto \log \| \upC(t, q)^{-1} \|$ are integrable.
We will consider cocycles taking values in $\Sp(2m, \ZZ)$. 
For these, the function $q \mapsto \log \| \upC(t, q) \|^{-1}$ is bounded above and its integrability is implied by the finiteness of the invariant measure.
So it will suffice to show that $q \mapsto \log \| \upC(t, q) \|$ is integrable.

If the flow-invariant measure is ergodic, then Oseledets theorem \cites{Ose}[Chapter 4]{Bar-Pes} applies.
It states that for almost every $q$ in $\calC$ and every non-zero vector $v \in \RR^{2m}$ the limit
\[
\lim_{t \to \infty} \frac{1}{t} \log \frac{\| \upC(t, q) \cdot v \|_1}{ \| v \|_1}
\]
exists and depends only on $v$ (and not on $q$).
Moreover, the limit can achieve up to $2m$ values, which, by symplecticity, have the form $\lambda_1 \geq \dots \geq \lambda_{m} \geq -\lambda_m \geq \dots \geq -\lambda_1$.
These numbers are known as the \emph{Lyapunov exponents}; 
the set of numbers is the \emph{Lyapunov spectrum}.
The spectrum is \emph{simple} if all inequalities are strict.

\subsection{Konstevich--Zorich cocycles}
\label{s:KZ_cocycles}

By Lefschetz duality, the Hodge bundles $\upH_\calC$ and $\upH'_\calC$ can also be defined with fibres in $\upH_1(S - Z(S); \RR)$ or $\upH_1(R - Z(R); \RR)$. 
By a standard procedure of filling in the punctures, we also obtain vector bundles with fibres in $\upH_1(S; \RR)$ or $\upH_1(R; \RR)$. By a slight abuse of terminology and notation, we also refer to them as Hodge bundles and denote them by $\upH_\calC$ and $\upH'_\calC$. 
In the sequel, we will only use this notation to refer to these versions with fibres in absolute homology of the (unpunctured) surfaces.

The (lift of the) diagonal flow acts on $\cover{\calC}_\root \cross \upH_1 (S, Z(S); \RR)$ by $g_t ((q,v), c) = (g_t (q,v), c)$.
Quotienting by the action of $\Mod(S, Z(S))$ on both factors, we obtain the \emph{Kontsevich--Zorich cocycle} $\upC^\KZ$ on $\calC_\root$.

Suppose now that $\calC_\root$ is quadratic.  
As in \Cref{s:orientation-covers} we have an isomorphic copy of $\calC_\root$, called $\calC'_\root$ lying in $\calB_\biroot$.  
The pullback of $\upH_\calB$ to $\calC_\root$ is (isomorphic to) $\upH'_\calC$.  
As previously discussed, the pullback splits as a direct sum $\upH'_+ \oplus \upH'_-$ called the plus and minus pieces.  
The plus piece $\upH'_+$ is isomorphic to $\upH_\calC$ while the minus piece $\upH'_-$ is isomorphic to the tangent bundle $\upT\calC_\root$. 
See~\cite[Sections~1 and~2]{Tre}.
The Kontsevich--Zorich cocycle $\upB^\KZ$ on $\calB_\biroot$ pulls back and splits to give 
\begin{itemize}
	\item
	the \emph{plus cocycle} $\upC^\KZ_+$, a copy of $\upC^\KZ$, and
	\item 
	the \emph{minus cocycle} $\upC^\KZ_-$
\end{itemize}
on the plus and minus pieces, respectively.
Again, see~\cite[Sections~1 and~2]{Tre}.

The point of \Cref{s:orientation-covers} is to give a representation of $\upH'_-$ where the minus cocycle $\upC^\KZ_-$ can be computed in terms of a basis of homology classes lying in $\upH_1(R, \ZZ)$.
We do this in \Cref{s:SpanningMinus}.

\subsection{Log-integrability of the continuous cocycles}
\label{s:integrability_continuous}
By \cite[Remark 31]{For-Mat} (which summarises consequences of \cite[Lemma 2.1]{For}), the continuous Kontsevich--Zorich cocycle (over any $SL(2,\RR)$--orbit) is $\log$--integrable.
In particular, this implies the $\log$--integrability of $\upC^\KZ_+$ and $\upC^\KZ_-$.



\subsection{Spanning sets for the plus piece}
\label{s:SpanningPlus}

For each irreducible generalised permutation, there is a preferred set of cycles $\{c_\alpha\}_{\alpha \in \calA}$ on $S - Z$ \cite[Section 4.1]{Gut17}. 
These give a basis for $\upH_1(S - Z; \RR)$ and are (Lefschetz) dual to the rectilinear arcs in \Cref{d:rectilinear}, which belong to $\upH_1(S, Z(S); \RR)$.
Lifting appropriately these also give a basis for $\upH_1^+(R, Z(R); \RR)$.

As temporary notation, we take $\upP = \upC^{\KZ}_+$ to be the plus cocycle.
Suppose that $\pi$ is a permutation. 
Suppose that $\alpha$ and $\beta$ (in $\calA$) are the rightmost top and bottom letters in $\pi$. 
Breaking symmetry, suppose that $\xi \from \pi \to \pi'$ is the Rauzy--Veech move in which $\alpha$ wins. 
Suppose that the cycles $c_\alpha$ and $c_\beta$ (for $\pi$) have non-zero algebraic intersection.
With respect to the plus spanning sets for $\pi$ and $\pi'$, we have 
\[
\upP_\xi = \Id + E(\alpha, \beta)
\]
(Here $E(\alpha, \beta)$ is the matrix whose entries are zero except the $(\alpha, \beta)$--entry which is one.)
Similarly, if $c_\alpha$ and $c_\beta$ have zero algebraic intersection, then 
\[
\upP_\xi = \Id - 2E( \alpha, \alpha) - E(\alpha, \beta)
\]

Recall that the linear change in width parameters in a finite Rauzy--Veech sequence $\delta$ is expressed by the Rauzy matrix $E_\delta$ (see \Cref{e:X-cones}). 

\begin{lemma}
	\label{l:RV_dominates_plus}
	Suppose that $\upP = \upC^{\KZ}_+$ is the plus co-cycle. 
	Then, for any finite Rauzy--Veech sequence $\delta$, we have the following:
	\[
	\| \upP_\delta \|_1 \leq \| E_\delta \|_1
	\]
\end{lemma}

\begin{proof}
	Suppose that $C$ is a matrix with real entries.
	We define $|C|$ to be the (non-negative) matrix whose coefficients are the absolute values of the coefficients of $C$. 
	
	For any Rauzy--Veech move $\delta$, the plus cocycle satisfies $| \upP_\delta | = E_\delta $.
	It follows that for any finite Rauzy--Veech sequence $\delta = \delta_1 \delta_2 \dotsc \delta_k$ 
	\begin{align*}
		\| \upP_\delta \|_1 &= \|  \upP_{\delta_1} \upP_{\delta_2} \dotsc \upP_{\delta_k} \|_1 \\
		&\leq \| |\upP_{\delta_1}| \cdot | \upP_{\delta_2} | \dotsc  | \upP_{\delta_k}| \|_1 \\ 
		&= \| E_{\delta_1} E_{\delta_2} \dotsc E_{\delta_k} \|_1 \\ 
		&=  \| E_\delta \|_1 \qedhere
	\end{align*}
\end{proof}

\subsection{Spanning sets for the minus piece}
\label{s:SpanningMinus}

We now use zippered rectangles to give an explicit spanning set for the minus piece.
Suppose that $\pi$ is an irreducible labelled generalised permutation. 
Suppose that $j = \sigma(i)$ and $i < j$. 
Let $g_i$ and $g'_i$ be the two (oriented) lifts to $\cover{q}$ of the rectilinear arc $\gamma_i$ (as in \Cref{d:rectilinear}).
By construction, the involution $\iota$ exchanges $g_i$ and $g'_i$ preserving orientation.
Let $\alpha = \pi(i)$.
We now define the relative homology class
\[
c_\alpha = [g_i] - [g'_i] \in \upH_1(\cover{S}, \cover{Z}; \RR)
\]
Applying $\iota$ we find that $c_\alpha$ lies in the minus piece $\upH_1^-(\cover{S}, \cover{Z}; \RR)$ of homology. 

See \Cref{f:minus-homology} for an illustration of these cycles.

\begin{figure}
	\centering
	\begin{subfigure}[b]{\textwidth}
		\centering
		\includegraphics[scale=0.80]{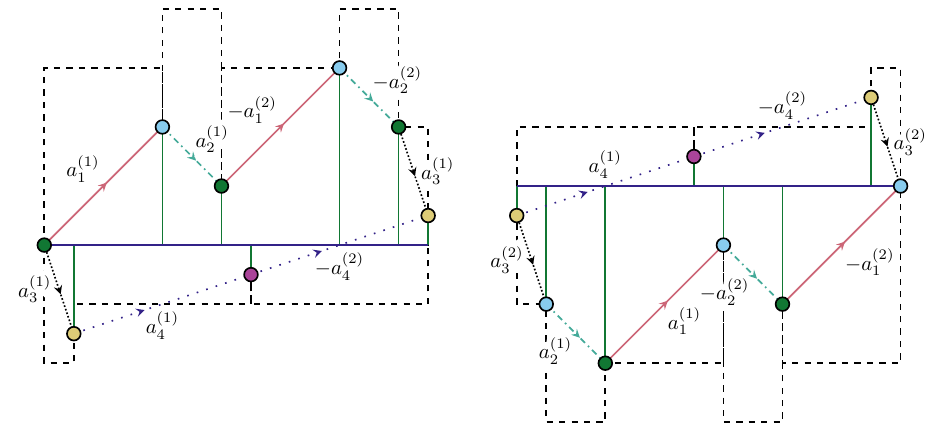}
		\caption{Original double cover.}
	\end{subfigure}
	
	\begin{subfigure}[b]{\textwidth}
		\centering
		\includegraphics[scale=0.80]{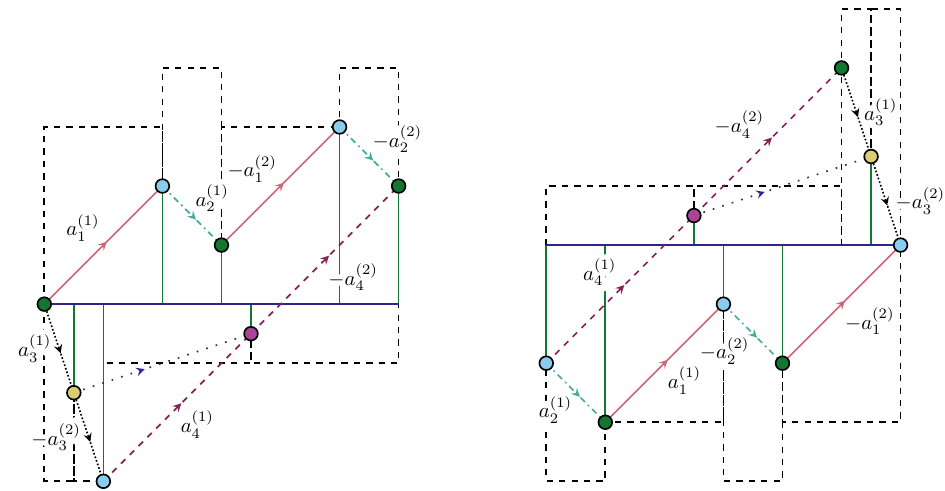}
		\caption{Double cover after a Rauzy move.}
	\end{subfigure}
	\caption{Example of the spanning set for the minus piece rendering the linear transformations coming from Rauzy moves equal to the Rauzy--Veech matrices. The original permutation is $\big(\begin{smallmatrix} 1 & 2 & 1 & 2 & 3 \\ & 3 & 4 & 4 \end{smallmatrix}\big)$ representing the stratum $\calQ(2,-1,-1)$, which becomes $\big(\begin{smallmatrix} 1 & 2 & 1 & 2 \\ 3 & 3 & 4 & 4 \end{smallmatrix}\big)$ after one bottom Rauzy move. In this case, the cycles in the spanning set can be tightened to saddle connections, so they are drawn in this manner. The general case is similar.}
	\label{f:minus-homology}
\end{figure}

By direct computation we obtain that in a Rauzy--Veech move, the matrix (with respect to initial and final minus spanning sets) for the minus cocyle $\upC^\KZ_-$ is exactly the Rauzy matrix.

\subsection{Log-integrability of the discrete cocycles}
\label{s:integrability_discrete}

Suppose that $(\calS, \nu)$ is a measure space with a measure preserving transformation $T: \calS \to \calS$.
A \emph{discrete symplectic cocycle} for $T$ 
(acting on $\RR^{2m}$) is a map $\upD \from \ZZ \times \calS \mapsto \Sp(2m, \ZZ)$ such that for all $q \in \calS$ and for all $s, t \in \ZZ$ we have the following:
\begin{itemize} 
	\item $\upD(0, q) = \Id$ and 
	\item $\upD(t + s, q) = \upD(t, T^s q) \, \upD(s, q)$.
\end{itemize} 

The definition of $\log$--integrability of $\upD$, and of the Lyapunov exponents for $\upD$, is similar to that in \Cref{s:symplectic_cocycles}. 

Recall our construction of the Poincaré section $\calS_P = \bigcup_s \calS^{(1)}_P(\theta_s| \theta_s)$.
Recall that the pieces of the first return are given by the countable alphabet $\Pi$ in \Cref{s:countable_shift}.
Suppose that $(x,y)$ lies in $\calS_P$ and in fact lies in $\calS_P (\theta_r |\xi  \theta_t)$ (as in \Cref{l:domain-and-range}).
We define $\upD(1, (x,y)) = E_\xi$.
We extend to obtain the discrete Rauzy cocycle $\upD^{\R} \from \ZZ \times \calS \mapsto \Sp(2m, \ZZ)$.

\begin{lemma}
	\label{l:RV_integrable}
	The Rauzy cocycle $\upD^{\R}$ is $\log$--integrable (with respect to $\nu^{(1)}_P$).
\end{lemma}

\begin{proof}
	Recall that the alphabet $\Pi'$ (from \Cref{s:first-return-time}) gives the first returns to the smaller section $\calS^{(1)}_P (\theta | \theta)$. 
	
	Suppose that $\xi$ is contained in $\Pi$.
	Recall that $\Pi'_\xi$ (also from \Cref{s:first-return-time}) is the set of extensions $\eta$ of $\xi$ that lie in $\Pi'$.
	Since $\| E_\xi \|_1 \leqslant \| E_\eta \|_1$, it suffices to prove $\log$--integrability (over $\calS^{(1)}_P(\theta| \theta)$) of the (discrete) cocycle defined by 
	\[
	\upD'(1, (x,y)) = E_\eta \qquad \mbox{for $(x, y)$ in $\calS^{(1)}_P(\theta| \eta \theta)$}
	\]
	
	Since $\eta$ is $k_\theta$--balanced, we have
	\[
	\|E_\eta \|_1 \leq k_\theta \cdot w(E_\eta x') 
	\]
	for any $x'$ where there is some $y'$ with $(x', y')$ in $\calS^{(1)}_P (\theta \eta | \theta)$.
	Recall that the first return time function $\rho'(x, y)$, for the section $\calS^{(1)}_P(\theta | \theta)$, satisfies \Cref{e:restricted_return_time}. 
	Hence, $\log \| E_\eta \|_1 \leq \rho'(x, y) + \log  k_\theta$.
	By \Cref{l:first-return-time-integrable} and \Cref{l:section-finite-volume}, we deduce that $\upD'$ is $\log$--integrable.
\end{proof}

We define the \emph{discrete cocycles} $\upD^{\KZ}_{\pm}$ by restricting $\upC^{\KZ}_{\pm}$ to $\Pi$.
By \Cref{l:RV_dominates_plus}, we deduce that $\| \upD^{\KZ}_+ \|_1 \leq \| \upD^{\R} \|_1$. 
Since $\upC^{\KZ}_{-, \xi} = E_\xi$ for all finite Rauzy--Veech sequences $\xi$, we deduce that $\| \upD^{\KZ}_- \|_1 = \| \upD^{\R} \|_1$.
Thus \Cref{l:RV_integrable} proves the following.

\begin{corollary}
	\label{c:discrete_integrable}
	The discrete cocycles $\upD^\KZ_{\pm}$ are $\log$--integrable.  \qed
\end{corollary}

\subsection{The discrete and continuous Lyapunov spectra}

Before discussing the Lyapunov spectra, we restrict from relative to absolute homology.
We do this following \cite[Section 3.3, page 34]{For-Mat}.

\begin{remark}
	\label{r:absolute_integrable}
	We note that $\upH_1(S; \RR)$ is a subgroup of $\upH_1(S, Z(S); \RR)$.  
	All cocycles preserve the resulting subbundle of the Hodge bundle.
	Moreover, the (images of the) cocyles are \emph{bounded} on the quotient bundle coming from $\upH_1(S, Z(S); \RR)/\upH_1(S; \RR)$. 
	Thus the Lyapunov exponents contributed by the relative cycles in $\upH_1(S, Z(S); \RR)$ vanish.
	
	Accordingly, we henceforth only consider cocycles restricted to the absolute homology. 
	Since they are obtained by restriction, they are again $\log$--integrable (by \Cref{s:integrability_continuous} and \Cref{c:discrete_integrable}).
	In a small abuse of notation we use the same notation for the restricted cocycles.
\end{remark}

We now apply Oseledets' theorem.
That is, for $\nu^{(1)}_P$--almost every $(x, y)$ in the Poincaré section $\calS = \bigcup_s \calS^{(1)}_P (\theta_s | \theta_s)$ and for every $v$ in the absolute part of the plus or minus piece we have that 
\[
\lim_{n \to \infty} \frac{1}{n} \log \frac{\| \upD^\KZ_\pm (n, (x,y)) \cdot v \|_1}{\| v \|_1}
\]
exists. 
Thus, the Lyapunov spectra for both of $\upD^\KZ_\pm$ exist.

Recall that $\rho$ is the first return time function for our section. 
We define 
\[
\rho_\av = \int_{\calS } \rho(x,y) \, \diff \nu^{(1)}_P 
\]
Suppose that $(x,y)$ lies in $\calS$ and returns infinitely often.
We denote the time between the $i$--th and the $(i-1)$--th return by $\rho_i(x,y)$.
By \Cref{t:all-ergodic}, it follows that for $\nu^{(1)}_P$--almost every $(x,y)$ in $\calS$ we have
\[
\lim_{n \to \infty} \frac{1}{n} \sum_{i=0}^{n-1} \rho_i(x,y) = \rho_\av 
\]

It follows that each Lyapunov exponent for the discrete cocycles $\upD^{\KZ}_\pm$ is $\rho_\av$ times the corresponding exponent for the continuous cocycles $\upC^{\KZ}_\pm$.
In particular, simplicity for the discrete cocyles implies simplicity for the continuous ones.

\subsection{Simplicity criterion}

We now state a criterion for Lyapunov simplicity of discrete cocycles. 
We state a specific version (for symplectic cocycles) that follows from combining \cite{Ben97} with \cites{Avi-Via07a}[Theorem 7.1]{Avi-Via07b}.
For further details, see \cite[Footnote 2 and Appendix A]{Avi-Mat-Yoc}.

Suppose that $\calC_\root$ is a stratum component of rooted differentials. 
Suppose that $\calS$ is our Poincaré section.
We consider almost-flow loops that arise from flow segments that start and end in $\calS$.

Suppose that $\gamma_1$ and $\gamma_2$ are almost-flow loops in $\calC_\root$.
The product $\gamma_1 \gamma_2$ is also realised by an almost-flow loop. 
Thus, almost-flow loops form a monoid under concatenations.
We call this the \emph{flow monoid}.
Suppose that $\upD$ is a locally constant integral symplectic cocycle over $\calS$.
By evaluating $\upD$ over each almost-flow loop, we obtain a representation of the flow monoid into the integral symplectic group.
We call this the \emph{symplectic flow monoid}.

\begin{criterion}\label{c:simplicity}
	Suppose that $\upD$ is a locally constant integral symplectic cocycle over $\calS$ that is $\log$--integrable with respect $\nu^{(1)}_P$.
	If the symplectic flow monoid is Zariski dense (in the symplectic group), then the Lyapunov spectrum is simple.
\end{criterion}

\section{Monodromy groups}
\label{s:RV}

Suppose that $\pi$ is a (generalised) permutation.
Suppose that $\calD_\lab$ is a labelled Rauzy diagram containing $\pi$ as its basepoint.
Suppose that $\calD_\root$ is the unlabelled Rauzy diagram containing $[\pi]$ as its basepoint.
Suppose that $\calC_\root$ and $\calC$ are the corresponding strata of (rooted) abelian or quadratic differentials; here we use $q^\pi$ as the basepoint (see \Cref{s:Rauzy_diagrams_flow_group}).
We assume that $\calC$ is non-empty.
Suppose that $\calM(S, Z(S))$ and $\calM(S)$ are the corresponding moduli spaces; here we use the conformal structure of $q^\pi$ as the basepoint.
Finally we take $\calA(S)$ to be the moduli space of \emph{principally polarised abelian varieties}, based at the image of $q^\pi$.
Thus we have natural maps of pointed spaces:
\[
\calD_\lab \to \calD_\root \to \calC_\root \to \calC \to \calM(S, Z(S)) \to \calM(S) \to \calA(S)
\]
This gives a sequence of (orbifold) fundamental groups.  
Pulling back, we obtain representations of these fundamental groups into the relative mapping class group $\Mod_Z(S) = \Mod(S, Z(S))$, the absolute mapping class group $\Mod(S)$, and the symplectic group $\Sp(S) = \Sp(2g(S), \ZZ)$. 
We arrange these groups in the following commutative diagram.
\begin{figure}[H]
	\[
	\begin{tikzcd}[column sep=small]
		\uppi_1(\calD_\lab) \arrow[equal]{r} \arrow[hook,"\text{f.i.}"]{d} & \uppi_1(\calD_\lab) \arrow[two heads]{r} \arrow[hook,"\text{f.i.}"]{d} & \Mod_Z(\calD_\lab) \arrow[two heads]{r} \arrow[hook,"\text{f.i.}"]{d} & \Mod(\calD_\lab) \arrow[two heads]{r} \arrow[hook,"\text{f.i.}"]{d} & \Sp(\calD_\lab) \arrow[hook,"\text{f.i.}"]{d} \\
		\uppi_1(\calD_\root) \arrow[two heads]{r} & \uppi_1(\calC_\root) \arrow[two heads]{r} \arrow[hook,"\text{f.i.}"]{d} & \Mod_Z(\calC_\root) \arrow[two heads]{r} \arrow[hook,"\text{f.i.}"]{d} & \Mod(\calC_\root) \arrow[two heads]{r} \arrow[hook,"\text{f.i.}"]{d} & \Sp(\calC_\root) \arrow[hook,"\text{f.i.}"]{d} \\
		& \uppi_1^\orb(\calC) \arrow[two heads]{r} & \Mod_Z(\calC) \arrow[two heads]{r} \arrow[hook]{d} & \Mod(\calC) \arrow[two heads]{r} \arrow[hook]{d} & \Sp(\calC) \arrow[hook]{d} \\
		& & \Mod_Z(S) \arrow[two heads]{r} & \Mod(S) \arrow[two heads]{r} & \Sp(S)
	\end{tikzcd}
	\]
	\caption{Relations between different groups in the abelian case.}
	\label{fig:commutative_diagram_abelian}
\end{figure}
Here ``f.i.'' stands for ``finite index''.
Note that the surjectivity of the homomorphism 
\[
\uppi_1(\calD_\root, [\pi]) \to \uppi_1(\calC_\root, q^\pi)
\]
is given by \Cref{t:pi1-surjective}.
We call the groups in the first row \emph{Rauzy--Veech} groups and those in the second and third rows \emph{monodromy groups}.
The groups in the third and fourth columns are the relative and absolute \emph{modular groups}; 
those in the fifth column are the \emph{symplectic groups}.

The abelian case of the following corollary provides a partial answer (that is, up to finite index) to a question by Yoccoz \cite[Remark in Section 9.3]{Yoc}.

\begin{corollary}
	\label{c:RV_finite_index_in_monodromy}
	Suppose that $\calC$ is a stratum component of abelian or quadratic differentials.
	Suppose that $\calD_\lab$ is a labelled Rauzy--Veech diagram associated to $\calC$.
	Then the relative modular Rauzy--Veech group $\Mod_Z(\calD_\lab)$ has finite index in relative modular monodromy group $\Mod_Z(\calC)$. \qed
\end{corollary}

In the quadratic case, there is also the following sequence of pointed spaces:
\[
\calD_\lab \to \calD_\root \to \calC'_\root \to \Mod_Z(R) \to \Mod(R) \to \calA(R)
\]
Recall that that the Hodge bundle over $\calC'_\root$ splits into the plus and minus pieces.
Again we may take fundamental groups and pull-back to obtain two commutative diagrams. 
\begin{figure}[H]
	\[
	\begin{tikzcd}[column sep=small]
		\uppi_1(\calD_\lab) \arrow[equal]{r} \arrow[hook,"\text{f.i.}"]{d} & \uppi_1(\calD_\lab) \arrow[two heads]{r} \arrow[hook,"\text{f.i.}"]{d} & \Mod_Z(\calD_\lab) \arrow[two heads]{r} \arrow[hook,"\text{f.i.}"]{d} & \Mod(\calD_\lab) \arrow[two heads]{r} \arrow[hook,"\text{f.i.}"]{d} & \Sp_{\pm}(\calD_\lab) \arrow[hook,"\text{f.i.}"]{d} \\
		\uppi_1(\calD_\root) \arrow[two heads]{r} & \uppi_1(\calC'_\root) \arrow[two heads]{r} & \Mod_Z(\calC'_\root) \arrow[two heads]{r} \arrow[hook,"\text{f.i.}"]{d} & \Mod(\calC'_\root) \arrow[two heads]{r} \arrow[hook,"\text{f.i.}"]{d} & \Sp_{\pm}(\calC'_\root) \arrow[hook,"\text{f.i.}"]{d} \\
		& & \Mod_Z(R) \arrow[two heads]{r} & \Mod(R) \arrow[two heads]{r} & \Sp_{\pm}(R)
	\end{tikzcd}
	\]
	\caption{Relations between different groups in the quadratic case.}
	\label{fig:commutative_diagram_quadratic}
\end{figure}
Note that $\calC'_\root$ is canonically isomorphic to $\calC_\root$.
Also, here there are only three rows, instead of four, because $\calC$ lives in the cotangent bundle to the moduli space $\calM_Z(S)$ (not $\calM_Z(R)$). 


We now combine \Cref{c:RV_finite_index_in_monodromy} with the work of Calderon and Calderon--Salter~\cite{Cal, Cal-Sal21a, Cal-Sal21b, Cal-Sal23}.
We thus obtain a description (up to finite index) of the modular Rauzy--Veech groups $\Mod_Z(\calD_\lab)$ and also the Rauzy--Veech groups lying in $\Aut(\upH_1(S, Z; \ZZ)$ for non-hyperelliptic abelian components in genus at least five.

\begin{corollary}
	Suppose that $\calC$ is a non-hyperelliptic, abelian stratum component. 
	Suppose that $\calM(S, Z)$ is the associated moduli space.
	Suppose that $S$ has genus at least five. 
	Suppose that $\phi$ is the framing of the unit tangent bundle $T^1(S - Z)$ induced by the horizontal vector field of $q \in \calC$. 
	Let $\Mod(S, Z)[\phi]$ denote the stabiliser of $\phi$ in $\Mod(S, Z)$.
	Then we have:
	\begin{enumerate}
		\item 
		The modular Rauzy--Veech group $\Mod_Z(\calD_\lab)$ is a finite-index subgroup of $\Mod(S, Z)[\phi]$.
		\item 
		The Rauzy--Veech group in $\PAut(\upH_1(S, Z; \ZZ))$ is a finite-index subgroup of the kernel of the crossed homomorphism 
		\[
		\Theta_\phi \from \PAut(\upH_1(S, Z; \ZZ)) \to \upH^1(S; \ZZ/2\ZZ)
		\]
		defined by Calderon--Salter~\cite[Section 4]{Cal-Sal21b}. \qed
	\end{enumerate}
\end{corollary}

This classification was already known for hyperelliptic components, and in this case the index is known to be one \cite{Avi-Mat-Yoc}.

\section{Classification of components and adjacencies}

\subsection{Classification of components of strata of abelian and quadratic differentials} \label{s:classification}

For the reader's convenience, we restate the complete classification of the components of abelian and quadratic strata.

\begin{theorem}[{\cite{Kon-Zor03}}]
	\label{Thm:ClassificationAbelian}
	The following is the classification of the components of the strata of abelian differentials (up to regular marked points).
	\begin{itemize}
		\item In genus one, the only stratum is $\calH(0)$. It is non-empty, connected and hyperelliptic.
		\item In genus two, the only strata are $\calH(2)$ and $\calH(1,1)$. They are non-empty, connected and hyperelliptic.
		\item In genus three, the strata $\calH(4)$ and $\calH(2,2)$ have two components. One of them is hyperelliptic and the other one corresponds to odd spin structures. Every other stratum is non-empty and connected.
		\item Finally, for genus $g$ at least four:
		\begin{itemize}
			\item The stratum $\calH(2g - 2)$ has three components. One of them is hyperelliptic, and the other two are distinguished by even and odd spin structures.
			\item The stratum $\calH(g - 1, g - 1)$ can have two or three components depending on the parity of $g$. If $g$ is even, it has two components. One of them is hyperelliptic, and the other one is not. If $g$ is odd, it has three components. One of them is hyperelliptic, and the other two are distinguished by even and odd spin structures.
			\item All other strata of the form $\calH(2\kappa_1, \dotsc, 2\kappa_n)$ have two components, distinguished by even and odd spin structures.
			\item The remaining strata are non-empty and connected.
		\end{itemize}
	\end{itemize}
\end{theorem}

\begin{theorem}[{\cite{Lan08,Che-Moe}}]
	\label{Thm:ClassificationQuadratic}
	The following is the classification of the components of the strata of quadratic differentials (up to regular marked points).
	\begin{itemize}
		\item In genus zero, every stratum is non-empty and connected.
		\item In genus one, the strata $\calQ(0)$ and $\calQ(1, -1)$ are empty. All other strata are nonempty and connected.
		\item In genus two, the strata $\calQ(4)$ and $\calQ(3, 1)$ are empty. Moreover, the stratum $\calQ(2, 2)$ is non-empty, connected and hyperelliptic.
		\item In genus three, the strata $\calQ(9, -1)$, $\calQ(6, 3, -1)$ and $\calQ(3, 3, 3, -1)$ have two components, known as \emph{regular} and \emph{irregular} components.
		\item In genus four, the strata $\calQ(6,6)$, $\calQ(6, 3, 3)$ and $\calQ(3, 3, 3, 3)$ have three components. One of them is hyperelliptic, and the other two are known as \emph{regular} and \emph{irregular} components. Moreover, the strata $\calQ(12)$ and $\calQ(9, 3)$ have two components, known as \emph{regular} and \emph{irregular} components.
		\item Finally, for genus at least two:
		\begin{itemize}
			\item The strata of the form $\calQ(4j + 2, 4k + 2)$, $\calQ(4j + 2, 2k - 1, 2k - 1)$ and $\calQ(2j - 1, 2j - 1, 2k - 1, 2k - 1)$ for $j, k \geq 0$ not contained in the previous list have two components. One of them is hyperelliptic, and the other one is not.
			\item The remaining strata are non-empty and connected.
		\end{itemize}
	\end{itemize}
\end{theorem}

\subsection{Simple extensions}
\label{s:extensions}

Suppose that $\pi$ and $\pi'$ are irreducible (generalised) permutations.

\begin{definition}
	\label{d:type_preserving}
	We say $\pi'$ is a \emph{type-preserving simple extension} of $\pi$ if 
	\begin{itemize}
		\item 
		$\pi$ and $\pi'$ have the same type (both abelian or both quadratic) and 
		\item
		$\pi'$ is obtained from $\pi$ by inserting a single letter $\alpha$.  
		Furthermore, in $\pi'$: 
		\begin{itemize}
			\item 
			if $\alpha$ is leftmost in one row (top or bottom) then it is not leftmost in the other row; and
			\item 
			$\alpha$ is not rightmost in either row (top or bottom).
			\qedhere
		\end{itemize}
	\end{itemize}
\end{definition}

\noindent
Similarly, we have
\begin{definition}
	\label{d:type_changing}
	We say that $\pi'$ is a \emph{type-changing simple extension} of $\pi$ if 
	\begin{itemize}
		\item 
		$\pi$ is abelian, $\pi'$ is quadratic, and 
		\item
		$\pi'$ is obtained from $\pi$ by inserting a single top flip letter $\alpha$ and a single bottom flip letter $\beta$.
		Furthermore, in $\pi'$:
		\begin{itemize}
			\item 
			at most one of $\alpha$ or $\beta$ is leftmost in its row; and
			\item neither $\alpha$ nor $\beta$ is rightmost in its row. \qedhere
		\end{itemize}
	\end{itemize}
\end{definition}

Note that there are no type-changing extensions from a quadratic permutation to an abelian one.
Type-preserving extensions in the abelian case were introduced by Avila--Viana~\cite[Section~5.2]{Avi-Via07b}.
Type-changing extensions were introduced in~\cite[Definition~3.1]{Gut17}.
By~\cite[Theorem~3.2]{Boi-Lan}, a (type-preserving or -changing) simple extension $\pi'$ of an irreducible (generalised) permutation $\pi$ is again irreducible.
See also~\cite[Lemma~3.2]{Gut17}

Regardless of the types, a simple extension either \emph{preserves} or \emph{increases} the genus of the underlying surface. 

We now follow the paragraph immediately before, and the proof of,~\cite[Lemma~3.3]{Gut17}.

\begin{definition}
	Suppose that $\pi$ lies in a labelled Rauzy diagram $\calD$.
	Suppose that a simple extension $\pi'$ of $\pi$ lies in $\calD'$.
	
	We now define $\calD''$, the \emph{labelled diagram of simple extensions containing $\pi'$}.
	The vertices of $\calD''$ are exactly the permutations in $\calD'$ that are simple extensions of permutations in $\calD$.
	The arrows of $\calD''$ are more subtle.
	(Essentially, when an inserted letter appears at the end, it must lose.) 
	
	Suppose that $\sigma \to \tau$ is a Rauzy move in $\calD$. 
	Suppose that $\delta$ is the (copy of the) losing letter and $\epsilon$ is the (copy of the) winning letter.
	Suppose that $\sigma'$ in $\calD''$ is a simple extension of $\sigma$. 
	Suppose that $\zeta$ is the (copy of the) letter immediately to the left of $\delta$. 
	We now have two cases.
	\begin{itemize}
		\item 
		Suppose that $\zeta$ is not an inserted letter. 
		In this case, we take $\tau' \in \calD''$ to be the result of $\delta$ losing to $\epsilon$.
		\item 
		Suppose that $\zeta$ is an inserted letter, say $\alpha$.
		Suppose that $\eta$ is the (copy of the) letter immediately to the left of $\zeta$.
		We now have two subcases. 
		\begin{itemize}
			\item 
			Suppose that $\eta$ is not an inserted letter.
			In this case, we take $\tau' \in \calD''$ to be the result of $\delta$ and then $\alpha = \zeta$ losing to $\epsilon$.
			\item
			Suppose that $\eta$ is an inserted letter. 
			Thus $\alpha = \zeta = \eta$ was a flip letter. 
			In this case, we take $\tau' \in \calD''$ to be the result of $\delta$ and then $\alpha = \zeta$ losing (the latter twice) to $\epsilon$.
		\end{itemize}
	\end{itemize}
	In all cases we add an arrow from $\sigma'$ to $\tau'$ to $\calD''$.
	(See the figures in the proof of~\cite[Lemma~3.3]{Gut17}.)
	This completes the construction of $\calD''$.
\end{definition}

By construction, the diagram $\calD''$ is a finite cover of $\calD$.
If $\calD'$ is abelian, then each connected component of $\calD''$ is a degree-one cover of $\calD$~\cite[Section~5.2]{Avi-Via07b}.

\begin{example}
	\label{exa:RV_noT(c)ontained}    
	By \cite[Definition 3.1]{Boi-Lan}, the permutation $\sigma$ given by
	\[
	\sigma = \left(\begin{smallmatrix} && 1 & 2 & 1 & 2 & 0 \\ 0 & 3 & 4 & 5 & 6 & 3 & 4 & 5 & 6 \end{smallmatrix}\right)
	\]
	is irreducible and lies in a labelled Rauzy diagram of $\calQ(2,4)$. 
	The finite Rauzy sequence $\xi$ in which $0$ wins once over $6$, $5$, $4$, and $3$ is a loop based at $\sigma$.
	Suppose we insert a label $\alpha$ to obtain the permutation 
	\[
	\tau = \left(\begin{smallmatrix} &&& 1 & 2 & 1 & 2 & 0 \\ 0 & \alpha & 3 & 4 & \alpha & 5 & 6 & 3 & 4 & 5 & 6 \end{smallmatrix}\right)
	\]
	Then the sequence $\xi''$ in $\calD''$ shadowing $\xi$ does not return to $\tau$.
	On the other hand, the sequence shadowing $\xi^2$ does return to $\tau$.
	Thus, in this example, $\calD''$ is a cover of $\calD$ of degree at least two. 
\end{example}

Since arrows in $\calD''$ are finite composites of arrows in $\calD'$, directed loops in $\calD''$ naturally give loops in $\calD'$.
We use this observation in the abelian and quadratic settings as follows.

Suppose that $\pi$ is abelian and $\pi'$ is a type-preserving simple extension of $\pi$.
Suppose additionally that the underlying surfaces for $\pi'$ and $\pi$ have the same genus.
Avila--Viana (\cite[Lemma 5.6]{Avi-Via07b}) prove that the actions on absolute homology of $S$ induced by an arrow in $\calD$ and an associated (composite) arrow in $\calD''$ coincide.
Composing the maps $\uppi_1(\calD'') \to \uppi_1(\calD')$ and $\uppi_1(\calD') \to \Sp(S)$ (the top row of \Cref{fig:commutative_diagram_abelian}), we obtain the image $\Sp(\calD'')$. 
It follows that $\Sp(\calD'') = \Sp(\calD)$.
We record this in the following commutative diagram.
\[
\begin{tikzcd}
	\uppi_1(\calD') \arrow[two heads]{r} & \Sp(\calD') \arrow[hook]{r} & \Sp(S) \arrow[equal]{d} \\
	\uppi_1(\calD'') \arrow[two heads]{r} \arrow["\isom"]{d} \arrow[hook]{u} & \Sp(\calD'') \arrow[hook]{r} \arrow[equal]{d} \arrow[hook]{u} & \Sp(S) \arrow[equal]{d} \\
	\uppi_1(\calD) \arrow[two heads]{r} & \Sp(\calD) \arrow[hook]{r} & \Sp(S)
\end{tikzcd}
\]

Suppose now that $\pi'$ is a quadratic permutation. 
If $\calD$ is abelian then we let $R$ be a disjoint union of two copies of $S$ and $\iota$ the involution that switches the two copies. 
Similar to the abelian case, \cite[Lemma 4.3]{Gut17} shows that the actions on the absolute part of the plus piece (in the homology of the orientation double cover $R$) induced by an arrow in $\calD$ and an associated (composite) arrow in $\calD''$ coincide.
Composing the maps $\uppi_1(\calD'') \to \uppi_1(\calD')$ and $\uppi_1(\calD') \to \Sp_+(R)$ (the top row of \Cref{fig:commutative_diagram_quadratic}), we obtain the image $\Sp_+(\calD'')$. 
It follows that $\Sp_+ (\calD'')$ is finite index in $\Sp_+ (\calD)$.
We record this in the following commutative diagram.
\[
\begin{tikzcd}
	\uppi_1(\calD') \arrow[two heads]{r} & \Sp_+(\calD') \arrow[hook]{r} & \Sp_+(R) \arrow[equal]{d} \\
	\uppi_1(\calD'') \arrow[two heads]{r} \arrow[hook, "\text{f.i.}"]{d} \arrow[hook]{u} & \Sp_+(\calD'') \arrow[hook]{r} \arrow[hook, "\text{f.i.}"]{d} \arrow[hook]{u} & \Sp_+(R) \arrow[equal]{d} \\
	\uppi_1(\calD) \arrow[two heads]{r} & \Sp_+(\calD) \arrow[hook]{r} & \Sp_+(R)
\end{tikzcd}
\]

We now have the following.

\begin{corollary}
	\label{c:RV_inclusion}
	Suppose that $\pi$ is an irreducible (generalised permutation) in a labelled Rauzy diagram $\calD$.
	Suppose that $\pi'$ in $\calD'$ is obtained from $\pi$ by a finite sequence of genus-preserving simple extensions. 
	Then,
	\begin{itemize}
		\item If $\pi$ is abelian, then $\Sp_+(\calD')$ contains a subgroup isomorphic to $\Sp(\calD)$.
		\item If $\pi$ is quadratic, then $\Sp_+(\calD')$ contains a subgroup isomorphic to a finite-index subgroup of $\Sp_+(\calD)$. 
	\end{itemize} 
\end{corollary}

\begin{proof}
	So suppose that $\pi$ is abelian. 
	Suppose that $\xi = (\pi_i)_i$ is a loop in $\calD$. 
	Suppose that $\xi'' = (\pi_i'')_i$ is the resulting path in $\calD''$ (which in turn has a natural inclusion into $\calD'$) 
	which shadows $\xi$. 
	A flip letter in $\pi_i''$ always loses to some translation letter in $\pi_i''$; 
	furthermore, this translation letter appears in $\pi_i$. 
	Thus flip letters remain flip letters and do not change side.  
	The rest of the proof follows Section~5.2 of~\cite{Avi-Via07b}.
	
	The second conclusion follows from our second commutative diagram (immediately above). 
\end{proof}

\subsection{Extending from basic components}
\label{s:extending_from_basic}

For this section only, we restrict to stratum components where the underlying genus is at least one.  
We also call the following stratum components \emph{basic components}: 
\begin{itemize}
	\item minimal abelian (a single zero),
	\item minimal quadratic (a single zero),
	\item hyperelliptic quadratic with two singularities,
	\item $\calQ(5, -1)$, or $\calQ(9, -1)^\irreg$.
\end{itemize}

If $\calQ(\kappa)$ contains a basic component then we call $\kappa$ a \emph{basic datum}.

\begin{table}
	\centering
	\begin{tabular}{|c|c|c|}
		\hline
		Source & Target & Permutation
		\\
		\hline
		$\calH(2)$ & $\calQ(6, -1, -1)^\hyp$ & $\begin{pmatrix*}
			1 & \alpha & \alpha & 2 & 3 & 4 \\ 4 & 3 & 2 & \beta & \beta & 1
		\end{pmatrix*}$
		\\
		\hline
		$\calH(2)$ & $\calQ(6, -1, -1)^\nonhyp$ & $\begin{pmatrix*}
			1 & \alpha & \alpha & 2 & 3 & 4 \\ 4 & 3 & \beta & \beta & 2 & 1
		\end{pmatrix*}$
		\\
		\hline
		$\calH(1, 1)$ & $\calQ(3, 3, -1, -1)^\hyp$ & $\begin{pmatrix*}
			1 & \alpha & \alpha & 2 & 3 & 4 & 5 \\ 5 & 4 & 3 & 2 & \beta & \beta & 1 
		\end{pmatrix*}$
		\\
		\hline
		$\calH(1, 1)$ & $\calQ(3, 3, -1, -1)^\nonhyp$ & $\begin{pmatrix*}
			1 & \alpha & \alpha & 2 & 3 & 4 & 5 \\ 5 & 4 & \beta & \beta & 3 & 2 & 1
		\end{pmatrix*}$
		\\
		\hline
	\end{tabular}
	\caption{Explicit extensions into the strata $\calQ(6, -1, -1)$ and $\calQ(3, 3, -1, -1)$ in genus two. The permutation in the third column represents the component in the second column. Erasing the letters $\alpha$ and $\beta$ produces a permutation representing the (connected) stratum in the first column. Moreover, the source permutation in $\calH(1, 1)$ is itself a simple extension of a permutation in $\calH(2)$.}
	\label{table:extensions_genus_two}
\end{table}

\begin{table}
	\centering
	\begin{tabular}{|c|c|c|}
		\hline
		Source & Target & Permutation
		\\
		\hline
		$\calQ(8)$ & $\calQ(9, -1)^\reg$ & $\begin{pmatrix*}
			1 & 2 & 1 & 3 & 4 & 3 & 5 & 2 \\
			& 6 & 5 & 4 & \alpha & \alpha & 6
		\end{pmatrix*}$
		\\
		\hline
		$\calQ(12)^\reg$ & $\calQ(6, 6)^\reg$ & $\begin{pmatrix*}
			1 & 2 & \alpha & 3 & 4 & 5 & \alpha & 6 & 7 & 5 \\
			& 2 & 4 & 6 & 8 & 7 & 8 & 3 & 1
		\end{pmatrix*}$
		\\
		\hline
		$\calQ(12)^\irreg$ & $\calQ(6, 6)^\irreg$ & $\begin{pmatrix*}
			& 1 & 2 & 3 & 4 & 3 & 5 & 6 & 7 \\
			8 & 1 & 6 & 8 & \alpha & 4 & 2 & 7 & \alpha & 5
		\end{pmatrix*}$
		\\
		\hline
		$\calQ(12)^\reg$ & $\calQ(9, 3)^\reg$ & $\begin{pmatrix*}
			1 & 2 & 3 & \alpha & 4 & 3 & 5 & 6 & 7 \\
			2 & 4 & 6 & 8 & 7 & 8 & 5 & \alpha & 1
		\end{pmatrix*}$
		\\
		\hline
		$\calQ(12)^\irreg$ & $\calQ(9,3)^\irreg$ & $\begin{pmatrix*}
			1 & 2 & 3 & 4 & 5 & \alpha & 4 & \alpha & 6 & 7 \\
			&2 & 6 & 8 & 5 & 3 & 7 & 8 & 1
		\end{pmatrix*}$
		\\
		\hline
	\end{tabular}
	\caption{Explicit extensions into non-hyperelliptic components of exceptional strata with less than three singularities. The permutation in the third column represents the component in the second column. Erasing the letter $\alpha$ produces a permutation representing the component of a minimal stratum in the first column.}
	\label{table:extensions_sporadic}
\end{table}

\begin{proposition}
	\label{prop:induction}
	Suppose that $\calC'$ is a stratum component in genus at least one.
	Suppose that $\calC'$ is not basic.
	Then there is a basic component $\calC$ and (generalised) permutations $\pi$ and $\pi'$ in the associated labelled Rauzy diagrams so that $\pi'$ is obtained from $\pi$ by a sequence of genus-preserving simple extensions. 
\end{proposition}

\begin{proof}
	Suppose that the stratum containing $\calC'$ is connected: that is, equals $\calC'$. 
	In this case, the proposition follows from~\cite[Lemma~6.5]{Gut19} (if $\calC'$ is abelian) or from~\cite[Lemma 5.1]{Gut17} (if $\calC'$ quadratic).
	So, for the rest of the proof, we suppose that the stratum containing $\calC'$ is not connected. 
	
	Suppose that $\calC'$ is abelian.
	From this and the classification of components (\Cref{Thm:ClassificationAbelian}) we deduce that the genus $g$ of $\calC'$ is at least three. 
	If $\calC'$ is hyperelliptic then $\calC' = \calH^\hyp (g-1, g-1)$ (for $\calC'$ to be not basic).   
	The permutation $\left(\begin{smallmatrix} 1 & \alpha & \cdots & 2g \\ 2g & \cdots & \alpha & 1\end{smallmatrix}\right)$ lies in a labelled Rauzy diagram of $\calC'$.
	It is obtained by a simple extension of the permutation $\left(\begin{smallmatrix} 1 & \cdots & 2g \\ 2g & \cdots & 1\end{smallmatrix}\right)$ in a labelled Rauzy diagram of the basic component $\calH^\hyp(2g-2)$.
	Suppose now that $\calC'$ is non-hyperelliptic. 
	Thus $\calC'$ is even or odd. 
	By \cite[Lemma~6.4]{Gut19}, spin is preserved under simple extensions. 
	So, we obtain the proposition in this case. 
	
	Suppose that $\calC'$ is quadratic. 
	From this and the classification of components (\Cref{Thm:ClassificationQuadratic}) we deduce that the genus $g$ of $\calC'$ is at least two. 
	\begin{itemize}
		\item 
		Suppose that $\calC'$ has genus two. Then $\calC'$ is a component of $\calQ(6, -1, -1)$ or $\calQ(3, 3, -1, -1)$. 
		The explicit simple extensions used to both components of each of these strata are laid out in \Cref{table:extensions_genus_two}.
		\item 
		Suppose that $\calC'$ has genus $g$ at least three.
		\begin{itemize}
			\item 
			Suppose that $\calC'$ is non-hyperelliptic. 
			Suppose further that the stratum containing $\calC'$ has exactly two components, the other component being hyperelliptic. 
			Recalling that hyperellipticity cannot arise by simple extensions, we choose $\calQ(4g - 4)$ as our basic component.
			\item 
			Suppose that $\calC'$ is non-hyperelliptic in a stratum with three components. 
			Then there are finitely many possibilities for $\calC'$. 
			The cases where the basic component is abelian can be found in~\cite[Table~1]{Gut17}. 
			The remaining cases are found in \Cref{table:extensions_sporadic}.
			\item 
			Suppose now that $\calC'$ is hyperelliptic. 
			Since $\calC'$ is not basic, it has three or four singularities. 
			So we can use $\calH^\hyp(2g - 2)$ as the basic component -- see~\cite[Section~5.2]{Gut17}.
		\end{itemize}
	\end{itemize}
	This completes the proof. 
\end{proof}

\section{Zariski density}
\label{s:density}

We state one of our main results. 

\begin{theorem}
	\label{t:zariski}
	The symplectic monodromy groups and the symplectic Rauzy--Veech groups for all abelian components, and the plus and minus symplectic monodromy groups and the plus and minus symplectic Rauzy--Veech groups for all quadratic components, are Zariski dense in their ambient symplectic groups.
\end{theorem}

\noindent
In this section we deal with the minus piece. 
In \Cref{s:zariski-density-plus} we address the plus piece for the basic components other than $\calQ(5, -1)$, $\calQ(9, -1)^{\irreg}$, $\calQ(12)^{\reg}$ and $\calQ(12)^{\irreg}$.
These last four cases are handled in \Cref{s:sporadic}.

\subsection{What is already known}

Suppose that $\calC$ is an abelian component. 
Suppose that $\calD$ is a labelled Rauzy diagram for $\calC$.
By \cite[Theorem 1.1]{Avi-Mat-Yoc} and \cite[Theorem 1.1]{Gut19}, the symplectic Rauzy--Veech group $\Sp(\calD)$ (hence the symplectic monodromy group $\Sp(\calC)$) has finite index in $\Sp(S)$.
Hence $\Sp(\calD)$ is Zariski dense.

Suppose that $\calC'$ is a quadratic component with a labelled Rauzy diagram $\calD'$.
Suppose that a labelled permutation $\pi'$ in $\calD'$ can be obtained by a finite sequence of simple extensions from an abelian permutation $\pi$. Suppose that $\pi$ lies in a labelled Rauzy diagram of an abelian component $\calC$. 
Recall that $\Sp_+(R)$ is isomorphic to $\Sp(S)$.
By \Cref{c:RV_inclusion}, the plus Rauzy--Veech group $\Sp_+ (\calD')$ contains a subgroup isomorphic to $\Sp(\calD)$. 
Hence, it is also finite index in its ambient symplectic group.
Hence, the group $\Sp_+ (\calD')$ is Zariski dense.
By \Cref{fig:commutative_diagram_quadratic}, the symplectic monodromy group $\Sp_+ (\calC')$ is also Zariski dense. 

\subsection{Overview of remaining cases}

The other cases (the plus symplectic monodromy groups of the remaining quadratic components and the minus symplectic monodromy groups of all quadratic components) are more difficult.
In \Cref{s:minus}, we show that the Zariski density of the minus symplectic monodromy groups follows directly from  Filip's work~\cite{Fil17}. 
By \Cref{c:RV_finite_index_in_monodromy}, the minus Rauzy--Veech groups are also Zariski dense. 

In \Cref{s:zariski-density-plus}, we prove (again using Filip's results) Zariski density of the plus symplectic monodromy groups of all basic components.
By \Cref{c:RV_finite_index_in_monodromy}, their plus Rauzy--Veech groups are also Zariski dense. 
Using \Cref{prop:induction} and \Cref{c:RV_inclusion}, we obtain density of the plus groups for all quadratic components.

\subsection{Algebraic hulls and Zariski closures of monodromies}
\label{s:algebraic_hulls_Zariski_closures}

Suppose that $\calC_\root$ is a stratum component of rooted abelian differentials. 
Suppose that $\calM(S, Z)$ is the associated moduli space.
We denote the holomorphic part of $\upH^1(S, Z(S); \CC)$ by $\upH^{1,0}(S, Z(S); \CC)$. 
As real vector spaces, the space $\upH^{1,0}(S, Z(S); \CC)$ is isomorphic to $\upH^1(S, Z(S); \RR)$ by the isomorphism $[c] \mapsto \Re([c])$.
A differential $\omega$ in $\calC_\root$ gives a cohomology class $[\omega]$ in $\upH^1(S, Z(S); \CC)$.
The complex line $\CC [\omega]$ is identified to a plane in $\upH^1(S, Z(S); \RR)$ by the isomorphism, which is spanned by $[\Re(\omega)]$ and $[\Im(\omega)]$.
This plane is called the \emph{tautological plane}.
See \cite{Fil17}.

Locally, the map that sends differentials in $\calC_\root$ to their cohomology classes in $\upH^{1,0}(S, Z(S); \CC)$ gives period coordinates on $\calC_\root$. 

Suppose that $\calN$ is a linear invariant submanifold in $\calC_\root$.
Suppose that $\omega$ lies in $\calN$.
Locally, in period coordinates, the submanifold $\calN$ is defined by homogeneous linear equations with real coefficients.
The tangent space $T_{\omega} \calN$ can therefore be identified with this linear subspace of $\upH^{1,0}(S, Z(S); \CC)$. 
By using the aforementioned isomorphism, the tangent space $T_{\omega} \calN$ can be identified with a linear subspace of $\upH^1(S, Z(S); \RR)$. The tangent bundle $T \calN$ is thus isomorphic to a sub-bundle of the (relative real) Hodge bundle over $\calN$.

We may project the subbundle to the absolute real Hodge bundle $\upH^{\abs}_{\calN}$.
We denote the image by $p(T\calN)$.

Recall that a linear representation of a group $G$ on a vector space $V$ is \emph{strongly irreducible} if its restriction to every finite index subgroup of $G$ is irreducible.

The symplectic monodromy group $\Sp(\calN)$ acts by automorphisms on each fibre of $\upH^{\abs}_{\calN}$.
This splits $\upH^{\abs}_{\calN}$ into a direct sum of strongly irreducible symplectic sub-bundles.
The sub-bundle $p(T\calN)$ is strongly irreducible.

Suppose that $V$ is a strongly irreducible piece in $\upH^{\abs}_{\calN}$. 
The \emph{algebraic hull} for $V$ is the smallest algebraic subgroup of $\Sp (V)$ that the cocycle $\upC_{\KZ}$ (restricted to $V$) measurably conjugates into.

We describe Filip's results~\cite[Theorem~1.2 and Corollary~1.7]{Fil17} for possible 
algebraic hulls for linear invariant submanifolds and their implications here.

\begin{itemize}
	\item The Zariski closere of the symplectic monodromy group restricted to $p(T \calN)$ (or any of its Galois conjugates) is the full symplectic group and the restricted $\KZ$--cocycle has no zero exponents on $p(T \calN)$.
	See \cite[Corollary 1.7]{Fil17}.
	In \Cref{s:minus}, we use this to show the Zariski density of the minus symplectic monodromy group of any quadratic component.
	\item For the remaining strongly irreducible pieces, Filip shows that the Lie algebra representations of their algebraic hulls must be, up to compact factors, from the following list \cite[Theorem 1.2]{Fil17}:
	\begin{enumerate}
		\item $\mathfrak{sp}(2g, \RR)$ in the standard representation,
		\item $\mathfrak{su}(p, q)$ in the standard representation or $\mathfrak{su}(p,1)$ in any exterior power representation,
		\item $\mathfrak{so}(2n-1, 2)$ in the spin representation,
		\item $\mathfrak{so}^\ast(2n)$ in the standard representation, or 
		\item $\mathfrak{so}(2n-2, 2)$ in either of the spin representations.
	\end{enumerate} 
\end{itemize}

Moreover, Eskin--Filip--Wright show that the algebraic hull of a strongly irreducible piece that does not contain the tautological plane equals the Zariski closure of the symplectic monodromy group~\cite[Theorem 1.1]{Esk-Fil-Wri}.   
Thus, Filip's list also classifies Lie algebra representations of the Zariski closure of symplectic monodromy groups.

\subsection{Zariski density for the minus piece}
\label{s:minus} 

Suppose that $\calC_\root$ is a quadratic component.
Then $\calC'_\root$ is a linear invariant submanifold in $\calB_{\biroot}$.
By Filip's classification, the Zariski closure of the sympelctic monodromy group for $p(T \calC'_\root)$ is Zariski dense.
Finally, note that
\[
p (T \calC') = \upH_-^1(R; \ZZ)
\]
Hence, by Lefschetz duality, the minus symplectic monodromy group (acting on the dual $\upH_1^-(R, Z(R); \ZZ)$) is Zariski dense.

\section{Zariski density of the plus piece for the basic components}
\label{s:zariski-density-plus}

Suppose that $\calC$ is a quadratic stratum component.  
As usual, let $S$ be the underlying surface and let $R$ be the resulting branched double cover. 
Recall that $\upH_1(S; \ZZ)$ is isomorphic to $\upH_1^+(R; \ZZ)$; 
the isomorphism conjugates the corresponding monodromies.

Observe that the monodromy actions of $\calC$ on $\upH_1^+(R; \ZZ)$ and on $\upH_+^1(R; \ZZ)$ are isomorphic by Poincaré duality. 
We use $M$ to denote this abstract symplectic monodromy group. 
Note that $\upH_+^1(R; \ZZ)$ does not contain the tautological plane (see \Cref{s:algebraic_hulls_Zariski_closures}). 
By Eskin--Filip--Wright~\cite[Theorem 1.1]{Esk-Fil-Wri}, we may apply Filip's classification to the Zariski closure of $M$. 
Moreover, Treviño \cite[Theorem 1]{Tre} proved that the plus Lyapunov spectrum contains no zero exponents.

In this section,  we show that the action of $M$ is strongly irreducible.  
This done, we apply Filip's classification as follows.
Let $\mathfrak{m}$ be the Lie algebra of the Zariski closure of $M$.
By Treviño there are no zero exponents, so $\mathfrak{m}$ is one of the following (up to compact factors):
\begin{enumerate}
	\item 
	$\mathfrak{sp}(2n, \RR)$ in the standard representation (degree $2n$, dimension $n(2n + 1)$);
	\item 
	$\mathfrak{su}(n, n)$ in the standard representation (degree $4n$, dimension $4n^2 - 1$);
	\item 
	$\mathfrak{so}(2n-1,2)$ in the spin representation (degree $2^n$, dimension $n(2n+1)$);
	\item 
	$\mathfrak{so}(2n-2,2)$ in one of the spin representations (degree $2^{n-1}$, dimension $n(2n-1)$); or
	\item 
	$\mathfrak{so}^*(2n)$ in the standard representation for even $n$ (degree $4n$, dimension $n(2n-1)$).
\end{enumerate}

This list can be refined further. 
Observe that, by strong irreducibility, the degree of the representation must equal $2g$, the rank of $\upH_1(S)$.
Thus we have the following:
\begin{enumerate}
	\item
	$2n = 2g$, so $\dim_{\RR} \mathfrak{sp}(2g, \RR) = g(2g + 1)$; \label{i:sp}
	\item 
	$4n = 2g$, so $\dim_{\RR} \mathfrak{su}(g/2, g/2) = g^2 - 1$; \label{i:su}
	\item 
	$2^n = 2g$, so $\dim_{\RR} \mathfrak{so}(2n-1,2) = \log_2(2g) \log_2(8g^2)$; \label{i:so odd}
	\item 
	$2^{n-1} = 2g$, so $\dim_{\RR} \mathfrak{so}(2n-2,2) = \log_2(4g) \log_2(8g^2)$; and \label{i:so even}
	\item 
	$4n = 2g$, so $\dim_{\RR} \mathfrak{so}^*(g/2, g/2) = g(g - 1)/2$. \label{i: so star}
\end{enumerate}

\begin{corollary}
	\label{c:Zarisiki_density_odd_genus}
	Suppose that $\calC$ is a quadratic component. 
	Suppose that the genus of the underlying surface is odd. 
	Suppose that the action of the symplectic monodromy group $M$ on $\upH_1(S; \RR)$ is strongly irreducible. 
	Then $\mathfrak{m} = \mathfrak{sp}(2g, \RR)$.
\end{corollary}

\begin{proof}
	The possibilities \eqref{i:su}--\eqref{i: so star} require the genus $g$ be even (or, in fact, a power of two). 
	Thus, assuming that the action of $M$ is strongly irreducible, if $g$ is odd then $\mathfrak{m}$ has to be $\mathfrak{sp}(2g, \RR)$. 
\end{proof}

Dealing with the case of even genus takes up the rest of \Cref{s:zariski-density-plus} and \Cref{s:sporadic}.
For each basic component we first prove that the action of $M$ is strongly irreducible.
We then give various delicate arguments to provide lower bounds for the dimension of $\mathfrak{m}$, the Lie algebra of the Zariski closure of $M$.
(This requires the additional flexibility given by \Cref{t:flow-general}.)
In this way, we prove that $\mathfrak{m}$ is the symplectic representation.

\subsection{Dehn twists and other tools}

We begin with a simple lemma. 

\begin{lemma}
	\label{l:cylinder}
	Suppose that $\calC$ is any stratum component.
	Suppose that $q$ is a differential in $\calC$.
	Suppose that $A$ is a flat cylinder in $q$. 
	Then the Dehn twist in $A$ lies in the symplectic monodromy group of $\calC$. \qed 
\end{lemma}

Suppose now that $c, d$ are oriented simple closed curves in $S$. 
(In a small abuse of notation we also use $c$ and $d$ to denote the corresponding cycles in homology.)
Suppose that $\omega(c, d)$ is their algebraic intersection number. 
Suppose that $T(c)$ is the right Dehn twist about $c$.  
(In small abuses of notation, we also use $T(c)$ to denote the corresponding action on homology.)
Then, in homology, we have the following:
\begin{equation}\label{e:twist}
	T(c)(d) = d + \omega(c, d) \cdot c
\end{equation}

The Dehn twist $T(c)$ acts on homology as a symplectic transvection.
Inspired by this, we define a linear map $D(c) \in \End(\upH_1(S; \RR))$ by
\begin{equation}
	\label{e:defn_D}
	D(c) (d) = T(c) (d) - d = \omega(c, d) \cdot c
\end{equation}
Note that $k D(c)(d) = T(c)^k(d) - d$.  

\begin{lemma}
	\label{l:lie_algebra_nilpotents}
	Suppose that $T(c)^k$ lies in $M$ for $k \in \NN$. 
	Then $k D(c)$ lies in $\mathfrak{m}$. \qed
\end{lemma}

The next lemma shows how Dehn twists in $M$ force invariant subspaces to ``expand''.

\begin{lemma}
	\label{l:twist}
	Suppose that $V$ is a non-trivial subspace of $\upH_1(S; \RR)$.
	Suppose further that $V$ is invariant under some finite-index subgroup of $M$.
	Suppose that $c, d \in \upH_1(S; \RR)$ have $\omega(c, d) \neq 0$.
	Suppose that $T(c)$ lies in $M$ and $d$ lies in $V$.
	Then $c$ also lies in $V$.
\end{lemma}

\begin{proof}
	Fix $k > 0$ so that $T(c)^k$ leaves $V$ invariant. 
	Thus $T(c)^k(d)$ lies in $V$. 
	By \Cref{e:twist}, we have that $k D(c)(d) = T(c)^k(d) - d$ is a non-zero multiple of $c$.
	Thus $c$ lies in $V$ and we are done. 
\end{proof}

We now give a criterion for an action to be strong irreducible.

\begin{definition}
	\label{d:chain}
	Suppose that $B$ is a collection of oriented simple closed curves. 
	Suppose that $c$ and $d$ lie in $B$.  
	Then a \emph{chain from $c$ to $d$ in $B$} is a sequence $(c = c_0, c_1, \ldots, c_n = d)$ in $B$ so that $\omega(c_i, c_{i+1}) \neq 0$ for all $i$.   
\end{definition}

For such a collection $B$ we write $[B] = \{[c] \in \upH_1(S; \RR) \st c \in B\}$.

\begin{lemma}
	\label{l:irreducible}
	Suppose that $B$ is a collection of oriented simple closed curves. 
	Suppose that 
	\begin{itemize}
		\item $[B]$ spans $\upH_1(S; \RR)$,
		\item all $c, d \in B$ are connected by a chain in $B$, and
		\item $T(c)$ lies in $M$ for all $c \in B$.
	\end{itemize}
	Then the action of $M$ on $\upH_1(S; \RR)$ is strongly irreducible.
\end{lemma}

\begin{proof}
	Suppose that $V$ is a non-trivial subspace preserved by some finite index subgroup $N < M$. 
	We must show that $V = \upH_1(S; \RR)$.
	It suffices to prove that $V$ contains the classes $[B]$. 
	
	Fix any non-zero $v$.  
	Since $[B]$ spans, there is some $c \in B$ so that $\omega(c, v) \neq 0$. 
	So, by \Cref{l:twist}, the cycle $c$ lies in $V$. 
	Since all pairs in $B$ are connected by chains, applying \Cref{l:twist} inductively, we deduce that $[B] \subset V$, as desired. 
\end{proof}

\begin{remark}
	In some cases, $[B]$ is, in fact, a basis for $\upH_1(S; \RR)$. 
	However, this is not always possible to arrange;
	the chain hypothesis may force us to add extra curves.
\end{remark}

Suppose that $c$ and $d$ are curves in $S$.  
We say they form a \emph{symplectic block} if $\omega(c, d) = 1$. 
We use the following two lemmas to produce new Dehn twists in $M$ from old ones. 
The first one is well-known; 
the second one is a slightly less general version of \cite[Corollary 2.10]{Gut19}:

\begin{lemma}
	\label{l:conjugation}
	Suppose that $c$ and $d$ are curves in $S$ forming a symplectic block. 
	If $T(c), T(d) \in M$, then $T(c + d) \in M$.
\end{lemma}

\begin{proof}
	We have that $T(c) T(d) T(c)^{-1} = T(T(c)(d)) = T(c + d) \in M$.
\end{proof}

\begin{lemma}
	\label{l:two_blocks}
	Suppose $c_1, c_2, d_1, d_2$ are curves in $S$ so that $(c_i, d_i)$ are a symplectic blocks for each $i \in \{1, 2\}$. 
	Let $b$ be another curve on $S$ such that $\omega(c_i, b) = \omega(d_i, b) = \pm 1$ for each $i \in \{1, 2\}$. 
	If $T(c_i), T(d_i) \in M$ for each $i \in \{1, 2\}$, then $T(c_1 + c_2)^2 \in M$. \qed
\end{lemma}

\subsection{Minimal strata}
\label{s:minimal strata}

For $g \geq 5$, the strata $\calQ(4g - 4)$ are connected. 
For $g =4$, the stratum $\calQ(12)$ has components $\calQ(12)^\reg$ and $\calQ(12)^\irreg$.
Let $d = 2g$ and consider
\[
\pi_d =
\begin{pmatrix*}
	1 & 2 & 1 & 3 & 4 & 5 & 6 & 7 & 8 & \cdots & d-1 \\
	2 & 4 & 3 & 6 & 5 & 8 & 7 & \cdots & d & d-1 & d
\end{pmatrix*}.
\]
The rooted differentials in $\calC(\pi_d)$ belong to $\calQ(4g - 4)$ (and to $\calQ(12)^\reg$ when $g =4$).

We provide a set $B$ satisfying the hypothesis of \Cref{l:irreducible}. 
The precise choice of curves in $B$ depends on the parity of $g$.  

\begin{figure}
	\centering
	\begin{subfigure}[b]{\textwidth}
		\centering
		\includegraphics[scale=0.8]{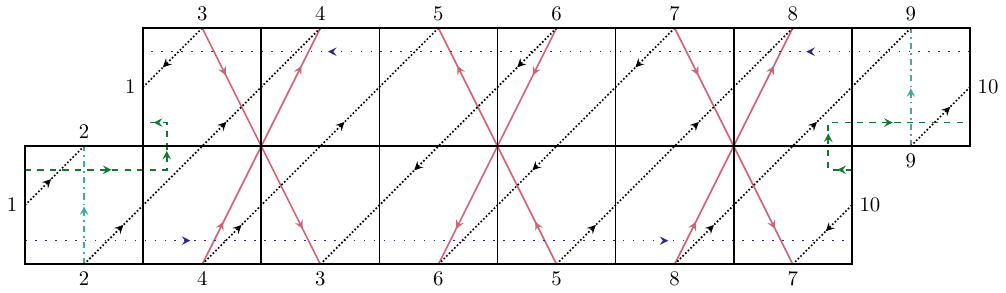}
		\caption{Odd $g$.}
		\label{f:odd-minimal}
	\end{subfigure}
	
	\begin{subfigure}[b]{\textwidth}
		\centering
		\includegraphics[scale=0.8]{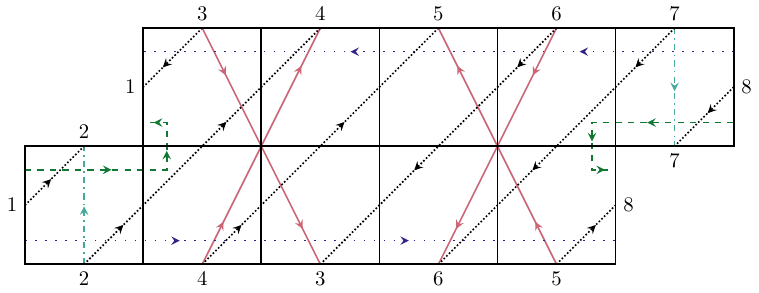}
		\caption{Even $g$ with the curve $b$ having slope one.}
		\label{f:even-minimal-1}
	\end{subfigure}
	
	\begin{subfigure}[b]{\textwidth}
		\centering
		\includegraphics[scale=0.8]{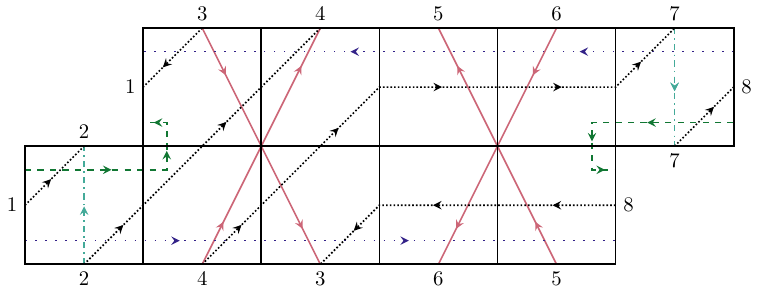}
		\caption{Even $g$ with the modified curve $b'$.}
		\label{f:even-minimal-2}
	\end{subfigure}
	\caption{Basis and useful curves.}
	\label{f:minimal}
\end{figure}

\begin{definition}
	\label{d:basis_minimal}
	Let $P_d$ be the square-tiled polygon with $2(d-3)$ on-axis squares glued as follows:
	\begin{itemize}
		\item there are two rows of $d-3$ squares and
		\item the top row begins one square to the right of the bottom row.
	\end{itemize}
	The left-bottom and right-top vertices partition the boundary of $P_d$ into two arcs, each consisting of $d$ sides of squares. 
	We call these arcs the \emph{top} and \emph{bottom} of $P_d$ and label them according to the rows of $\pi_d$. 
	Gluing sides of $P_d$ according to their labels gives the underlying surface $S$.
	See \Cref{f:minimal} for examples.
	
	Let $c_i$ be the simple closed curve in $S$ meeting the side labelled $i$ once transversely and no other sides.
	For $2 < i < 2d - 1$ we orient $c_i$ from the bottom to the top of $P_d$ if and only if $i$ is congruent to zero or one modulo four.
	We orient $c_1$ from the left to the right and $c_2$ from the bottom to the top. 
	We do the same for curves $c_{d}$ and $c_{d-1}$ if $g$ is odd; we use the opposite orientations if $g$ is even.
	(For examples, see the dashed, dashed-dotted, and solid curves in \Cref{f:minimal}.)
	
	Let $b$ and $p$ be curves representing the following sums in homology:
	\[
	b = \sum_{i = 1}^d c_i, \qquad p = c_1 + (-1)^g c_d
	\]
	When $g$ is even, let $b'$ be a curve representing the following sum in homology:
	\[
	b' = \sum_{i=1}^{d-4} c_i - c_{d-1} - c_d
	\]
	(The curve $b$ is densely dotted and has slope one in Figures~\ref{f:odd-minimal} and~\ref{f:even-minimal-1}.
	The curve $p$ is loosely dotted and is horizontal in all three figures.
	The curve $b'$ is densely dotted and has slope one (in all symplectic blocks except the last) in \Cref{f:even-minimal-2}.)
	
	Finally, we set 
	\[
	B = \begin{cases}
		\{c_2, \dotsc, c_{d-1}, b, p\}         & \text{if $g$ is odd}  \\
		\{ c_2, \dotsc, c_{d-1}, b, b', p\}    & \text{if $g$ is even}
	\end{cases} \qedhere
	\]
\end{definition}

Note that $[B]$ is a spanning set for $\upH_1(S; \RR)$.
It is a basis when $g$ is odd.

\begin{lemma}
	\label{l:irreducible-minimal}
	The action of $M$ on $\upH_1(S; \RR)$ is strongly irreducible.
\end{lemma}

\begin{proof}
	We check that the set $B$ defined in \Cref{d:basis_minimal} satisfies the hypothesis of \Cref{l:irreducible}.
	
	As seen in Figures~\ref{f:odd-minimal} and~\ref{f:even-minimal-1}, the curves $\{ c_2, \dotsc, c_{d-1}, b\}$ are core curves of flat cylinders.
	By \Cref{l:cylinder}, the Dehn twists in these curves lie in $M$.
	
	Now consider the flat cylinder with core curve $c_{d-3} - c_{d-2}$.  
	A left-handed shear applied inside this cylinder makes the four horizontal saddle connections have slope one.  
	Equivalently, it performs a one-quarter Dehn twist. 
	This straightens the curve $b'$; see \Cref{f:even-minimal-2}. 
	Thus, $b'$ is the core curve of a flat cylinder on a different differential. 
	By \Cref{l:cylinder}, the Dehn twist $T(b')$ lies in $M$.
	
	Thus, for any curve $u$ in $B$, the Dehn twist $T(u)$ is in $M$.
	We also note that for any pair $c$ and $d$ in $B$ there exists a chain connecting them.
	Thus the set $B$ satisfies the hypothesis of \Cref{l:irreducible}, as desired.
\end{proof}

From \Cref{c:Zarisiki_density_odd_genus}, we deduce the following:

\begin{corollary}
	\label{c:Zariski_minimal_odd_genus}
	For odd genus $g \geq 3$, we have that $\mathfrak{m} = \mathfrak{sp}(2g, \RR)$. \qed
\end{corollary}

We will use the following estimate to handle the case where $g$ is even, except for $g = 4$ that is done in \Cref{s:sporadic}.

\begin{lemma}
	\label{l:dimension_bound_minimal}
	For any $g \geq 3$, we have that
	\[
	\dim_{\RR} \mathfrak{m} \geq 2g - 1 + \binom{2g - 4}{2} = 2g^2 - 7g + 9.
	\]
\end{lemma}

\begin{proof}
	The cycles $c_2, \dotsc, c_{d-1}$ and $p$ are realised as core curves of flat cylinders.
	Hence, the Dehn twists $T(c_i)$ for $i = 2, \dotsc, d-1$ and $T(p)$ are in $M$. 
	With $D(c)$ as defined in \Cref{e:defn_D} we have the following:
	\[
	\begin{aligned}[c]
		D(c_2)(c_1) &= -c_2 \\
		D(c_{2i+1})(c_{2i+2}) &= c_{2i+1} \\
		D(c_{2i+2})(c_{2i+1}) &= -c_{2i+2} \\
		D(c_{d-1})(c_d) &= -c_{d-1} \\
		D(p)(c_2) = p, \, D(p)(c_{d-1}) &= (-1)^g p
	\end{aligned}
	\qquad
	\begin{aligned}[c]
		D(c_2)(c_i) &= 0 \text{ for } i \neq 1 \\
		D(c_{2i+1})(c_j) &= 0 \text{ for } j \neq 2i+2 \\
		D(c_{2i+2})(c_j) &= 0 \text{ for } j \neq 2i+1 \\
		D(c_{d-1})(c_j) &= 0 \text{ for } j \neq d \\
		D(p)(c_j) &= 0 \text{ for } j \notin \{2, d-1\}
	\end{aligned}
	\]
	
	\begin{claim*}
		For each $3 \leq i < j \leq 2g - 2$, the element $T(c_i + c_j)^2$ belongs to $M$.
	\end{claim*}
	
	\begin{proof}
		\Cref{l:conjugation} implies that $T(c_{2k+1} + c_{2k+2})^2$ lies in $M$.
		In all other cases, the result follows from \Cref{l:two_blocks}.
	\end{proof}
	
	We define $E(c) = 2 D(c)$ and observe the following:
	\begin{align*}
		E(c_{2i+1} + c_{2j+1})(c_{2i+2}) &= 2(c_{2i+1} + c_{2j+1}) \\
		E(c_{2i+1} + c_{2j+1})(c_{2j+2}) &= 2(c_{2i+1} + c_{2j+1}) \\
		E(c_{2i+1} + c_{2j+2})(c_{2i+2}) &= 2(c_{2i+1} + c_{2j+2}) \\
		E(c_{2i+1} + c_{2j+2})(c_{2j+1}) &= -2(c_{2i+1} + c_{2j+2}) \\
		E(c_{2i+2} + c_{2j+2})(c_{2i+1}) &= -2(c_{2i+2} + c_{2j+2}) \\
		E(c_{2i+2} + c_{2j+2})(c_{2j+1}) &= -2(c_{2i+2} + c_{2j+2})
	\end{align*}
	
	We make the following definitions:
	\begin{align*}
		\mathfrak{d} &= \{D(c_i) \st 2 \leq i \leq d-1\} \cup \{D(p)\} \\
		\mathfrak{e} &= \{E(c_i + c_j) \st 3 \leq i < j \leq 2g - 2\} \\
		\mathfrak{f} &= \mathfrak{d} \cup \mathfrak{e}
	\end{align*}
	By \Cref{l:lie_algebra_nilpotents}, we have that $\mathfrak{f}$ is a subset of $\mathfrak{m}$.
	Note that the cardinality of $\mathfrak{d}$ is $2g - 1$ and the cardinality of $\mathfrak{e}$ is $\binom{2g - 4}{2}$. 
	Hence, the cardinality of $\mathfrak{f}$ is $2g - 1 + \binom{2g - 4}{2} = 2g^2 - 7g + 9$.
	
	\begin{claim*}
		$\mathfrak{f}$ is linearly independent in $\mathfrak{m}$.
	\end{claim*}
	
	\begin{proof}
		Using the (symplectic) basis $\{c_1, \dotsc, c_d\}$  we identify the vector space of linear transformations of $\upH_1(S; \RR)$ with the square matrices $\Mat_d(\RR)$.
		We may then associate a matrix to each endomorphism in $\mathfrak{f}$. 
		Let $M_{i,j}$ be the matrix with the $(i,j)$--entry one and all other entries zero. 
		Using the calculations above, we express each endomorphism in $\mathfrak{f}$ as a linear combination of $M_{i,j}$ as follows:
		\begin{itemize}
			\item The matrix for $D(c_{2i+1})$ is $M_{(2i+1), (2i+2)}$.
			\item The matrix for $D(c_{2i+2})$ is $-M_{(2i+2), (2i+1)}$.
			\item The matrix $M_{(2i+2), (2j+1)}$ features only in the linear combination of the matrix for $E(c_{2i+1} + c_{2j+1})$.
			\item The matrix $M_{(2i+2), (2j+2)}$ features only in the linear combination of the matrix for $E(c_{2i+1} + c_{2j+2})$.
			\item The matrix $M_{(2i+1), (2j+2)}$ features only in the linear combination of the matrix for $E(c_{2i+2} + c_{2j+2})$.
		\end{itemize}
		The claim follows.
	\end{proof}
	\noindent
	The estimate for $\dim_{\RR} \mathfrak{m}$ follows.
\end{proof}

For even $g \geq 6$, the estimate in the previous lemma is enough to establish Zariski-density:

\begin{corollary}
	\label{c:Zariski_density_six_above}
	For even genus $g \geq 6$, we have $\mathfrak{m} = \mathfrak{sp}(2g, \RR)$.
\end{corollary}

\begin{proof}
	Since $\dim_{\RR} \mathfrak{m} \geq 2g^2 - 7g + 9$, it follows that if $g \geq 6$ then
	\begin{align*}
		\dim_{\RR} \mathfrak{m} &> \dim_{\RR} \mathfrak{su}(g/2,g/2) = g^2 -1\\
		\dim_{\RR} \mathfrak{m} &> \dim_{\RR} \mathfrak{so}(2n-1,2) = \log_2(2g) \log_2(8g^2) \text{ for } n = \log_2(2g)\\
		\dim_{\RR} \mathfrak{m} &> \dim_{\RR} \mathfrak{so}(2n-2,2) = \log_2(4g) \log_2(8g^2) \text{ for } n = \log_2(4g)\\
		\dim_{\RR} \mathfrak{m} &> \dim_{\RR} \mathfrak{so}^*(g/2) = g(g - 1)/2.
	\end{align*}
	The corollary follows.
\end{proof}

Zariski density for $\calQ^\reg(12)$ and $\calQ^\irreg (12)$ (the remaining cases) is shown \Cref{s:simplicity}. 

\subsection{Hyperelliptic components with two singularities}
\label{s:hyperelliptic with two}

Let $r, s \geq 1$ be odd positive integers.
We give a permutation $\pi_{r,s}$ that represents the component $\calQ(2r, 2s)^{\hyp}$. 
In particular, we may assume $s \leq r$.
The permutation $\pi_{r,s}$ is defined by the top row being $(\alpha, 0, 1 \cdots, r-1, \alpha, r, \cdots, r+s-1, r+s)$ and the bottom row being $(r+s, r+s-1, \cdots, r, \beta, r-1, \cdots, 1, 0, \beta)$.
Note that $\pi_{r,s}$ is symmetric under rotation by $\pi$.
When $s < r$ the permutation $\pi_{r,s}$ looks as follows:
\[
\begin{pmatrix*}
	\alpha & 0 & \cdots & s & s+1 & \cdots & r-1 & \alpha & r & \cdots & r+s \\
	r+s & \cdots & r & \beta & r-1 & \cdots & s+1 & s & \cdots & 0 & \beta \\
\end{pmatrix*}
\]
When $s = r$, the permutation $\pi_{r,s}$ looks as follows:
\[
\begin{pmatrix*}
	\alpha & 0 & \cdots & r-1 & \alpha & r & \cdots & 2r-1 & 2r \\
	2r & 2r-1 & \cdots & r & \beta & r-1 & \cdots & 0 & \beta \\
\end{pmatrix*}
\]

See \Cref{f:hyp_s_neq_r} for a flat surface in $\calQ(6, 2)^{\hyp}$ (with $r = 3$ and $s = 1$) and \Cref{f:hyp_s_eq_r} for a flat surface in $\calQ(6, 6)^{\hyp}$ (with $r = 3$ and $s = 3$). 
The genus of the underlying surface is $g = (r + s + 2)/2$.

\begin{figure}
	\begin{subfigure}[b]{\textwidth}
		\centering
		\includegraphics[scale=0.8]{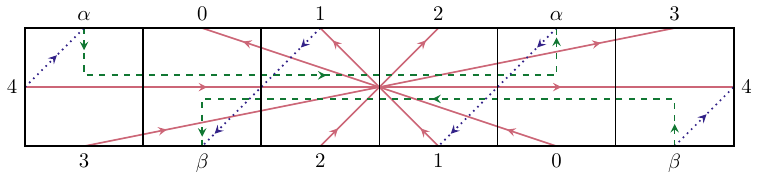}
		\caption{When $s < r$.}
		\label{f:hyp_s_neq_r}
	\end{subfigure}
	
	\begin{subfigure}[b]{\textwidth}
		\centering
		\includegraphics[scale=0.8]{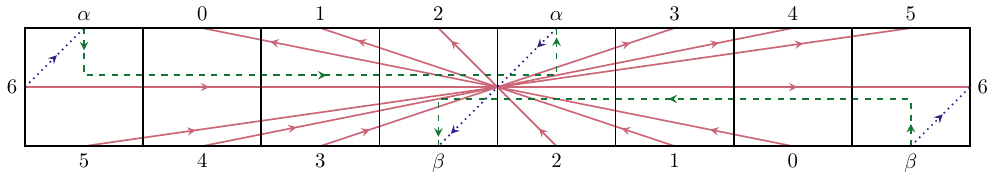}
		\caption{When $s = r$.}
		\label{f:hyp_s_eq_r}
	\end{subfigure}
	\caption{Curves for the hyperelliptic case.}
	\label{f:hyperelliptic}
\end{figure}

\begin{definition}
	\label{d:basis_hyperelliptic}
	Let $P_{r,s}$ be a square-tiled polygon with $r + s + 2$ on-axis squares arranged in a row.
	The line segment joining the left-top vertex and the right-bottom vertex partitions the boundary of $P_{r,s}$ into two arcs, each consisting of $r + s + 3$ sides of squares.
	We call these arcs the \emph{top} and \emph{bottom} of $P_{r,s}$ and label them according to the rows of $\pi_{r,s}$. 
	Gluing sides of $P_{r,s}$ according to their labels gives the underlying surface $S$.
	See \Cref{f:hyperelliptic} for two examples.
	
	Let $c_i$ (respectively $c_\alpha$ and $c_\beta$) be the simple closed curve in $S$ that meets the side labelled $i$ (respectively $\alpha$ and $\beta$) once transversely and no other sides.
	We orient $c_i$ from the bottom to the top of $P_{r,s}$. 
	We orient $c_\alpha$ from left to right and $c_\beta$ from right to left.
	(For examples, see the solid and dashed curves in \Cref{f:hyperelliptic}.)
	
	Let $c_{\alpha \beta}$ be the simple closed curve in $P_{r,s}$ with slope one that intersects the edges labelled $\alpha$ and $\beta$ transversely and once.
\end{definition}

Observe that the union of $c_\alpha$ and $c_\beta$ (with the given orientations) is the boundary of a sub-surface.
Thus $c_\alpha + c_\beta = 0$ in absolute homology. Using this, we see that
\begin{equation}
	\label{e:c_alphabeta}
	c_{\alpha\beta} =
	\begin{cases}
		-c_s + c_{r+s} - 2c_\alpha & s < r \\
		c_{2r} - 2c_\alpha & s = r
	\end{cases}
\end{equation}
The set $B' = \{c_0, \cdots, c_{r+s}, c_\alpha\}$ is a basis for $\upH_1(S;\RR)$; 
to see this, note that the matrix of algebraic intersection numbers has determinant one. We will use this basis several times in this section.

\begin{lemma}
	\label{l:strong_irreducible_hyperelliptic}
	The action of $M$ on $\upH_1(S; \RR)$ is strongly irreducible.
\end{lemma}

\begin{proof}
	Since $B'$ is a basis of $\upH_1(S; \RR)$, from \Cref{e:c_alphabeta} we see that the set $B = \{c_0, \cdots, c_{r+s}, c_{\alpha \beta} \}$ is also a basis.
	Note that
	\begin{itemize}
		\item all curves in $B$ are core curves of flat cylinders, 
		\item any pair of cycles in $\{c_0, \dotsc, c_{r+s}\}$ intersect, and
		\item $c_{\alpha \beta}$ intersects $c_0$.
	\end{itemize}
	By \Cref{l:cylinder}, for all $c \in B$, the twist $T(c)$ lies in $M$.
	Thus $B$ satisfies the hypotheses of \Cref{l:irreducible} and we are done.
\end{proof}

\begin{lemma}
	\label{l:dimensio_bound_hyperelliptic}
	For any odd positive integers $r \geq s \geq 1$, we have that 
	\[
	\dim_{\RR} \mathfrak{m} \geq 2g - 1 + \binom{2g - 1}{2} = g(2g - 1) + 1
	\]
	where $g = (r + s + 2)/2$.
\end{lemma}

\begin{proof}
	Observe that the curves $c_i$, for $0 \leq i \leq r + s$, and $c_{\alpha\beta}$ are core cores of flat cylinders (as can be observed in \Cref{f:hyperelliptic}). Hence, the twists $T(c_i)$ and $T(c_{\alpha\beta})$ lie in $M$. By \Cref{l:conjugation}, the twist $T(c_i + c_j)$ also lies in $M$ for each $0 \leq i < j \leq r + s$.
	Recall that we define $D(c) = T(c) - \Id$ in \Cref{e:defn_D}. 
	By \Cref{l:lie_algebra_nilpotents}, the elements $D(c_i)$, $D(c_i + c_j)$ and $D(c_{\alpha\beta})$ all lie in $\mathfrak{m}$.
	
	We now use the basis $B' = \{ c_0 , c_1, \dotsc, c_{r+s}, c_\alpha\}$ to identify the linear transformations of $\upH_1(S; \RR)$ with the square matrices in $\Mat_{2g}(\RR)$. 
	For $0 \leq i, j \leq r+s$ let $M_{i,j}$ be the matrix with $(i,j)$--entry one and remaining entries zero. 
	Similarly, we define $M_{i, \alpha}$ and $M_{\alpha, \alpha}$.
	
	If $P(i,j)$ is a logical proposition on $i$ and $j$, we define
	\[
	\llbracket P(i, j) \rrbracket = \begin{cases}
		1 & \text{if } P(i,j) \text{ is true} \\
		0 & \text{if } P(i,j) \text{ is false}.
	\end{cases}
	\]
	
	For any $0 \leq i, k \leq r + s$ we have 
	\[
	D(c_i)(c_k) = \begin{cases}
		c_i & k < i \\
		0   & k=i \\
		-c_i & k > i
	\end{cases} \qquad \mbox{and} \qquad
	D(c_i)(c_\alpha) = \begin{cases}
		-c_i & i < r \\
		0 & i \geq r
	\end{cases}
	\]
	Hence, for each $0 \leq i \leq r + s$,
	\begin{equation}
		\label{e:D(c_i)}
		D(c_i) = \sum_{k < i} M_{i, k} - \sum_{k > i} M_{i, k} - \llbracket i < r \rrbracket  M_{i, \alpha}.
	\end{equation}
	
	For any $0 \leq i < j \leq r + s$ and $0 \leq k \leq r + s$, let $c_{ij} = c_i + c_j$. Then,
	\[
	D(c_{ij})(c_k) = \begin{cases}
		2c_{ij} & k < i \\
		c_{ij} & k = i \\
		0 & i < k < j \\
		-c_{ij} & k = j \\
		-2c_{ij} & k > j
	\end{cases} \qquad \mbox{and} \qquad
	D(c_{ij})(c_\alpha) =  \begin{cases}
		-2c_{ij} & j < r \\
		-c_{ij} & i < r \leq j \\
		0 & i \geq r
	\end{cases}
	\]
	Hence, for each $0 \leq i < j \leq r + s$,
	\begin{equation}
		\label{e:D(cij)}
		\begin{aligned}
			D(c_i+c_j) ={}&  (M_{i,i} + M_{j, i}) - (M_{i, j} + M_{j,j}) \\
			&{} + 2\sum_{k < i}  (M_{i, k} + M_{j, k}) - 2 \sum_{j <  k} (M_{i, k} + M_{j,k}) \\
			& {} - (2 \llbracket j < r \rrbracket +  \llbracket i < r \leq j \rrbracket) (M_{i, \alpha} + M_{j, \alpha})
		\end{aligned}
	\end{equation}
	
	Finally, by \Cref{e:c_alphabeta}, we deduce that the $\alpha$--th row of the matrix $D(c_{\alpha\beta})$ does not vanish.  
	
	We set
	\[
	\mathfrak{d} = \{D(c_i) \st 0 \leq i \leq r + s\} \cup \{D(c_i + c_j) \st 0 \leq i < j \leq r + s\} \cup \{D(c_{\alpha\beta})\}
	\]

	\begin{claim*}\label{l:independence}
		$\mathfrak{d}$ is linearly independent in $\mathfrak{m}$.
	\end{claim*}
	
	\begin{proof}
		
		Let $0 \leq i \leq r + s$. Let $W$ be the $i$--th row of $D(c_i)$. For $i + 1 \leq j \leq r + s$, let $V_j$ be the $i$--th row of $D(c_i + c_j)$. We define $\mathfrak{R}_i = \{W, V_{i+1}, \dotsc, V_{r+s}\}$. The crux of the argument is the following:
		
		\begin{subclaim*}
			\label{c:vectors}
			$\mathfrak{R}_i$ is linearly independent.
		\end{subclaim*}
		
		\begin{proof}
			We denote the standard basis in $\RR^{2g}$ by $\{ e_0 , \dotsc , e_{r+s}, e_\alpha \}$. 
			From \Cref{e:D(c_i)}, we deduce
			\[
			W = \sum_{k < i} e_k - \sum_{k > i} e_k - \llbracket i < r \rrbracket e_\alpha
			\]
			From \Cref{e:D(cij)} we deduce for $i + 1 \leq j \leq r + s$
			\[
			V_j = 2 \sum_{k < i} e_k + e_i - e_j - 2 \sum_{j < k} e_k - (2\llbracket j < r\rrbracket + \llbracket i < r \leq j\rrbracket)e_\alpha
			\]
			
			We show linear independence by a Gaussian elimination argument. 
			For $i+1 \leq j < r+s$, we eliminate $e_k$ for $k > j$ from the linear combination for $V_j$, as follows. 
			Consider the vector
			\[
			V'_j = V_j + 2\sum_{k = j+1}^{r+s} (-1)^{k+j} V_k
			\]
			In terms of the basis, we obtain
			\[
			V'_j = 2(-1)^{j+1} \sum_{k < i} e_k + (-1)^{j+1} e_i + e_j + 2(-1)^j \llbracket j < r\rrbracket e_\alpha + (-1)^j \llbracket i < r \leq j \rrbracket e_\alpha.
			\]
			We also let $V'_{r+s} = V_{r+s}$.
			
			We apply a similar process to $W$, by eliminating all instances of $e_k$ for $k > i$.
			Let $W'$ be the vector $W - V'_{r+s} - V'_{r+s-1} - \dotsb - V'_{i+1}$. 
			We obtain 
			\[
			W' = \begin{cases}
				-\sum_{k < i} e_k + \llbracket i < r\rrbracket e_\alpha & \text{ if $i$ is even} \\
				+\sum_{k \leq i} e_k - \llbracket i < r\rrbracket e_\alpha  & \text{ if $i$ is odd}.
			\end{cases}
			\]
			
			Observe that 
			\begin{itemize}
				\item for $i+1 \leq j \leq r+s$, the vector $e_j$ features only in the linear combination for $V'_j$, and
				\item the vector $W'$ is nonzero and features none of the $e_j$ for $i+1 \leq j \leq r+s$.
			\end{itemize}
			The vectors $\{W', V'_{i+1}, \cdots , V'_{r+s} \}$ are thus linearly independent, as desired.
		\end{proof}
		
		A linear combination
		\[
		L = \sum_{i = 0}^{r + s} a_i D(c_i) + \sum_{i = 0}^{r + s} \sum_{j = i + 1}^{r + s} a_{i, j} D(c_i + c_j) + a_{\alpha\beta} D(c_{\alpha\beta})
		\]
		of matrices in $\mathfrak{d}$ can be regrouped as $\sum_{i = 0}^{r+s} S_i + a_{\alpha\beta} D(c_{\alpha\beta})$ where
		\[
		S_i = a_i D(c_i) + \sum_{j = i+1}^{r+s} a_{i,j} D(c_i+ c_j)
		\]
		In particular, when $i = r+s$ we have $S_{r+s} = a_{r+s} D(c_{r+s})$. Assume that $L = 0$. 
		
		The $\alpha$--th row of $D(c_{\alpha\beta})$ is non-zero and of every other matrix in $\mathfrak{d}$ is zero. 
		It follows that $a_{\alpha\beta} = 0$.
		Thus $L = \sum_{i = 0}^{r+s} S_i$.
		We prove by induction that the remaining coefficients are also zero.
		
		The $0$--th row vanishes for all $S_j$ with $j> 0$.
		Thus $0$--th row of $L$ equals the $0$--th row of $S_0$. 
		Hence, by restricting to the $0$--th row, we obtain a linear combination of the elements of $\mathfrak{R}_0$ that equals zero. 
		By the subclaim, we deduce that $a_0 = 0$ and $a_{0, j} = 0$ for each $j > 0$. 
		Thus, $S_0 = 0$ and $L = \sum_{j = 1}^{r+s} S_j$.
		
		Now suppose that all coefficients in $S_p$ are zero for $0 \leq p < i \leq r + s$. 
		Thus $L = S_i + \sum_{j = i+1}^{r+s} S_j$.
		Note that the $i$--th row of $S_j$ vanishes for all $j>i$.
		Thus by restricting to the $i$--th row, we obtain that the $i$--th row of $S_i$ equals zero. 
		This row is a linear combination of vectors in $\mathfrak{R}_i$. 
		Again, by the subclaim, we deduce that $a_i = a_{i, j} = 0$ for every $i < j \leq r + s$. 
		The claim follows by induction.
	\end{proof}
	The cardinality of $\mathfrak{d}$ is exactly $2g - 1 + \binom{2g - 1}{2} + 1 = g(2g - 1) + 1$, so the lemma follows from the claim.
\end{proof}

\begin{corollary}
	\label{c:Zariski_density_five_above_hyperelliptic}
	For $g \geq 2$, we have $\mathfrak{m} = \mathfrak{sp}(2g, \RR)$.
\end{corollary}
\begin{proof}
	Suppose that $g \geq 4$. Since $\dim_{\RR} \mathfrak{m} \geq g(2g - 1) + 1$, it follows that
	\begin{align*}
		\dim_{\RR} \mathfrak{m} &> \dim_{\RR} \mathfrak{su}(g/2,g/2) = g^2 -1\\
		\dim_{\RR} \mathfrak{m} &> \dim_{\RR} \mathfrak{so}(2n-1,2) = \log_2(2g) \log_2(8g^2) \text{ for } n = \log_2(2g)\\
		\dim_{\RR} \mathfrak{m} &> \dim_{\RR} \mathfrak{so}(2n-2,2) = \log_2(4g) \log_2(8g^2) \text{ for } n = \log_2(4g)\\
		\dim_{\RR} \mathfrak{m} &> \dim_{\RR} \mathfrak{so}^*(g/2,g/2) = g(g - 1)/2.
	\end{align*}
	
	For $g = 3$, we have directly from the list that $\mathfrak{m} = \mathfrak{sp}(2g, \RR)$ since the action of $M$ on $\upH_1(S; \RR)$ is strongly irreducible.
	
	The only remaining case is $g = 2$. Here, we have that
	\begin{align*}
		\dim_{\RR} \mathfrak{m} \geq g(2g - 1) = 6 &> 3 = \dim_{\RR} \mathfrak{su}(1,1) \\
		\mathfrak{sp}(4, \RR) &\cong \mathfrak{so}(3,2) \\
		\dim_{\RR} \mathfrak{sp}(4, \RR) = 10 &< 15 = \dim_{\RR} \mathfrak{so}(4,2),
	\end{align*}
	so the only possibility for $\mathfrak{m}$ is $\mathfrak{sp}(4, \RR) \cong \mathfrak{so}(3,2)$.
\end{proof}

\subsection{Exceptional non-minimal strata}

In \Cref{table:extensions_sporadic}, we exhibit an explicit simple extension from a component of a minimal stratum to each non-hyperelliptic component of exceptional non-minimal strata that was not already treated by the fourth author \cite[Table 1]{Gut17}, except for $\calQ(9, -1)^\irreg$. Indeed, there does not exist a simple extension from $\calQ(8)$ to this last component as noted by Lanneau \cite{Lan08},
so we treat it separately in \Cref{s:sporadic}. These computations were performed by using the \texttt{surface\_dynamics} package for SageMath~\cite{Sage20}.

\section{Simplicity}
\label{s:simplicity}

We now prove the Kontsevich--Zorich conjecture.

\begin{theorem}
	\label{t:KZ}
	The Kontsevich--Zorich cocycle has a simple spectrum for all components of all strata of abelian differentials. The plus and minus Kontsevich--Zorich cocycles also have a simple spectrum for all components of all strata of quadratic differentials. 
\end{theorem}

\begin{proof}
	Let $\calC$ be any component of a stratum of abelian or quadratic differentials. We have the following facts.
	\begin{itemize}
		\item 
		The diagonal flow on $\calC_\root$ (a finite cover of $\calC$) admits a coding as a countable shift with an approximate product structure.
		See \Cref{l:almost-bernoulli}.
		\item 
		The plus and minus cocycles lifted to this cover are locally constant and integrable. 
		See \Cref{s:integrability_continuous} and \Cref{c:discrete_integrable}.
		\item 
		The symplectic monodromy and Rauzy--Veech groups for the plus and minus cocycles (lifted to $\calC_\root$) are Zariski dense in their ambient symplectic groups.
		See \Cref{s:classification,s:density,s:zariski-density-plus}.
		Since the Zariski closure of a monoid is a group\footnote{See Footnotes~5 and~6 of~\cite{Sol95}.}, 
		we deduce that the corresponding monoids are also Zariski dense.
	\end{itemize}
	The hypotheses of \Cref{c:simplicity} are then met. The conjecture follows. 
\end{proof}

\newpage

\appendix

\section{Examples}

\subsection{The decomposition of \texorpdfstring{$\calC_\root$}{C\textasciicircum{}root} is not polytopal}
\label{a:non-polytopal}

In this section, we present an explicit example showing that the decomposition of $\calC_\root$ into the union of the $\overline{\calC(\pi)}$, where $\pi \in \calR$, is ``not polytopal''. Indeed, a compact arc in $\calC_\root$ may intersect $\calS = \calC_\root - \bigcup_{\pi \in \calR} \calC(\pi)$ infinitely many times even if it is transverse to $\calV$.

\begin{figure}
	\centering
	\begin{subfigure}[b]{\textwidth}
		\centering
		\includegraphics[scale=0.60]{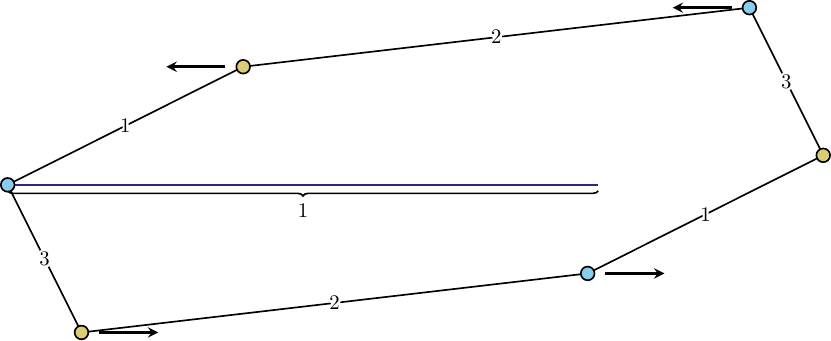}
		\caption{$\gamma(0)$.}
	\end{subfigure}
	
	\begin{subfigure}[b]{\textwidth}
		\centering
		\includegraphics[scale=0.60]{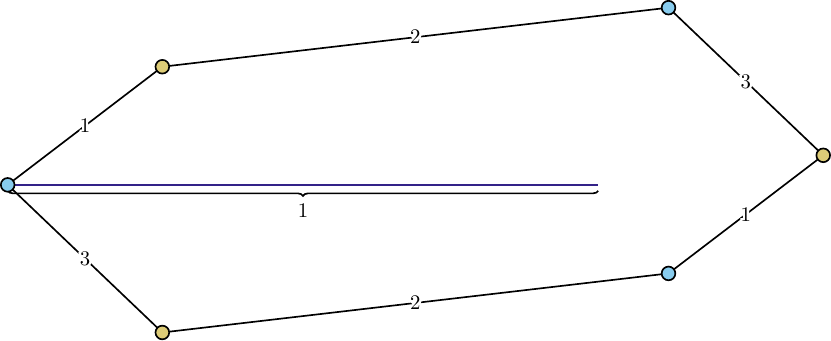}
		\caption{$\gamma(1)$.}
	\end{subfigure}
	\caption{A curve $\gamma \from [0, 1] \to \calH(0, 0)^\root$ that ends at $\calV$. The real periods are deformed following the arrows, while the imaginary periods remain constant.}
	\label{f:nonpolytopal example}
\end{figure}

Let $\calC = \calH(0, 0)$ and consider the curve shown in \Cref{f:nonpolytopal example}. The rooted differential $\gamma(1)$ contains a ``wide'' vertical cylinder, that is, a vertical cylinder such that an arc of length one emanating from the root ends before it crosses the cylinder entirely. On the other hand, we assume that the rooted differential $\gamma(0)$ is not bi-saddled, and that it has an admissible zippered rectangles construction for the underlying permutation $\pi = \big( \begin{smallmatrix} 1 & 2 & 3 \\ 3 & 2 & 1 \end{smallmatrix} \big)$.

Starting from $s = 0$ and as $s$ increases, the distinguished base-arc shrinks until its length is exactly equal to $1 + \min\{ x_1, x_3 \}$ at $s = s_1 > 0$. 
Thus, $\gamma(s)$ admits a distingushed base-arc for every $0 \leq s < s_1$, but $\gamma(s_1)$ does not (as an equality in \Cref{e:distinguished} is achieved). Indeed, $\gamma(s_1)$ belongs to a flow face. As $\gamma$ passes through this flow face, a forward Rauzy move must be performed to again obtain admissible parameters.
The winning letter is $3$, so the resulting permutation after the Rauzy move is again $\pi$.

This process continues inductively. Indeed, starting from $s = s_k$, for any integer $k \geq 1$, and as $s$ increases, the distinguished base-arc continues to shrink until the curve hits the flow face again at $s = s_{k+1} > s_k$. Thus, $\gamma(s)$ admits a distinguished base-arc for every $s_k < s < s_{k+1}$. A Rauzy move must be performed when the curve crosses the flow face; the winning letter continues to be $3$. Hence, the resulting permutation is again $\pi$.

In summary, there exists a countable collection $0 < s_1 < \dotsb < s_k < s_{k+1} < \dotsb {}$ such that:
\begin{enumerate}
	\item $\gamma(s)$ admits a distinguished base-arc for $0 \leq s < s_1$ and every $s_k < s < s_{k+1}$ for $k \geq 1$;
	\item $\gamma(s_k)$ belongs to a flow face for every $k \geq 1$.
\end{enumerate}
Thus, $\gamma$ intersects $\calS$ infinitely many times and there is no finite Rauzy--Veech sequence shadowing $\gamma$.
Moreover, as this process unfolds, the width of the rectangle $R_1$ goes to zero, while its height grows indefinitely. Hence, the admissible zippered rectangles constructions become more and more degenerate along $\gamma$, and do not converge to a well-defined element of $P(\pi)$. See \Cref{f:nonpolytopal rectangles example} for an illustration of this phenomenon.

\begin{figure}
	\centering
	\begin{subfigure}[b]{\textwidth}
		\centering
		\includegraphics[scale=0.60]{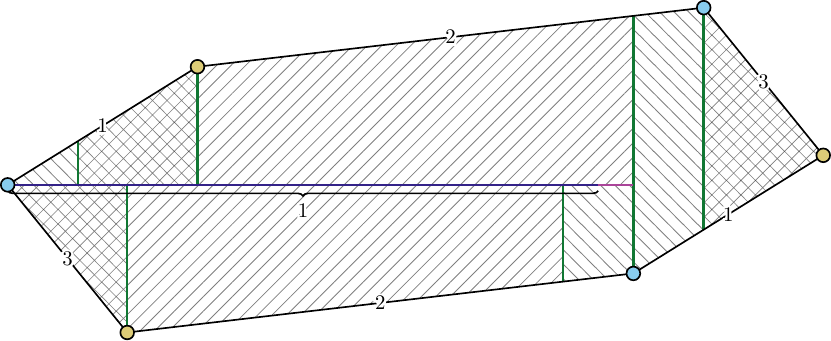}
		\caption{$\gamma(0)$.}
	\end{subfigure}
	
	\begin{subfigure}[b]{\textwidth}
		\centering
		\includegraphics[scale=0.60]{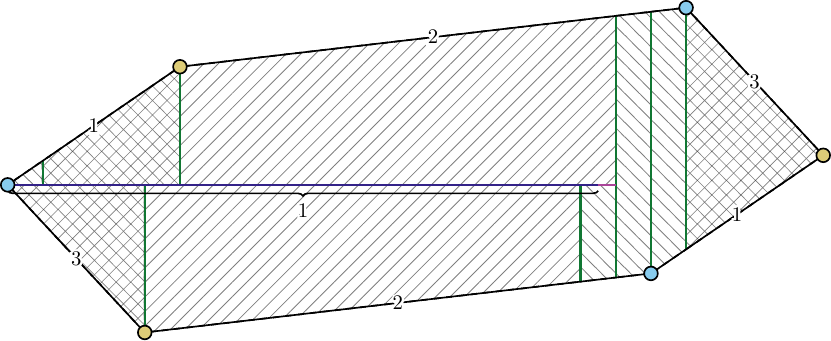}
		\caption{$\gamma(s)$ for $s_1 < s < s_2$.}
	\end{subfigure}
	
	\begin{subfigure}[b]{\textwidth}
		\centering
		\includegraphics[scale=0.60]{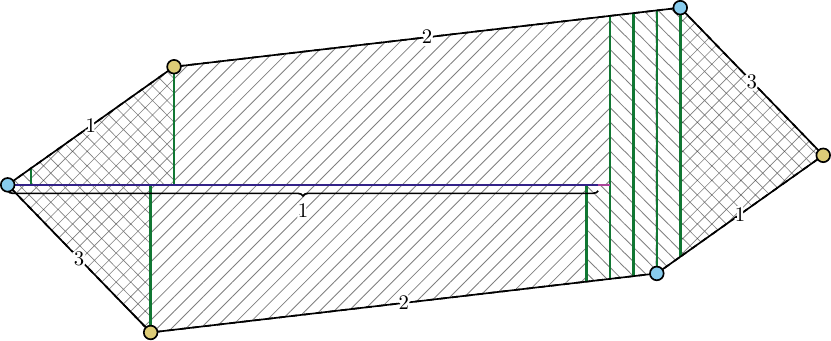}
		\caption{$\gamma(s)$ for $s_2 < s < s_3$.}
	\end{subfigure}
	\caption{Admissible zippered rectangles construction of $\gamma(s)$.}
	\label{f:nonpolytopal rectangles example}
\end{figure}

Similar examples exist also for \emph{toppling faces}, that is, when either some width $x_\alpha$ or some zipper height goes to zero.
It is possible to find a compact arc in $\calC_\root$ transverse to $\calV$ that intersects infinitely many toppling faces. Indeed, a simple way to obtain such an example is to consider a horizontal slit $J$ that does not meet the base-arc $I$, and such that some vertical segments emanating from $I$ meet $J$ before their first return. Then, the slit can be rotated until it becomes vertical (in a way that it still does not meet the base-arc). This forces the widths of some rectangles to hit zero infinitely many times before the slit becomes vertical.

\subsection{Crossing a toppling face}
\label{s:crossing faces}

In this section, we present concretely the construction used in the based-loop theorem, namely \Cref{t:based-loops}, in which any loop in $\uppi_1(\calC_\root, q_0)$ is written as a finite concatenation of paths that are forward (or backward) diagonal flow segments or are contained inside a polytope. The path in $\calC_\root$ that we present is not closed, but it still illustrates the key point. For more complicate paths or loops, this procedure has to be done several times.

\begin{figure}
	\centering
	\begin{subfigure}[b]{\textwidth}
		\centering
		\includegraphics[scale=0.60]{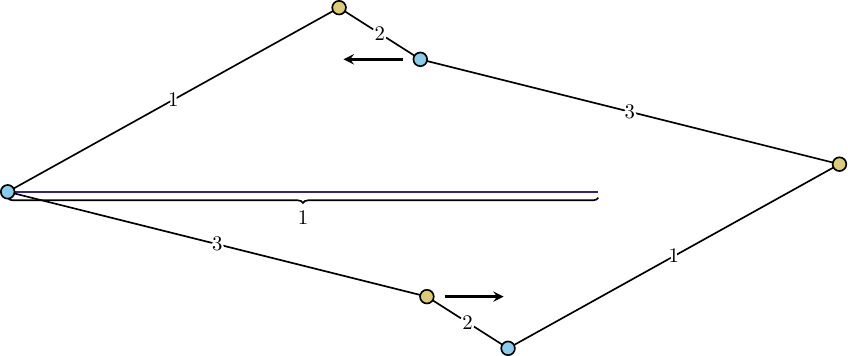}
		\caption{$\gamma(0)$.}
	\end{subfigure}
	
	\begin{subfigure}[b]{\textwidth}
		\centering
		\includegraphics[scale=0.60]{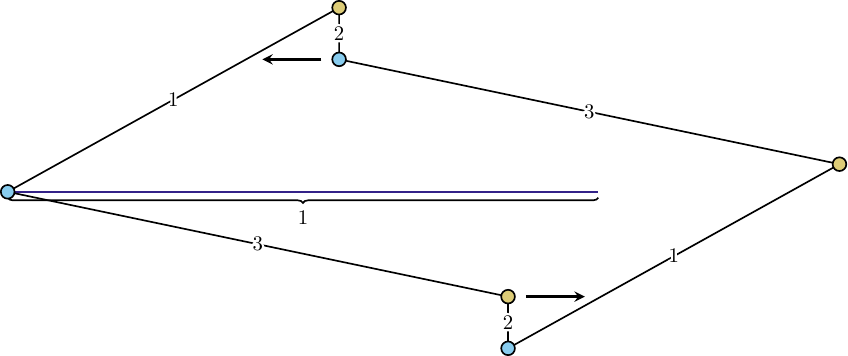}
		\caption{$\gamma(1/2)$.}
	\end{subfigure}
	
	\begin{subfigure}[b]{\textwidth}
		\centering
		\includegraphics[scale=0.60]{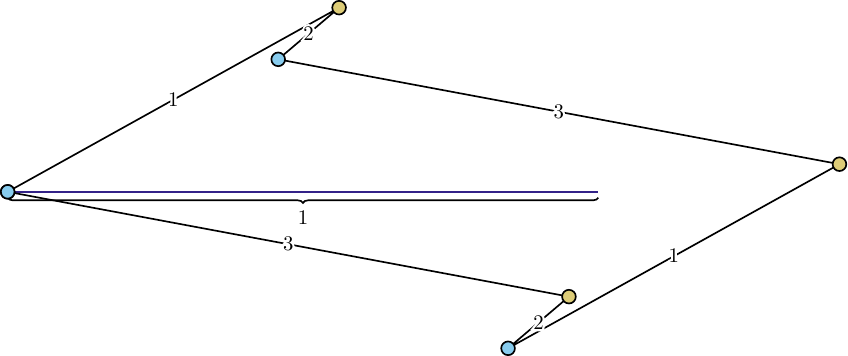}
		\caption{$\gamma(1)$.}
	\end{subfigure}
	\caption{A curve $\gamma \from [0, 1] \to \calH(0, 0)^\root$ that passes through $\calV$. The real periods are deformed following the arrows, while the imaginary periods remain constant.}
	\label{f:curve example}
\end{figure}

Let $\calC = \calH(0, 0)$. 
Consider the path $\gamma \from [0, 1] \to \calC_\root$ illustrated in \Cref{f:curve example}. 
Assume that $\gamma(0)$ and $\gamma(1)$ are not bi-saddled, while $\gamma(1/2)$ has a vertical saddle connection and, thus, belongs to $\calV$. 

\begin{figure}
	\centering
	\includegraphics[scale=0.60]{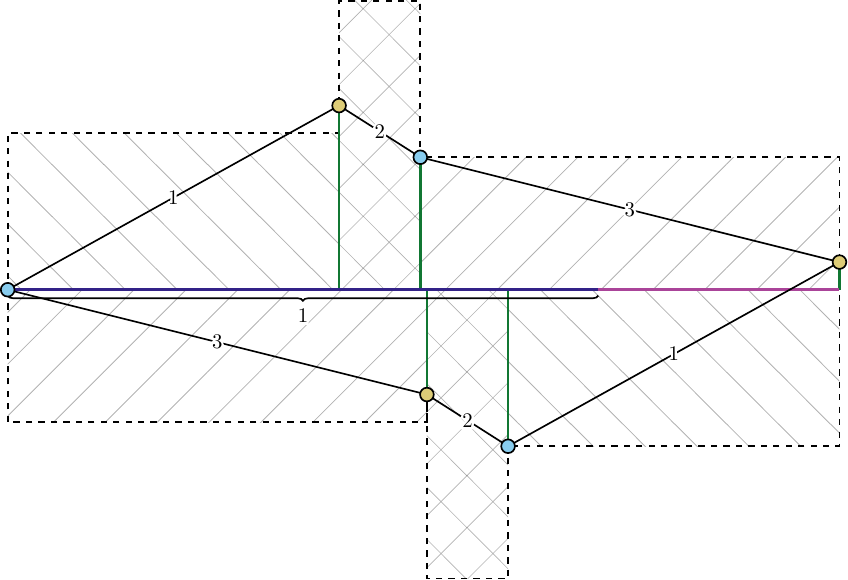}
	\caption{Admissible zippered rectangles construction of $\gamma(0)$.}
	\label{f:curve example 0 rect}
\end{figure}

We may assume that $\gamma(0)$ is not bi-saddled, so it admits a distinguished base-arc.
The resulting zippered rectangles construction, with underlying permutation $\pi = \big( \begin{smallmatrix} 1 & 2 & 3 \\ 3 & 2 & 1 \end{smallmatrix} \big)$, is shown in \Cref{f:curve example 0 rect}.

If $0 \leq s < 1/2$, a parameter $(x_s, y_s) \in P(\pi)$ of this zippered rectangles construction satisfies $q_\pi(x_s, y_s) = \gamma(s)$. As $s$ increases towards $1/2$, these parameters approach the boundary of $P(\pi)$ and $\gamma(1/2) \notin \calC(\pi)$. In particular, as $s$ increases to $1/2$, the width $x_2$ tends to zero while all other parameters stay bounded away from zero and infinity, and thus $\gamma(1/2)$ can be said to lie on a toppling face.

\begin{figure}
	\centering
	\begin{subfigure}[b]{\textwidth}
		\centering
		\includegraphics[scale=0.60]{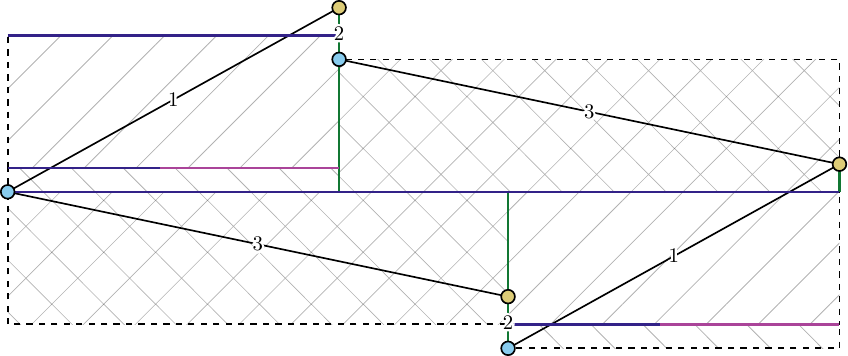}
		\caption{Before cutting and pasting.}
	\end{subfigure}
	
	\begin{subfigure}[b]{\textwidth}
		\centering
		\includegraphics[scale=0.60]{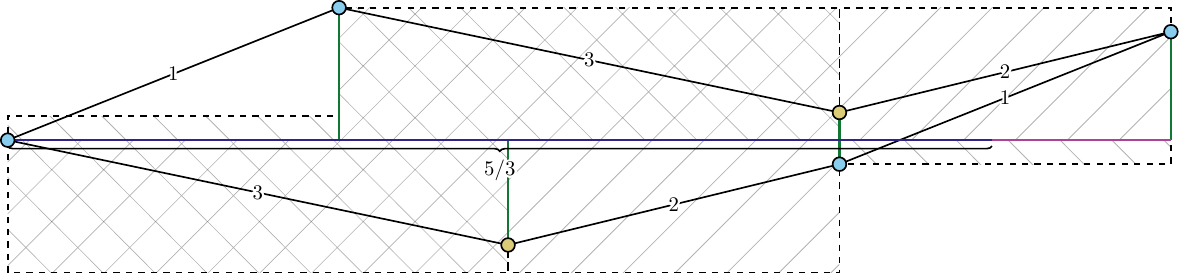}
		\caption{After cutting and pasting by two backward Rauzy moves.}
	\end{subfigure}
	\caption{Zippered rectangles construction of $\gamma(1/2)$ with a base-arc of length at least $5/3$.}
	\label{f:curve example 1/2 rect}
\end{figure}

On the other hand, $\gamma(1/2)$ does not lie in $\calW$. 
Thus, by \Cref{l:base-arc-infinite} the differential $\gamma(1/2)$ admits a base-arc. 
We choose a base-arc of length at least $5/3$, since the interior of any horizontal segment with length at least $5/3$ meets every leaf of the vertical foliation.
The resulting zippered rectangles construction, with underlying permutation $\sigma = \big( \begin{smallmatrix} 1 & 3 & 2 \\ 3 & 2 & 1 \end{smallmatrix} \big)$, is shown in \Cref{f:curve example 1/2 rect}. 
After diagonal flow by $T = -\log(5/3)$, this base-arc becomes a distinguished base-arc.  
Thus $g_T(\gamma(1/2))$ lies in $\calC(\sigma)$.
Observe that $\sigma$ is obtained from $\pi$ by two backward Rauzy moves. 

\begin{figure}
	\centering
	\begin{subfigure}[b]{\textwidth}
		\centering
		\includegraphics[scale=0.60]{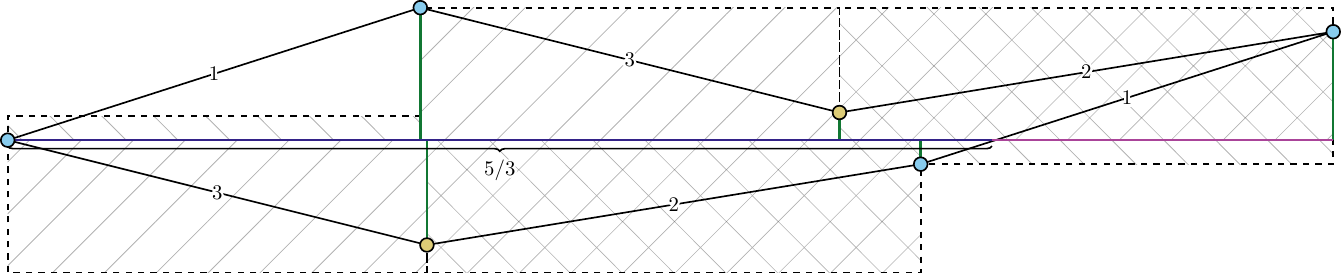}
		\caption{$\gamma(0)$.}
	\end{subfigure}
	
	\begin{subfigure}[b]{\textwidth}
		\centering
		\includegraphics[scale=0.60]{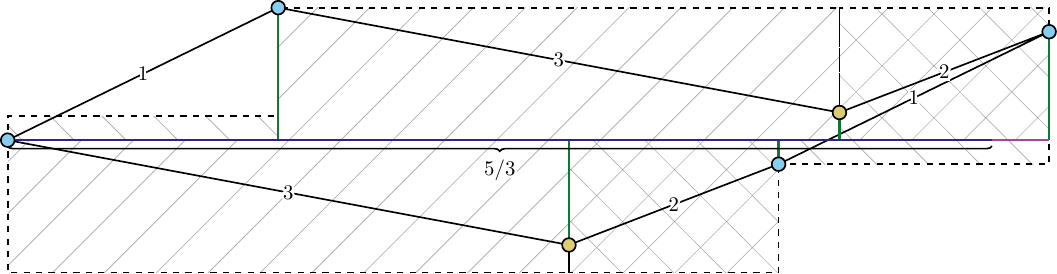}
		\caption{$\gamma(1)$.}
	\end{subfigure}
	\caption{Zippered rectangles constructions with a base-arc of length at least $5/3$.}
	\label{f:curve example 0 1 non-admiss rect}
\end{figure}

For ``large'' deformations of parameters, the zippered rectangles construction with a base-arc of length at least $5/3$ is contained in $\calC(\sigma)$ after diagonal flow by $T = -\log(5/3)$. In particular, we have that $g_T(\gamma(s)) \in \calC(\sigma)$ for every $s \in [0, 1]$. Let $(x_s', y_s') \in P(\sigma)$ such that $q_\sigma(x_s', y_s') = g_T(\gamma(s))$. \Cref{f:curve example 0 1 non-admiss rect} shows these zippered rectangles constructions for $\gamma(0)$ and $\gamma(1)$.

\begin{figure}
	\centering
	\begin{subfigure}[b]{\textwidth}
		\centering
		\includegraphics[scale=0.60]{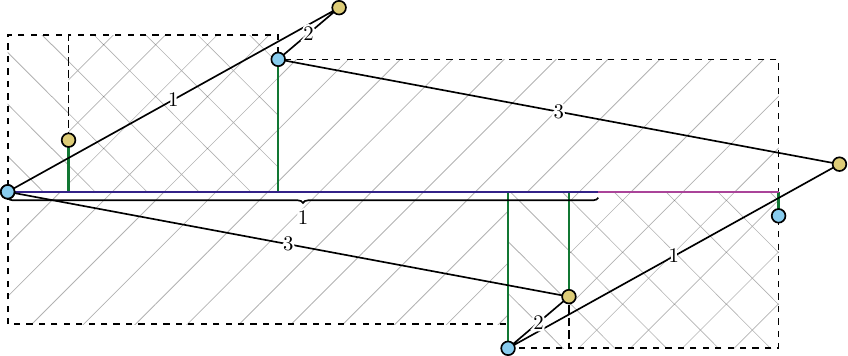}
		\caption{Before cutting and pasting.}
	\end{subfigure}
	
	\begin{subfigure}[b]{\textwidth}
		\centering
		\includegraphics[scale=0.60]{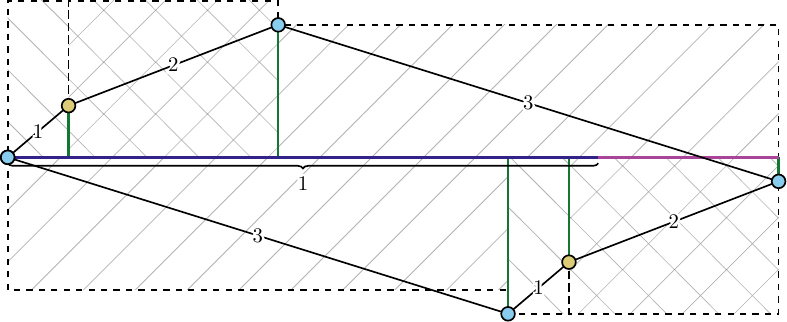}
		\caption{After cutting and pasting by two backward Rauzy moves followed by two forward Rauzy moves.}
	\end{subfigure}
	\caption{Admissible zippered rectangles construction of $\gamma(1)$.}
	\label{f:curve example 1 rect}
\end{figure}

Finally, we assume that $\gamma(1)$ is also not bi-saddled, so it admits a distinguished base-arc. 
The resulting zippered rectangles construction, with underlying permutation $\tau = \big( \begin{smallmatrix} 1 & 2 & 3 \\ 3 & 1 & 2 \end{smallmatrix} \big)$, is shown in \Cref{f:curve example 1 rect}. Observe that $\tau$ is obtained from $\pi$ by two backward Rauzy moves followed by two forward Rauzy moves.

If $1/2 < s \leq 1$, a parameter $(x_s, y_s)$ of this zippered rectangles construction satisfies $q_\tau(x_s, y_s) = \gamma(s)$. As $s$ decreases towards $1/2$, these parameters approach the boundary of $P(\tau)$ and $\gamma(1/2) \notin \calC(\tau)$.

Putting everything together, we obtain three open sets $U_0, U_{1/2}, U_1 \subseteq \calC_\root$ satisfying:
\begin{itemize}
	\item $U_0 = q_\pi(W_0)$, where $W_0 \subseteq P(\pi)$ is an open set containing $(x_0, y_0)$ whose closure is contained in $P(\pi)$;
	\item $U_{1/2} = g_{-T}(q_\sigma(W_{1/2}))$, where $W_{1/2} \subseteq P(\sigma)$ is an open set containing $(x_s', y_s')$ for every $s \in [0, 1]$ whose closure is contained in $P(\sigma)$; and
	\item $U_1 = q_\tau(W_1)$, where $W_1 \subseteq P(\tau)$ is an open set containing $(x_1, y_1)$ whose closure is contained in $P(\tau)$.
\end{itemize}

Then, $\gamma$ is homotopic, relative to its endpoints, to the concatenation of the paths:
\begin{itemize}
	\item $g_t \gamma(0)$ for $t \in [0, T]$;
	\item $g_T \gamma(s)$ for $s \in [0, 1]$; and
	\item $g_{-t} \gamma(1)$ for $t \in [0, T]$.
\end{itemize}

Therefore, the combinatorial description of this concatenation is
\[
\begin{pmatrix*}
	1 & 2 & 3 \\
	3 & 2 & 1
\end{pmatrix*} \xrightarrow{\Rbot^{-1}}
\begin{pmatrix*}
	1 & 3 & 2 \\
	3 & 2 & 1
\end{pmatrix*} \xrightarrow{\Rtop^{-1}}
\begin{pmatrix*}
	1 & 3 & 2 \\
	3 & 2 & 1
\end{pmatrix*} \xrightarrow{\Rbot}
\begin{pmatrix*}
	1 & 2 & 3 \\
	3 & 2 & 1
\end{pmatrix*} \xrightarrow{\Rtop}
\begin{pmatrix*}
	1 & 2 & 3 \\
	3 & 1 & 2
\end{pmatrix*}
\]
which is the (undirected) Rauzy--Veech sequence shadowing $\gamma$.

\section{Zariski density of the remaining cases}\label{s:sporadic}

In this section, we explicitly check the Zariski density for the plus piece of the four remaining components, namely $\calQ(5, -1)$, $\calQ(9, -1)^{\irreg}$, $\calQ(12)^{\reg}$ and $\calQ(12)^{\irreg}$.
We do this by using the following sufficient criterion.

\begin{criterion}[{\cite[Theorem 9.10]{Pra-Rap}}] \label{c:prasad rapinchuk}
	Let $G$ be a subgroup of $\Sp(2g, \ZZ)$. We have that $G$ is Zariski dense in $\Sp(2g, \RR)$ provided the Zariski closure of $G$ is not a power of $\SL(2, \RR)$, and there exist elements $A, B \in G$ satisfying:
	\begin{enumerate}
		\item $A$ is Galois-pinching in the sense of Matheus--Möller--Yoccoz \cite{Mat-Moe-Yoc}. That is, all of its eigenvalues are real and have distinct moduli, and the Galois group of its characteristic polynomial is maximal; and \label{i:Galois pinching}
		\item $B$ has infinite order and does not commute with $A$. \label{i:infinite order}
	\end{enumerate}
\end{criterion}

Since $A$ is symplectic, its characteristic polynomial $P$ is reciprocal. 
Thus, the Galois group of $P$ is contained inside an appropriate hyperoctahedral group. 
Hence, this group is maximal if and only if it has order $2^g g!$. 
Moreover, if the Zariski closure of a symplectic monodromy group is a power of $\SL(2, \RR)$, then it has more than one noncompact factor, which is forbidden for strongly irreducible pieces \cites[Theorem 1.2]{Fil17}[Theorem 1.1]{Esk-Fil-Wri}. 
Thus, if we can establish \Cref{c:prasad rapinchuk} together with \Cref{l:irreducible}, we obtain the Zariski density of $G$ inside $\Sp(2, \RR)$.

For the remaining components, we follow the same strategy.
We start with a specific permutation $\pi$.
We then exhibit two cycles $\delta_1$ and $\delta_2$ based at $[\pi]$ in the Rauzy diagram. 
We arrange matters so that the squares $\delta_1^2$ and $\delta_2^2$ are cycles in the \emph{labelled} Rauzy diagram based at $\pi$. 
Let $D_1$ and $D_2$ be the matrices coming from the actions of $\delta_1^2$ and $\delta_2^2$ on absolute homology, respectively, in the basis preferred by $\pi$.  
Let $A = D_1 D_2$ and $B = D_1$.  
Then we claim that $A$ and $B$ satisfy \Cref{c:prasad rapinchuk}.

These cycles $\delta_1$ and $\delta_2$ were found by a randomised computer search on the Rauzy diagrams.
We choose relatively short cycles to ensure that entries of the matrices $A$ and $B$ are relatively small.

\subsection{Zariski density of \texorpdfstring{$\calQ(5, -1)$}{Q(5, -1)}}
\label{s:Q(5 -1)}

Let
\[
\pi =
\begin{pmatrix*}
	1 & 2 & 3 & 2 & 4 \\
	4 & 5 & 5 & 3 & 1
\end{pmatrix*}
\]
and
\begin{align*}
	\delta_1 = {} & \Rbot^3 \Rtop^2 \Rbot^3 \Rtop \Rbot^2 \Rtop \Rbot^3 \Rtop^2 \Rbot^3 \\
	\delta_2 = {} & \Rtop^2 \Rbot \Rtop \Rbot \Rtop \Rbot \Rtop^3 \Rbot \Rtop \Rbot \Rtop \Rbot^2 \Rtop^2
\end{align*}

\begin{figure}
	\centering
	\includegraphics[scale=0.8]{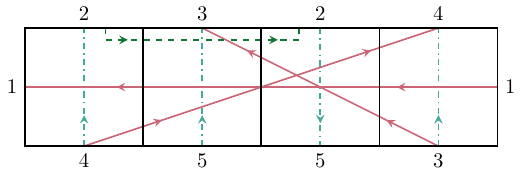}
	\caption{Representative of $\calQ(5, -1)$.}
	\label{f:Q(5 -1)}
\end{figure}

Consider the four curves $c_1, \dotsc, c_4$ depicted in \Cref{f:Q(5 -1)} as solid or dashed lines. These cycles form a basis for the absolute homology as their intersection matrix is
\[
\Omega =
\begin{pmatrix*}[r]
	0 & 0 & -1 & -1 \\
	0 & 0 & 1 & 0 \\
	1 & -1 & 0 & -1 \\
	1 & 0 & 1 & 0
\end{pmatrix*}
\]
that has determinant $1$. On the other hand, the cycle $v \in \upH_1(S; \RR)$ depicted in \Cref{f:Q(5 -1)} as dash-dotted vertical lines can be written as $v = -c_2 + c_3 + c_4$. Thus, the set $B = \{c_1, c_3, c_4, v\}$ readily satisfies the hypotheses of \Cref{l:irreducible}, so the $M$--action is strongly irreducible.

In the chosen basis, the matrices induced by $\delta_1^2$ and $\delta_2^2$ are
\[
A_1 = \begin{pmatrix*}[r]
	1 & -2 & -2 & 0 \\
	0 & -1 & -2 & 0 \\
	0 & 0 & 1 & 2 \\
	0 & 0 & 0 & -1
\end{pmatrix*}
\qquad
A_2 = \begin{pmatrix*}[r]
	-1 & 0 & 0 & 0 \\
	0 & -2 & 2 & -1 \\
	1 & 2 & -1 & 2 \\
	-2 & 1 & -2 & 0
\end{pmatrix*}
\]
Then, $A = A_1 A_2$ has the form
\[
A = \begin{pmatrix*}[r]
	-1 & 0 & 1 & -2 \\
	2 & 2 & -4 & 3 \\
	2 & 6 & -7 & 0 \\
	0 & 5 & -4 & -4 \\
\end{pmatrix*}
\]
The characteristic polynomial $P$ of $A$ is $P(t) = t^4 + 10 t^3 + 22 t^2 + 10 t + 1$. 
We verified in Magma \cite{Magma97} that $A$ is Galois pinching, that is, it satisfies condition \eqref{i:Galois pinching} of \Cref{c:prasad rapinchuk}. 
Setting $B = A_1$, we similarly check that $B$ satisfies condition \eqref{i:infinite order} of the criterion. 
Thus, the plus piece of $\calQ(5, -1)$ is Zariski dense inside $\Sp(4, \RR)$.

\subsection{Zariski density of \texorpdfstring{$\calQ(9, -1)^{\irreg}$}{Q(9, -1)\textasciicircum{}irr}}

Let
\[
\pi = 
\begin{pmatrix*}
	1 & 2 & 3 & 4 & 5 & 6 & 3 \\
	7 & 7 & 6 & 5 & 4 & 2 & 1
\end{pmatrix*}
\]
and
\begin{align*}
	\delta_1 = {} & \Rbot^4 \Rtop^5 \Rbot^3 \Rtop \Rbot^5 \Rtop \Rbot^6 \Rtop^2 \\
	\delta_2 = {} & \Rbot^4 \Rtop^2 \Rbot^3 \Rtop^3 \Rbot^7 \Rtop^3 \Rbot \Rtop^3 \Rbot \Rtop^2 \Rbot^2 \Rtop^3 \Rbot^2 \Rtop^2 \Rbot^2 \Rtop^2 \Rbot^2
\end{align*}

Consider the six curves $c_1, \dotsc, c_6$ depicted in \Cref{f:Q(9 -1)irr} as solid or dashed lines. These cycles form a basis for the absolute homology, as their intersection matrix is
\[
\Omega =
\begin{pmatrix*}[r]
	0 & -1 & 0 & -1 & -1 & -1 \\
	1 & 0 & 0 & -1 & -1 & -1 \\
	0 & 0 & 0 & 1 & 1 & 1 \\
	1 & 1 & -1 & 0 & -1 & -1 \\
	1 & 1 & -1 & 1 & 0 & -1 \\
	1 & 1 & -1 & 1 & 1 & 0
\end{pmatrix*}
\]
that has determinant $1$. On the other hand, the cycle $v \in \upH_1(S; \RR)$ depicted in \Cref{f:Q(9 -1)irr} as dash-dotted vertical lines can be written as $v = c_2 + c_3$. Thus, the set $B = \{c_1, c_2, c_4, c_5, c_6, v\}$ readily satisfies the hypotheses of \Cref{l:irreducible}, so the $M$--action is strongly irreducible.

\begin{figure}
	\centering
	\includegraphics[scale=0.8]{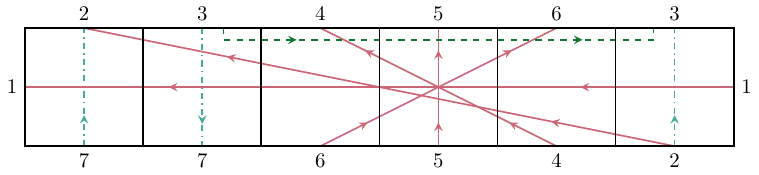}
	\caption{Representative of $\calQ(9, -1)^{\irreg}$.}
	\label{f:Q(9 -1)irr}
\end{figure}

In the chosen basis, the matrices induced by $\delta_1^2$ and $\delta_2^2$ are
\[
A_1 = \begin{pmatrix*}[r]
	0 & 2 & -1 & -1 & -1 & -2 \\
	0 & -1 & 0 & 0 & 0 & 0 \\
	-1 & -1 & 0 & -2 & -2 & -1 \\
	0 & 0 & 0 & -1 & 0 & 0 \\
	0 & 0 & 0 & 0 & -1 & 0 \\
	1 & 3 & -1 & 2 & 2 & 0
\end{pmatrix*}
\quad
A_2 = \begin{pmatrix*}[r]
	-3 & -10 & -2 & -4 & -6 & -4 \\
	1 & 3 & 0 & 2 & 2 & 0 \\
	2 & 3 & -2 & 0 & -1 & 0 \\
	0 & -2 & -2 & -1 & -2 & -2 \\
	-3 & -7 & 1 & -2 & -2 & 0 \\
	0 & 0 & 0 & 0 & 0 & -1
\end{pmatrix*}
\]
Then, $A = A_1 A_2$ has the form
\[
A = \begin{pmatrix*}[r]
	-2 & 0 & 2 & 0 & -1 & -1 \\
	-6 & -1 & 3 & -2 & 0 & -3 \\
	7 & -1 & -2 & 2 & 3 & 1 \\
	3 & -3 & 2 & 1 & 3 & -2 \\
	5 & -3 & 3 & 2 & 3 & -2 \\
	8 & -2 & -2 & 2 & 5 & 0
\end{pmatrix*}
\]
The characteristic polynomial $P$ of $A$ is $P(t) = t^6 + t^5 - 22 t^4 - 52 t^3 - 22 t^2 + t + 1$. 
Again, we use Magma to check that $A$ is Galois pinching. Setting $B = A_1$, we can readily check that $B$ satisfies condition \eqref{i:infinite order} of the criterion. Thus, the plus piece of $\calQ(9, -1)^{\irreg}$ is Zariski dense inside $\Sp(6, \RR)$.

\subsection{Zariski density of \texorpdfstring{$\calQ(12)^{\reg}$}{Q(12)\textasciicircum{}reg}}

Let
\[
\pi = 
\begin{pmatrix*}
	1 & 2 & 1 & 3 & 4 & 5 & 6 & 7 \\
	2 & 4 & 3 & 6 & 5 & 8 & 7 & 8
\end{pmatrix*}
\]
and
\begin{align*}
	\delta_1 = {} & \Rbot^4 \Rtop^2 \Rbot \Rtop^6 \Rbot \Rtop^4 \Rbot^2 \Rtop \Rbot^4 \Rtop^2 \Rbot^7 \Rtop^2 \Rbot^2 \Rtop \Rbot^5 \Rtop^4 \Rbot \Rtop \Rbot \Rtop \Rbot \\
	\delta_2 = {} & \Rbot \Rtop \Rbot \Rtop^3 \Rbot \Rtop^6 \Rbot^2 \Rtop^4 \Rbot \Rtop^3 \Rbot^2 \Rtop^3 \Rbot^2 \Rtop \Rbot^4 \Rtop^2 \Rbot \Rtop^3 \Rbot^2 \Rtop^2 \Rbot^3 \Rtop \Rbot^5
\end{align*}

Let $M$ be the symplectic monodromy group of $\calQ(12)^{\reg}$. We have that the $M$--action is strongly irreducible by \Cref{l:irreducible-minimal}.

Consider the six curves $c_1, \dotsc, c_6$ shown in \Cref{f:Q(12)reg}. Ordered appropriately, these curves form a symplectic basis. In this basis, the matrices induced by $\delta_1^2$ and $\delta_2^2$ are
\begin{align*}
	A_1 &= \begin{pmatrix*}[r]
		2 & 5 & -1 & 7 & 6 & 2 & 1 & 0 \\
		-1 & 0 & 1 & -1 & -2 & 1 & 3 & 0 \\
		-1 & 2 & -1 & 2 & 1 & 1 & -2 & 0 \\
		-1 & -4 & 1 & -6 & -5 & -2 & 0 & 0 \\
		1 & 1 & 2 & 1 & 0 & 2 & 5 & 0 \\
		-1 & -3 & 1 & -4 & -4 & -1 & 3 & 0 \\
		0 & 0 & 0 & 0 & 0 & 0 & -1 & 0 \\
		-1 & -5 & 7 & -11 & -10 & 2 & 11 & -1
	\end{pmatrix*} \\
	A_2 &= \begin{pmatrix*}[r]
		1 & 3 & -5 & 1 & 6 & -7 & -3 & -1 \\
		0 & -2 & 2 & 0 & -3 & 3 & 1 & 1 \\
		4 & -1 & -1 & 2 & 3 & -4 & 1 & 0 \\
		-7 & 5 & -3 & -2 & -1 & 3 & -5 & -1 \\
		5 & -5 & 3 & 2 & -2 & 1 & 5 & 2 \\
		2 & 0 & -1 & 0 & 2 & -3 & 0 & -1 \\
		8 & -1 & -4 & 3 & 7 & -10 & 0 & -1 \\
		2 & 3 & -5 & 1 & 6 & -7 & -3 & -2
	\end{pmatrix*}
\end{align*}
Then, $A = A_1 A_2$ has the form
\[
A = \begin{pmatrix*}[r]
	17 & -7 & 15 & -17 & 5 & 11 & 36 & 20 \\
	23 & -13 & 25 & -38 & 8 & 24 & 62 & 33 \\
	6 & 0 & 0 & 7 & 0 & -7 & -5 & -2 \\
	33 & -20 & 34 & -50 & 6 & 37 & 89 & 51 \\
	28 & -16 & 31 & -47 & 9 & 33 & 81 & 44 \\
	15 & -7 & 12 & -15 & 3 & 8 & 27 & 15 \\
	21 & -6 & 5 & 12 & -6 & -6 & 7 & 11 \\
	1 & -1 & 0 & 1 & -2 & 1 & 1 & 2
\end{pmatrix*}
\]

\begin{figure}
	\centering
	\includegraphics[scale=0.8]{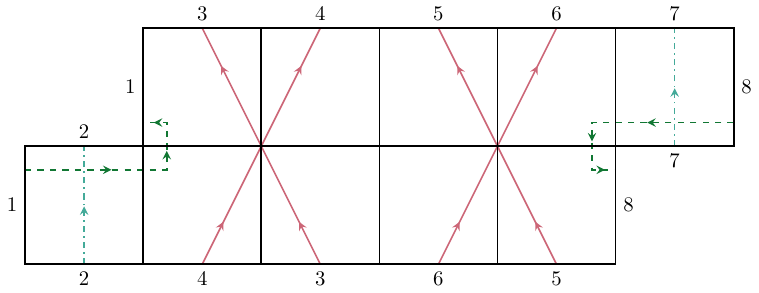}
	\caption{Representative of $\calQ(12)^{\reg}$.}
	\label{f:Q(12)reg}
\end{figure}

The characteristic polynomial $P$ of $A$ is $P(t) = t^8 + 20 t^7 - 1686 t^6 -24 t^5 + 36258 t^4 - 24 t^3 - 1686 t^2 + 20 t + 1$.
Using Magma, we can check that $A$ is Galois pinching.
Setting $B = A_1$, we can also similarly check that $B$ satisfies condition \eqref{i:infinite order} of the criterion.
Thus, the plus piece of $\calQ(12)^{\reg}$ is Zariski dense inside $\Sp(8, \RR)$.

\subsection{Zariski density of \texorpdfstring{$\calQ(12)^{\irreg}$}{Q(12)\textasciicircum{}irr}}

Let
\[
\pi = 
\begin{pmatrix*}
	1 & 2 & 1 & 3 & 4 & 5 & 6 & 7 \\
	2 & 6 & 5 & 4 & 3 & 8 & 7 & 8
\end{pmatrix*}
\]
and
\begin{align*}
	\delta_1 = {} & \Rbot^3 \Rtop \Rbot^2 \Rtop^5 \Rbot \Rtop^5 \Rbot^2 \Rtop^3 \Rbot \Rtop^2 \Rbot^2 \Rtop^6 \Rbot \Rtop^2 \Rbot^2 \Rtop \Rbot \Rtop^3 \Rbot^2 \Rtop^6 \Rbot^4 \\
	\delta_2 = {} & \Rbot^5 \Rtop^5 \Rbot \Rtop \Rbot \Rtop^4 \Rbot^3 \Rtop^7 \Rbot^2 \Rtop \Rbot \Rtop^4 \Rbot^4
\end{align*}

Consider the six curves $c_1, \dotsc, c_6$ depicted in \Cref{f:Q(12)irr} as solid, dashed or dash-dotted lines. These cycles form a basis for the absolute homology as their intersection matrix is
\[
\begin{pmatrix*}[r]
	0 & 1 & 0 & 0 & 0 & 0 & 0 & 0 \\
	-1 & 0 & 0 & 0 & 0 & 0 & 0 & 0 \\
	0 & 0 & 0 & -1 & -1 & -1 & 0 & 0 \\
	0 & 0 & 1 & 0 & -1 & -1 & 0 & 0 \\
	0 & 0 & 1 & 1 & 0 & -1 & 0 & 0 \\
	0 & 0 & 1 & 1 & 1 & 0 & 0 & 0 \\
	0 & 0 & 0 & 0 & 0 & 0 & 0 & 1 \\
	0 & 0 & 0 & 0 & 0 & 0 & -1 & 0
\end{pmatrix*}
\]
and has determinant $1$. On the other hand, the cycle $b$, depicted as the slope-$1$ densely dotted lines, can be written as $b = c_1 + c_2 - c_3 + c_4 + c_7 - c_8$, and the cycle $p$ depicted as loosely dotted horizontal lines can be written as $p = c_1 + c_8$. Thus, the set $B = \{c_2, c_3, c_4, c_5, c_6, c_7, b, p\}$ readily satisfies the hypotheses of \Cref{l:irreducible}, so the $M$--action is strongly irreducible.

\begin{figure}
	\centering
	\includegraphics[scale=0.8]{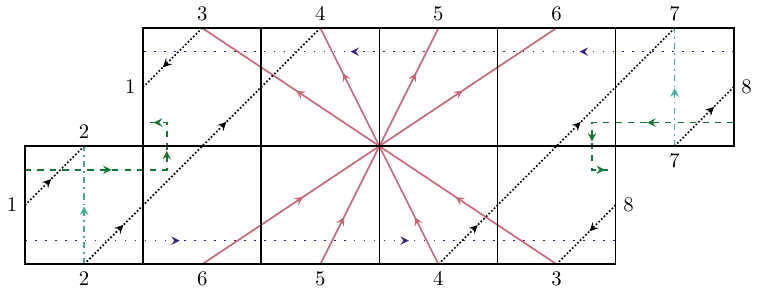}
	\caption{Representative of $\calQ(12)^{\irreg}$.}
	\label{f:Q(12)irr}
\end{figure}

In the chosen basis, the matrices induced by $\delta_1^2$ and $\delta_2^2$ are
\begin{align*}
	A_1 &= \begin{pmatrix*}[r]
		-1 & 6 & -2 & 6 & 11 & 8 & 6 & -4 \\
		2 & -5 & 0 & -6 & -9 & -6 & -4 & 6 \\
		-1 & -3 & -2 & -6 & -8 & -5 & -3 & 4 \\
		-1 & -5 & -1 & -9 & -11 & -7 & -5 & 6 \\
		0 & 0 & 0 & 0 & -1 & 0 & 0 & 0 \\
		1 & 3 & 1 & 6 & 8 & 4 & 3 & -4 \\
		-2 & -4 & 2 & -4 & -8 & -6 & -5 & 0 \\
		0 & -6 & 2 & -6 & -11 & -8 & -6 & 3 \\
	\end{pmatrix*} \\
	A_2 &= \begin{pmatrix*}[r]
		-2 & 1 & 1 & 1 & -1 & -2 & -3 & 0 \\
		3 & -5 & 3 & 3 & 3 & 0 & -2 & 0 \\
		0 & 0 & -1 & 0 & 0 & 0 & 0 & 0 \\
		0 & 0 & 0 & -1 & 0 & 0 & 0 & 0 \\
		0 & 2 & -2 & -2 & -1 & 2 & 2 & 0 \\
		-3 & 4 & -3 & -3 & -3 & -1 & 2 & 0 \\
		0 & 0 & 0 & 0 & 0 & 0 & -1 & 0 \\
		-7 & 13 & -9 & -9 & -7 & 2 & 5 & -1
	\end{pmatrix*}
\end{align*}
Then, $A = A_1 A_2$ has the form
\[
A = \begin{pmatrix*}[r]
	6 & -15 & 1 & 1 & 6 & 12 & 2 & 43 \\
	-19 & 27 & 3 & 5 & 4 & -25 & 4 & -43 \\
	-7 & -19 & 2 & 1 & 12 & 18 & -2 & 51 \\
	-33 & 11 & 6 & 9 & 22 & -11 & 4 & 13 \\
	-41 & 34 & 8 & 11 & 21 & -33 & 8 & -29 \\
	-24 & 30 & 5 & 7 & 8 & -28 & 6 & -40 \\
	-15 & 24 & 3 & 5 & 4 & -23 & 5 & -35 \\
	32 & -12 & -4 & -6 & -16 & 10 & 0 & 5
\end{pmatrix*}
\]
The characteristic polynomial $P$ of $A$ is $P(t) = t^8 - 47 t^7 - 794 t^6 + 11691 t^5 - 22022 t^4 + 11691 t^3 - 794 t^2 - 47 t + 1$. By using Magma again, we can explicitly check that $A$ is Galois pinching. Setting $B = A_1$, we can similarly check that $B$ satisfies condition \eqref{i:infinite order} of the criterion. Thus, the plus piece of $\calQ(12)^{\irreg}$ is Zariski dense inside $\Sp(8, \RR)$.

\section{Masur polygons}
\label{a:Masur}

As is well-known~\cites[Figure 17]{Via06}[Figure 8]{Boi-Lan}, the Masur polygon associated with given height parameters is not necessarily embedded.  We give an example in \Cref{f:zippered self intersections}.  This is one justification for our use of singularity parameters (\Cref{s:singularity_parameters}).

\begin{figure}
	\centering
	\begin{subfigure}[b]{0.45\textwidth}
		\centering
		\includegraphics[width=\textwidth]{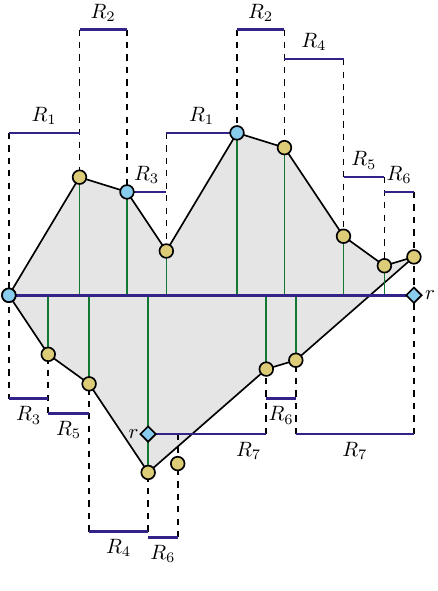}
		\caption{Without self-intersections.}
		\label{f:zippered}
	\end{subfigure}
	\hfill
	\begin{subfigure}[b]{0.45\textwidth}
		\centering
		\includegraphics[width=\textwidth]{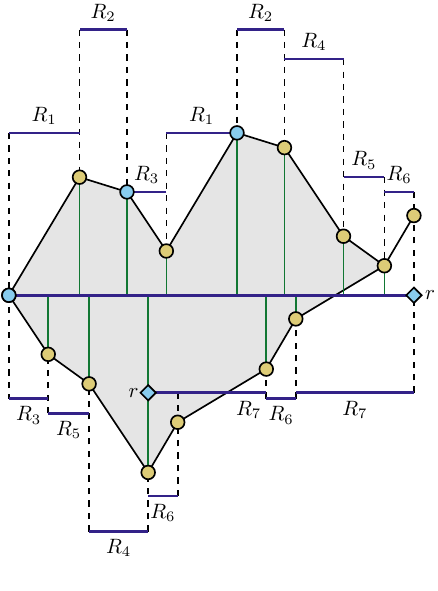}
		\caption{With self-intersections.}
		\label{f:zippered self intersections}
	\end{subfigure}
	\caption{Two illustrations of the zippered rectangles construction for the permutation $\big( \begin{smallmatrix} 1 & 2 & 3 & 1 & 2 & 4 & 5 & 6 \\ & 3 & 5 & 4 & 7 & 6 & 7 \end{smallmatrix} \big)$.  The two singularities are marked with circles while the right endpoint $r$ is marked with a diamond.  On the left (A) the polygon constructed by joining the singularities has no self-intersections.  On the right (B) the polygon self-intersects. In each figure we include an extra copy of $R_6$ in order to provide an open neighbourhood of the zipper running through the (copy of the) right endpoint $r$ (here marked with a diamond).  Note that to provide an open neighbourhood of the boundary of the polygon in (A) we would have to include an extra copy of $R_5$.  In (B) we would have to include an extra copy of $R_5$ and $R_4$.}
	\label{f:zippered rectangles}
\end{figure}

\section{Veech, Rauzy, and Masur--Smillie--Veech measures}
\label{a:measures}

Here we give a brief account of some of the various measures arising on stratum components, on lower boundaries of polytopes, and on spaces of (non-classical) interval exchange transformations.  

\subsection{A special case of the disintegration theorem}
\label{s:disintegration}

Suppose that $M$ is a smooth $m$--manifold.
Let $\mathrm{C}_{\mathrm{c}}(M)$ be the vector space of continuous and compactly supported functions from $M$ to $\RR$. 
By the representation theorem, 
the vector space of linear functionals on $\mathrm{C}_{\mathrm{c}}(M)$ is (canonically) isomorphic to the vector space of (signed) Radon measures on $M$.  
This isomorphism sends the cone of positive linear functionals to the cone of (unsigned) Radon measures. 

Suppose now that $M$ is smoothly embedded in $N$, a smooth $n$--manifold.
Let $\iota \from M \to N$ be the inclusion.
Suppose that $\omega$ is a $m$--form on $N$. 
Define $\varphi_\omega \from \mathrm{C}_{\mathrm{c}}(M) \to \RR$ by $\varphi_\omega(f) = \int_M f \cdot \iota^* \omega$.
Applying the isomorphism, let $\mu_\omega$ be the resulting Radon measure. 

As an example of this, suppose that $M = N = \RR^n$. 
Suppose that $x = (x_i)$ are the usual coordinates on $\RR^n$. 
Let $\diff x = \bigwedge_i \diff x_i$ be the resulting volume form.  
Then $\leb_N = \mu_{\diff x}$ is the usual Lebesgue measure on $\RR^n$. 

Suppose now that $F \from N \to \RR^k$ is a smooth function.
Suppose that $F = (f_j)$ are the components of $F$.
Define $\diff F = \bigwedge_j \diff f_j$.
Suppose that $v \in \RR$ is a regular value for $F$.
So $M = F^{-1}(v)$ is a smooth $(n - k)$--submanifold. 
Suppose that $\omega$ and $\omega'$ are any pair of $n - k$--forms on $N$. 
Suppose that $\omega \wedge \diff F = \omega' \wedge \diff F$.
We deduce that $\mu_\omega$ and $\mu_{\omega'}$ are equal as measures on $M$. 


Suppose that $N$, $F$, $v$, $M$, and $\omega$ are as above.
Suppose also that $N$ is equipped with a volume form $\diff x$.
Let $\leb_N$ be the resulting Lebesgue measure on $N$.
Suppose finally that $\omega \wedge \diff F = \diff x$.  
Set $\mu = \mu_\omega$. 
Then we say that $\mu$ is obtained by \emph{disintegrating $\leb_N$ with respect to $F$}.  
By the discussion in the previous paragraph, the measure $\mu$ depends only on $F$ and not on the choice of $\omega$. 
Note that this is a (very special) case of the disintegration theorem;
the level sets of $F$ provide the foliation needed in that result.

\subsection{Measures on full cones}
\label{s:polytopal-measures-full}


Suppose that $\pi$ is a generalised permutation, as defined in \Cref{s:combinatorics}.
Let $W(\pi)$ and $H(\pi)$ each be copies of $\RR^\calA$; 
these record the widths and heights of the singularity parameters introduced in \Cref{s:singularity_parameters}.
Let $V(\pi) = W(\pi) \cross H(\pi)$.
We equip $V(\pi)$ with the volume form $\diff x \wedge \diff y$; 
this gives the Lebesgue measure $\leb_{W(\pi)} \cross \leb_{H(\pi)}$. 
We equip $V(\pi)$ with the \emph{diagonal flow} defined by $(x,y) \mapsto (e^t x, e^{-t} y)$.
(This is sometimes called the \emph{Veech flow}~\cite[Section~3]{Avi-Gou-Yoc}.)
Note that the diagonal flow on $V(\pi)$ preserves volume in $V(\pi)$, but not in $W(\pi)$ or in $H(\pi)$. 

\begin{definition}
	\label{d:cones}
	Let $X(\pi)$ be the open cone in $W(\pi)$ given by the positivity conditions (that is, $x_\alpha > 0$ for all $\alpha \in \calA$) and also, in the quadratic case, the width equality in \eqref{e:e.x-y}.
	
	Let $Y(\pi)$ be the open cone in $H(\pi)$ given by the zipper inequalities 
	\eqref{e:top-zippers} and \eqref{e:bottom-zippers} and also, in the quadratic case, the height equality in \eqref{e:e.x-y}.
	
	We call $X(\pi)$ and $Y(\pi)$ the \emph{cone of widths} and the \emph{cone of heights}, respectively, for $\pi$.
\end{definition}

If we replace the strict inequalities by non-strict ones, we obtain the \emph{closed cones} $\close{X}(\pi)$ and $\close{Y}(\pi)$.  
We now suppose that $\pi$ is irreducible, as defined by~\cite[Definition~3.1]{Boi-Lan}. 
Then we have the following.
\begin{itemize}
	\item 
	By~\cite[Theorem 3.2]{Boi-Lan}, the cones $X(\pi)$ and $Y(\pi)$ are non-empty. 
	\item 
	Furthermore, if $\pi$ is abelian then the cones have dimension $|\calA|$.  
	If $\pi$ is quadratic then they have dimension $|\calA| - 1$. 
	\item
	The closed cones are the closures of the open ones.  
	\item 
	None of the cones (open or closed) contain a (non-trivial) linear subspace.
\end{itemize}

In the abelian case we define $\lambda_\pi$ to be the restriction of the product measure $\leb_{W(\pi)} \times \leb_{H(\pi)}$ to $X(\pi) \cross Y(\pi)$.
Note that $\lambda_\pi$ is invariant under the diagonal flow.

In the quadratic case we define the function $w_\flip \from V(\pi) \to \RR$ to be the sum of the widths of the top flip letters minus the sum of the widths of the bottom flip letters. 
We define $h_\flip \from V(\pi) \to \RR$ similarly on heights. 
We define $F = (w_\flip, h_\flip)$.
Note that $X(\pi) \cross Y(\pi)$ is an open subset of the zero-set of $F$.
Flowing for time $t$ scales $w_\flip$ by $e^t$ and $h_\flip$ by $e^{-t}$.
We deduce that $\diff F = \diff w_\flip \wedge \diff h_\flip$ is invariant under the diagonal flow. 
We define the measure $\lambda_\pi$ on $X(\pi) \cross Y(\pi)$ by disintegrating $\leb_{W(\pi)} \times \leb_{H(\pi)}$ with respect to $F$.  
Since both $\diff x \wedge \diff y$ and $\diff F$ are invariant under the diagonal flow, the same holds for $\lambda_\pi$.
We may also restrict $w_\flip$ and $h_\flip$ to $X(\pi)$ and $Y(\pi)$, respectively. 
If we disintegrate $\leb_{W(\pi)}$ with respect to $w_\flip$ we get a measure $\lambda_{X(\pi)}$ on $X(\pi)$.
Similarly, if we disintegrate $\leb_{H(\pi)}$ with respect to $h_\flip$ we get a measure $\lambda_{Y(\pi)}$ on $Y(\pi)$.
It follows that $\lambda_\pi$ is the product measure $\lambda_{X(\pi)} \cross \lambda_{Y(\pi)}$.

\subsection{Invariance under Rauzy--Veech moves}
\label{s:RV_invariance}

Suppose that the generalised permutations $\pi$ and $\pi'$ are related by a single Rauzy--Veech move. 
Let $E$ be the associated Rauzy--Veech matrix.
From \Cref{conv:column-vectors} the induced Rauzy--Veech map $\RV \from V(\pi') \to V(\pi)$ is given by $\RV (x', y') = (E x', Ey')$.
Note that, by \Cref{e:parameter-change} the matrix $E$ is unimodular;
thus $\RV$ preserves volume (as given by the forms $\diff x' \wedge \diff y'$ and $\diff x \wedge \diff y$).

In the abelian case we deduce that the pull-back of $\lambda_\pi$ by $\RV^*$ gives $\lambda_{\pi'}$.

The quadratic case is slightly harder. 
By \Cref{e:parameter-change}, and by working through the two relevant cases (where the losing letter is either a translation or a flip letter), we deduce that the pull-back of the function $w_\flip$ (on $V(\pi)$) by $\RV^*$ gives $w'_\flip$ (on $V(\pi')$).
The same holds for $h_\flip$ and $h'_\flip$.
We deduce that, in the quadratic case, the pull-back of $\lambda_\pi$ by $\RV^*$ gives $\lambda_{\pi'}$.

\subsection{The area-one locus}

Suppose that $\pi$ is a generalised permutation. 
The function $q_\pi$ extends naturally to a map from $X(\pi) \cross Y(\pi)$ to $\calC_\root$. 
We fix a parameter $(x, y) \in X(\pi) \cross Y(\pi)$ and set $q = q_\pi(x, y)$. 
The area of the differential $q$ is the sum of the areas of the rectangles $R_\alpha$ appearing in the zippered rectangle decomposition.
Let $h_\alpha$ be the height of $R_\alpha$.
Thus the area of $q$ is
\[
A(x, y) = \sum_{\alpha \in \calA} x_\alpha h_\alpha
\]
By \Cref{e:height_top_zipper,e:height_bottom_zipper,e:rectangle height translation letter,e:rectangle height flip letter}, each $h_\alpha$ is a linear combination of the height parameters $(y_\beta)$. 
Thus $A \from X(\pi) \cross Y(\pi) \to \RR$ is bilinear; 
we deduce that $A$ is an invariant of the diagonal flow.

By disintegrating $\lambda_\pi$ with respect to the area $A$ we get a flow-invariant measure $\lambda^{(1)}_\pi$ on the unit area locus in $X(\pi) \cross Y(\pi)$.

The area $A$ is also invariant under Rauzy--Veech moves. 
Fix one such; say $\RV \from V(\pi') \to V(\pi)$.
Since $\lambda_\pi$ pulls back by $\RV^*$ to $\lambda_{\pi'}$, it follows that $\lambda^{(1)}_\pi$ pulls back by $\RV^*$ to $\lambda^{(1)}_{\pi'}$.

Recall from \Cref{d:polytope_of_parameters} that $P(\pi)$ is the polytope of parameters: the open subset of $X(\pi) \cross Y(\pi)$ cut out by the distinguished base-arc inequalities (\Cref{d:polytope_of_parameters} and \Cref{e:distinguished}).
Note that, from \Cref{d:polytope_of_differentials}, we have that $\calC(\pi) = q_\pi(P(\pi))$ is the polytope of differentials. 

We define $P^{(1)}(\pi) \subset P(\pi)$ to be set of parameters with $A(x, y) = 1$.
Recall that $P = \bigsqcup_\pi P(\pi)$. 
We denote the map $\bigsqcup_\pi q_\pi$ by $q_P$.

\begin{notation}
	We take $P^{(1)} = \bigsqcup_\pi P^{(1)}(\pi)$.
\end{notation}

We define $\lambda_P$ and $\lambda^{(1)}_P$ by restricting the measures $\lambda_\pi$ and $\lambda^{(1)}_\pi$ (as $\pi$ varies) and then taking the appropriate unions.
Note that both are invariant under the diagonal flow and also under the Rauzy--Veech moves.
We call $\lambda_P^{(1)}$ the \emph{Veech measure}.

\begin{remark}
	\label{r:Veech_measure_ergodic}
	Veech proved that the diagonal flow on $P^{(1)}$ is weak mixing, hence ergodic, for $\lambda_P^{(1)}$~\cite[Theorem 6.13]{Vee86}.
	
	By \Cref{r:calC-lab}, $\calD_\lab$ is a regular cover of $\calD_\root$. 
	Let $d_\lab$ be the degree of this cover.
	Note that this is the number of reindexings $s \in \Sym(\calA)$ such that $\pi_s = \pi \circ s$ is in the same component of $\calD_\lab$ as $\pi$.
	
	Recall that $q_P= \bigsqcup_\pi q_\pi$. 
	We now define the measure $\lambda^{(1)}$ on $\calC^{(1)}_\root$ by 
	\begin{equation}
		\label{d:Veech_measure}
		\lambda^{(1)} = \frac{1}{d_\lab} (q_P)_\ast \lambda^{(1)}_P
	\end{equation}
	We again call the image \emph{Veech measure}.
	It follows that this is ergodic for the diagonal flow. 
\end{remark}

\subsection{Induced measures on backwards flow faces}
\label{s:measures on backwards flow faces}

We now produce a measure suitable for our coding. 

We define $w \from X(\pi) \cross Y(\pi) \to \RR$ to be the sum of the widths.
From \Cref{d:backwards-flow-face} we have that the backwards flow face $\bdy^- P(\pi)$ is a level set of $w$, namely:
\[
\{(x, y) \in X(\pi) \cross Y(\pi) \st w(x, y) = 1\}
\]
is the backwards flow face $\bdy^- P(\pi)$. 
Disintegrating $\lambda_\pi$ with respect to $w$ we obtain a measure $\nu_\pi$ on $\bdy^- P(\pi)$.

Let $w_X \from X(\pi) \to \RR$ be, again, the sum of the widths. 
We define:
\[
X_1(\pi) = \{ x \in X(\pi) \st w_X(x) = 1\}
\]
Disintegrating the measure $\lambda_{X(\pi)}$ with respect to $w_X$ we obtain a measure $\nu_{X_1(\pi)}$ on $X_1(\pi)$.


\begin{lemma}
	\label{l:nu_X_one_finite}
	The measure $\nu_{X_1(\pi)}$ on $X_1(\pi)$ is finite. \qed
\end{lemma}

\begin{proof}
	We define $L$ to be the subspace of $W(\pi)$ given by $w_\flip = 0$. 
	Note that the cone $X(\pi)$ is contained in $L$.
	We disintegrate $\leb_{W(\pi)}$ with respect to $w_\flip$ to obtain the measure $\lambda_L$ on $L$.
	Note that $\lambda_{X(\pi)}$ is the restriction of $\lambda_L$ to $X(\pi)$.
	
	We define $L_1$ to be the affine subspace of $L$ given by $w_X = 1$. 
	Note that the polytope $X_1(\pi)$ is contained in $L_1$.
	We define $\nu_{L_1}$ to be the measure on $L_1$ given by disintegration of $\lambda_L$ with respect to $w_X$.
	Note that $\nu_{X_1(\pi)}$ is the restriction of $\nu_{L_1}$ to $X_1(\pi)$.
	
	Suppose that $v$ is any vector in $L$ such that $w_X(v) = 0$. 
	Translation by $v$ preserves $L_1$ and leaves $\nu_{L_1}$ invariant.
	Since $\nu_{L_1}$ is a smooth measure, it is a scalar multiple of the Lebesgue measure on $L_1$.
	
	The lemma now follows from the fact that $\closure{X_1(\pi)}$ is compact in $L_1$.
\end{proof}

Note that $\bdy^- P(\pi)$ equals the product $X_1(\pi) \cross Y(\pi)$.
Disintegration commutes with the product structure; 
thus the measure $\nu_\pi$ is the product measure $\nu_{X_1(\pi)} \cross \lambda_{Y(\pi)}$.

Suppose that $\pi$ to $\pi'$ is a Rauzy--Veech move.
Breaking symmetry, suppose that $\pi \to \pi'$ is a top move with winning label $\alpha$, losing label $\beta$, and matrix $E$. 
Let $\RV \from V(\pi') \to V(\pi)$ be the induced map.
We define the following subcones:
\[
Y^\alpha(\pi') = \{ y' \in Y'(\pi) \st y'_\alpha < 0 \}
\quad
\mbox{and}
\quad
X^\beta_1(\pi)  = \{ x \in X_1(\pi) \st x_\beta < x_\alpha \}
\]
We define the \emph{Rauzy--Veech renormalisation} map $\RV^\bdy \from X_1(\pi') \cross Y^\alpha(\pi') \to X^\beta_1(\pi) \cross Y(\pi)$ by
\[
\RV^\bdy(x', y') = \left( \frac{1}{w(Ex')} \cdot Ex', w(Ex') \cdot Ey' \right)
\]
(Here we suppress the second coordinate $y$ inside of $w$.)
Note that this is a ``bi-projective'' homeomorphism.

\begin{lemma}
	\label{l:RVinvariance}
	With notation as above: $\RV^\bdy_* (\nu_{\pi'}) = \nu_\pi$.
\end{lemma}

\begin{proof}
	We take $w' \from X(\pi') \cross Y(\pi') \to \RR$ to be the sums of the widths.
	We define the map $E^\bdy \from X(\pi') \cross Y^\alpha(\pi') \to X^\beta(\pi) \cross Y(\pi)$ as follows:
	\[
	E^\bdy(x', y') = \left( \frac{w'(x')}{w(E x')} \cdot E x', \frac{w(Ex')}{w'(x')} \cdot E y'\right)
	\]
	(Again, we may suppress the height coordinates inside of $w$ and $w'$.)
	Note that $E^\bdy$ is an extension of $\RV^\bdy$. 
	
	Recall that the Rauzy matrix $E$ is unimodular.  
	Note that, in $E^\bdy$, the widths and heights are scaled by reciprocal factors.
	We deduce that $\diff \lambda_\pi$ pulls back, via $E^\bdy$, to $\diff \lambda_{\pi'}$.
	
	Since $w$ is linear, we have:
	\[
	w(E^\bdy(x', y')) = w \left( \frac{w'(x')}{w(Ex')} \cdot Ex' \right) = w'(x', y')
	\]
	That is, $\diff w$ pulls back, via $E^\bdy$, to $\diff w'$.
	
	Let $\omega$ be any form on $X(\pi) \cross Y(\pi)$ so that $\omega \wedge \diff w = \diff \lambda_\pi$. 
	Note that, as discussed in \Cref{s:disintegration}, integrating against $\omega$ gives the measure $\nu_\pi$. 
	Define $\omega'$ to be the pullback, via $E^\bdy$, of $\omega$.  
	It follows that $\omega' \wedge \diff w' = \diff \lambda_{\pi'}$. 
	So, by \Cref{s:disintegration}, the form $\omega'$ gives the measure $\nu_{\pi'}$.
	Since $E^\bdy$ is an extension of $\RV^\bdy$, and since ``pulling back'' is 
	functorial, the lemma is proved.
\end{proof}

\begin{remark}
	\label{r:coning}
	Instead of defining $\nu_\pi$ via disintegration, we could define it through a coning construction, as follows. 
	Suppose that $U$ is a small open set in $\bdy^- P(\pi)$. 
	We define $C(U)$ to be its cone to the origin in $X(\pi) \cross Y(\pi)$. 
	We now take $\nu_\pi(U)$ equal to $\lambda_\pi(C(U))$.
	Under this definition, the fact that $\nu_\pi$ is invariant under $\RV^\bdy$ (that is, \Cref{l:RVinvariance}) follows directly from the flow-invariance and the Rauzy--Veech invariance of $\lambda_\pi$.
	A standard exercise then shows that this measure, obtained by coning, is scalar multiple of the measure obtained by disintegration. 
\end{remark}

In a small abuse of notation we now use $\RV^\bdy$ to denote the disjoint union of all of the Rauzy--Veech renormalisation maps.  
We define $\nu_P$ to be the induced measure on the union $\bdy^- P$.  
By further disintegration, with respect to the area $A$, we obtain a measure $\nu^{(1)}_P$ on $\bdy^- P^{(1)}$.

\begin{remark}\label{r:nu_P_invariant}
	Since the area and $\nu_P$ are invariant under $\RV^\bdy$ the same holds for $\nu^{(1)}_P$.
\end{remark}

\subsection{Projecting to widths}
\label{s:projecting to widths}

Fix $x \in X_1(\pi)$.
Note that all coordinates of $x$ are positive. 
We define the function $A_x \from Y(\pi) \to \RR$ by $A_x (y) = A(x,y)$.
We define the following:
\begin{align*}
	Y(\pi, x) &= \{ y \in Y(\pi) \st A(x, y) = 1 \} \\
	\close{Y}(\pi, x) &= \{ y \in \close{Y}(\pi) \st A(x, y) = 1 \}
\end{align*}
Disintegrating $\lambda_{Y(\pi)}$ with respect to the function $A_x$, we get a measure $\nu_{Y(\pi,x)}$ on $Y(\pi, x)$.

\begin{lemma}
	\label{l:cross_section}
	Suppose that $y \in \close{Y}(\pi)$ is a non-zero vector.  
	Then $A_x(y)$ is positive. 
\end{lemma}

It follows that $\close{Y}(\pi, x)$ meets all rays in $\close{Y}(\pi)$.  
Thus $\close{Y}(\pi, x)$ and $\close{Y}(\pi, x')$ are projectively equivalent for all $x, x' \in X_1(\pi)$.

\begin{proof}[Proof of \Cref{l:cross_section}]
	Recall that $x$ lies in $X_1(\pi)$.
	Since $X_i(\pi)$ is a subset of $X(\pi)$, all widths in $x$ are positive.
	
	Breaking symmetry, 
	we may assume that there is a label $\alpha$ which is a top label, which is not last on top, and so that the singularity height $y_\alpha$ is positive.  
	We further may assume that $y_\alpha$ is the first such; 
	that is, all labels $\beta$ on top and to the left of $\alpha$ have height $y_\beta = 0$. 
	We deduce that $A_x(y) \geq x_\alpha y_\alpha$.  
	Thus, the area $A_x(y)$ is positive.
\end{proof}

\begin{lemma}
	\label{l:fibres-finite-vol}
	For any irreducible generalised permutation $\pi$, and for any $x \in X_1(\pi)$, we have
	\[
	\nu_{Y(\pi,x)} (Y(\pi,x)) < \infty
	\]
\end{lemma}

\begin{proof}
	Let $L(x)$ be the hyperplane in $H(\pi)$ defined by $A_x(y) = 1$.
	Thus $Y(\pi, x) = Y(\pi) \cap L(x)$ and $\close{Y}(\pi, x) = \close{Y}(\pi) \cap L(x)$.
	Note that $\nu_{Y(\pi,x)}$ is the restriction of a Lebesgue measure $\nu_x$ on $L(x)$. 
	So it suffices to prove that $\close{Y}(\pi, x)$ is compact in $L(x)$.
	
	Note that $\close{Y}(\pi)$ is an intersection of finitely many closed half-spaces.
	By the main theorem for cones~\cite[Theorem~1.3]{Ziegler95} there is a finite list $(v^i)_{i = 1}^n$ of vectors in $\close{Y}(\pi)$ so that every $v \in \close{Y}(\pi)$ is a non-negative linear combination of the vectors $(v^i)_i$. 
	We may assume that the zero vector is not in this collection. 
	
	We define $w^i = v^i / A_x(v^i)$;
	this is well-defined by \Cref{l:cross_section}. 
	Note that the vectors $(w^i)_i$ again generate $\close{Y}(\pi)$. 
	Also, these vectors lie in $L(x)$.
	
	Recall that 
	\[
	\Delta^n = \left\{ a \in \RR^n \,\middle|\, a_i \geq 0, \sum a_i = 1 \right\}
	\]
	is the standard simplex in $\RR^n$. 
	Note that $\Delta^n$ is compact.
	We now define $D \from \Delta^n \to \close{Y}(\pi)$ by $D(a) = \sum a_i w^i$.
	Since $D$ is continuous, the image of $\Delta^n$ in $\close{Y}(\pi)$ is compact.
	
	\begin{claim}
		$D(\Delta^n) = \close{Y}(\pi, x)$.
	\end{claim}
	
	\begin{proof}
		Suppose that $a$ lies in $\Delta^n$.
		So $A_x(D(a)) = A_x(\sum a_i w^i) = \sum a_i A_x(w^i) = 1$. 
		Thus $D(a)$ lies in $\close{Y}(\pi, x)$.
		
		Suppose that $v$ lies in $\close{Y}(\pi, x)$.
		Thus $v = \sum E_i w^i$ for some $E_i$, which are non-negative.  
		Thus $1 = A_x(v) = A_x(\sum E_i w^i) =  \sum E_i A_x(w^i) = \sum E_i$. 
		Thus $D(b) = v$.
	\end{proof}
	
	This completes the proof of the lemma.
\end{proof}

By \Cref{l:fibres-finite-vol} the volume 
\[
\vol_\pi(x) = \nu_{Y(\pi,x)} (Y(\pi,x))
\]
is finite. 
Thus $\vol_\pi(x)$ is a smooth function of $x$.
From \Cref{r:coning} we deduce that the volume scales as 
\[
\vol_\pi(rx) = \vol_\pi(x) / r^{\dim Y(\pi)}
\]
So we may pick a basepoint $x_0 \in X_1(\pi)$ with $\vol_\pi(x_0) = 1$.  
That is, $\nu_{Y(\pi,x_0)}$ is a probability measure. 
Thus the product measure $\nu_{X_1(\pi)} \cross \nu_{Y(\pi,x_0)}$ on $X_1(\pi) \cross Y(\pi, x_0)$, under projection to widths, pushes forward to $\nu_{X_1(\pi)}$.

Recall that $Y(\pi, x_0)$ and $Y(\pi, x)$ are projectively equivalent, as both meet all rays in $Y(\pi)$.
Let $f_x \from Y(\pi, x) \to Y(\pi, x_0)$ be the resulting diffeomorphism. 
By \Cref{r:coning}, the measures $\nu_{Y(\pi,x)}$ and $\nu_{Y(\pi,x_0)}$ agree with the measures obtained by coning to the origin in $Y(\pi)$.
Hence, the Radon--Nikodym derivative of the pushforward measure $(f_x)_* \nu_{Y(\pi,x)}$ with respect to $\nu_{Y(\pi,x_0)}$ is constant over $Y(\pi,x_0)$.
Since $\vol_\pi(x_0) = 1$, we deduce that this constant is $\vol_\pi(x)$.

We define $p^\pi \from \bdy^- P^{(1)}(\pi) \to X_1(\pi)$ to be the projection to the widths. 
Note that this is a surjection.
The fibre of $p^\pi$ over $x \in X_1(\pi)$ is exactly $\{x \} \times Y(\pi, x)$.
The fibrewise diffeomorphisms $f_x$ vary smoothly with $x$.
Thus we obtain a diffeomorphism $f \from \bdy^- P^{(1)}(\pi) \to X_1(\pi) \cross Y(\pi, x_0)$.
From the paragraph above, we conclude that the pushforward measure $f_* \nu^{(1)}_\pi$ is absolutely continuous with respect to the product measure $\nu_{X_1(\pi)} \cross \nu_{Y(\pi,x_0)}$, and has density $\vol_\pi$.

Let $\phi_{X_1(\pi)} = p^\pi_* \nu_\pi^{(1)}$ be the pushforward measure on $X_1(\pi)$. 
We deduce the following:
\begin{equation}
	\label{e:density}
	\frac{\diff \phi_{X_1(\pi)}}{\diff \nu_{X_1(\pi)}} = \vol_\pi
\end{equation}

Suppose now that $\pi \to \pi'$ is a Rauzy--Veech move.
Breaking symmetry, we suppose that it is a top move, with winning label $\alpha$, losing label $\beta$, and matrix $E$.
We define the \emph{Rauzy renormalisation map} $\R^\bdy \from X_1(\pi') \to X^\beta_1(\pi)$ by 
\[
\R^\bdy (x') = \frac{E x'} {w_X(E x')}
\]
From \Cref{l:RVinvariance}, we deduce the following.

\begin{lemma}
	\label{l:R-invariance}
	With notation as above:
	\[
	\pushQED{\qed}
	\R^\bdy_* (\phi_{X_1(\pi')}) = \phi_{X_1(\pi)} \qedhere
	\popQED
	\]
\end{lemma}


In a small abuse of notation we now use $\R^\bdy$ to denote the union of all Rauzy renormalisation maps.  
We also define $\phi$ to be the induced measure on $X_1 = \bigsqcup_\pi X_1 (\pi)$.
Similarly, we set $\vol = \bigsqcup_\pi \vol_\pi$.
By \Cref{l:R-invariance}, the measure $\phi$ is invariant for $\R^\bdy$.

Recall by \Cref{l:flow-face-parameters} that $\bdy^- \calC(\pi) = q_\pi(\bdy^- P(\pi))$. 
We similarly define $\bdy^- \calC^{(1)}(\pi)$ to be $q_\pi( \bdy^- P^{(1)}(\pi))$. 
We now define $\bdy^- \calC_\root$ to be the union $\bigsqcup_\pi \bdy^- \calC(\pi)$.
Also we define $\bdy^- \calC^{(1)}_\root$ to be the union $\bigsqcup_\pi \bdy^- \calC^{(1)}(\pi)$.

Recall that $q_P = \bigsqcup_\pi q_\pi$ and that $d_\lab$ is the degree of the cover $\calD_\lab \to \calD_\root$.

\begin{definition}
	\label{d:Veech_and_Rauzy--Veech}
	We define measures $\nu$ and $\nu^{(1)}$ on $\bdy^- \calC$ and $\bdy^- \calC^{(1)}$ by 
	\begin{equation}
		\label{e:measures_on_rooted_differentials}
		\nu = \frac{1}{d_\lab} (q_P)_\ast \nu_P \, \text{ and } \nu^{(1)} = \frac{1}{d_\lab} (q_P)_\ast \nu_P^{(1)} 
	\end{equation}
	We call the measure $\nu^{(1)}$ the \emph{Veech measure} on $\bdy^- \calC^{(1)}_\root$ and the measure $\nu^{(1)}$ the \emph{Rauzy--Veech measure} on $\bdy^- \calC^{(1)}_\root$.
\end{definition}

\subsection{Summary}
\label{s:measures-summary}

The previous discussion gives us two measures, $\phi_{X_1}$ and $\nu_{X_1}$, on $X_1$ that are absolutely continuous with respect to  each other.
We showed that the density of $\phi_{X_1}$ with respect to $\nu_{X_1}$ is in fact the volume of the fiber in $Y$, and so is smooth. 
The measure $\phi_{X_1}$ is invariant under the Rauzy renormalisation $R_\bdy$, while $\nu_{X_1}$ is not. 
However, $\nu_{X_1}$ is easier to handle in estimates required to get the appropriate coding (in \Cref{s:dynamics}) for the diagonal flow. 
The coding is built with a pre-compact Poincaré section and hence its projection to widths is pre-compact in $X_1$.
We then use the smoothness of $\vol_\pi$ to conclude that up to a uniform constant (that depends on the choice of section) the estimates hold for $\phi_{X_1}$. 

\begin{figure}
	\[
	\begin{tikzcd}
		& & \text{Real dimension} \\
		& (W \times H, \leb_W \times \leb_H, g_t) & {2|\calA|} \\
		(P, \lambda_P, g_t) \arrow[hook,"\text{inclusion}"]{r} & (X \times Y, \lambda_X \times \lambda_Y, \mbox{$g_t$ and $\RV$}) \arrow[hook,"w_{\text{flip}} \times h_{\text{flip}}"]{u} & {2|\calA| - 2} \\
		(P^{(1)}, \lambda_P^{(1)}, g_t) \arrow[hook,"A"]{u} & & {2|\calA| - 3} \\
		(\bdy^- P^{(1)}, \nu_P^{(1)}, \RV^\bdy) \arrow[hook,"w"]{u} \arrow[two heads,"p"]{dd} & & {2|\calA| - 4} \\
		& (X, \lambda_X) & {|\calA| - 1}  \\
		(X_1, \phi_{X_1}, \R^\bdy) \arrow[leftarrow,"\mbox{$(\text{id},\vol)$}"]{r} & 
		(X_1, \nu_{X_1}) \arrow[hook,"w_X"]{u} & {|\calA| - 2}
	\end{tikzcd}
	\]
	\caption{The various dynamical systems (and measure spaces) discussed above, in the quadratic case.  For the most part the arrows are decorated with the function used to disintegrate the earlier measure and so obtain the later measure.}
	\label{f:measure_diagram}
\end{figure}

The diagram in \Cref{f:measure_diagram} summarises the construction of these measures in the quadratic case. 
We start with the Lebesgue measure $\leb_W \cross \leb_H$ on the disjoint union $W \cross H = \bigsqcup_\pi W(\pi) \cross H(\pi)$. 
This measure is invariant under the diagonal flow $g_t$.
We next disintegrate with respect to $w_\flip \cross h_\flip$ to give the measure $\lambda_X \cross \lambda_Y$ on the disjoint union $X \cross Y = \bigsqcup_\pi X(\pi) \cross Y(\pi)$.
By \Cref{s:polytopal-measures-full} and by \Cref{s:RV_invariance}, this measure is invariant under the diagonal flow and under the Rauzy--Veech moves, here denoted $\RV$.
We restrict $\lambda_X \cross \lambda_Y$ to $P = \bigsqcup_\pi P(\pi)$ to obtain the measure $\lambda_P$; 
thus this is again invariant for the diagonal flow. 
We then disintegrate $\lambda_P$ with respect the area $A$ to obtain a $g_t$--invariant measure $\lambda^{(1)}_P$ on the area-one locus $P^{(1)}$.
By \Cref{r:Veech_measure_ergodic}, this gives us the \emph{Veech measure} $\lambda^{(1)}$ on $\calC^{(1)}_\root$.

We disintegrate $\lambda^{(1)}_P$ again, this time with respect to the sum of the widths $w$, to obtain a measure $\nu_P^{(1)}$, which is now $\RV^\bdy$--invariant.
By \Cref{d:Veech_and_Rauzy--Veech}, this gives us the \emph{Rauzy--Veech} measure $\nu^{(1)}$ on $\bdy^- \calC^{(1)}$.

Recall that $X_1 = \bigsqcup_\pi X_1(\pi)$. 
The measure $\nu_P^{(1)}$ is pushed forward to a measure $\phi_{X_1}$ on $X_1$ via the projection $p \from \bdy^- P^{(1)} \to X_1$: the disjoint union of the projection maps $p^\pi$.
The measure $\phi_{X_1}$ is now $\R^\bdy$--invariant;  
this gives what we call \emph{Rauzy} measure.

We now may state one of the major results in the area.

\begin{theorem}
	\label{t:all-ergodic}
	The dynamical systems
	$(X_1, \phi_{X_1}, \R^\bdy)$, 
	$(\bdy^- \calC^{(1)}, \nu^{(1)}, \RV^\bdy)$, and 
	$(\calC^{(1)}, \lambda^{(1)}, g_t)$ are ergodic. \qed
\end{theorem}

For the interval exchange transformations, and for abelian differentials,  see~\cite[Corollaries~27.2 and~27.3]{Via06}.
For non-classical interval exchange transformations, the ergodicity of $\R^\bdy$ is proved in~\cite[Theorem 13.1]{Gad}.
Viana~\cite{Via06} deduces the ergodicity of the systems acting on abelian differentials from the case of interval exchange transformations.
The same proof carries over for quadratic components.

There is one more measure to consider. 
We directly disintegrate the Lebesgue measure $\lambda_X$ on the open cone $X$ with respect to the sum of the widths $w_X$ to obtain $\nu_{X_1}$.
This measure \emph{is} finite; 
it is essentially a Lebesgue measure on the interior of a compact polytope. 
However, it is \emph{not} $\R^\bdy$--invariant; 
this is because its construction only takes widths (and not heights) into account. 
Nevertheless, it is absolutely continuous, with a smooth (hence finite) density, with respect to $\phi_{X_1}$.

\subsection{Masur--Smillie--Veech measure}

We now define the \emph{Masur--Smillie measure} (see \cite[Proposition 4.4]{Mas} for abelian components and the principal quadratic stratum and see \cite[Section 1]{Mas-Smi} for general quadratic strata).
This is a diagonal flow-invariant measure on $\cover{\calC}$, the locus of abelian differentials that are orientation covers of differentials in $\calC$.
It is then pushed forward to $\calC$ (where it receives the same name). 

The lattice $\upH_1(\cover{S}, \cover{Z}; \ZZ)$ splits integrally over the minus and plus pieces in $\upH_1(\cover{S}, \cover{Z}; \CC)$. 
We normalise the Lebesgue measure on $\upH_1(\cover{S}, \cover{Z}; \CC)$ so that the covolume of the integer lattice in the minus piece is one.
Pulling the normalised measure (in the minus piece) back by period coordinates, we get a measure on $\cover{\calC}$. 
As a change of period coordinates is unimodular (in fact, symplectic) the measure is well-defined. 
As the diagonal flow stretches real periods by $e^t$ and contracts imaginary periods by $e^{-t}$ the measure is flow-invariant. 
Disintegrating with respect to area gives the \emph{Masur--Smillie measure}.
Since it is defined by periods, after lifting to $\calC_\root$ we find that the Masur--Smillie measure is absolutely continuous with respect to the Veech measure.
By the ergodic theorem, it follows that the Masur--Smillie and Veech measures are multiples of each other. 
We call the probability measure in this measure class the \emph{Masur--Smillie--Veech} measure.

\sloppy\printbibliography
\end{document}

%% file: settings.tex
\usepackage{microtype}
\usepackage[utf8]{inputenc}
\usepackage[T1]{fontenc}

\usepackage[top=25mm, bottom=25mm, left=40mm, right=40mm]{geometry}  

\usepackage{tikz}
\usepackage{xcolor}
\usetikzlibrary{calc}
\usetikzlibrary{decorations.markings}
\usetikzlibrary{arrows}
\usetikzlibrary{patterns}
\usepackage{tikz-cd}

\usepackage{amsmath, amssymb, amscd, color}
\usepackage{mathtools} 
\usepackage{enumitem} 

\usepackage{graphicx}
\usepackage[hidelinks]{hyperref}

\makeatletter
\define@key{href}{font}{#1}
\makeatother
\usepackage{xpatch}
\newcommand\hrefdefaultfont{\ttfamily}
\xpatchcmd\href{\setkeys{href}{#1}}{\setkeys{href}{font=\hrefdefaultfont,#1}}{}{\fail}

\usepackage{comment}
\usepackage{subcaption}
\usepackage{float}
\usepackage{cleveref}
\usepackage{stmaryrd} 
\usepackage{longtable}
\usepackage[Symbol]{upgreek}

\usepackage{todonotes}
\usepackage[style=alphabetic,natbib=true,giveninits=true,url=false,maxbibnames=99]{biblatex}

\usepackage[scaled=0.9]{sourcecodepro} 

\DeclareFieldFormat{title}{\mkbibquote{#1}}
\DeclareFieldFormat[online]{date}{Preprint (#1)}
\DeclareFieldFormat[misc]{date}{Personal communication (#1)}
\AtEveryBibitem{\clearfield{issn}}
\AtEveryCitekey{\clearfield{issn}}

\makeatletter
\let\c@equation\c@subsection
\makeatother
\numberwithin{equation}{section}

\makeatletter
\let\c@figure\c@equation
\makeatother
\numberwithin{figure}{section}

\theoremstyle{plain}
\newtheorem{theorem}[equation]{Theorem}
\newtheorem{corollary}[equation]{Corollary}
\newtheorem{lemma}[equation]{Lemma}

\newtheorem{proposition}[equation]{Proposition}

\newtheorem{axiom/}[equation]{Axiom}

\theoremstyle{definition}
\newtheorem{definition/}[equation]{Definition}

\newtheorem{example/}[equation]{Example}
\newtheorem{fact*}{Fact}
\newtheorem{fact/}[equation]{Fact}
\newtheorem{question/}[equation]{Question}
\newtheorem*{question*}{Question}
\newtheorem*{answer*}{Answer}
\newtheorem*{application*}{Application}

\newtheorem{convention/}[equation]{Convention}
\newtheorem{criterion/}[equation]{Criterion}

\newenvironment{definition}
  { 
   \pushQED{\qed}\begin{definition/}}
  {\popQED\end{definition/}}

\newenvironment{example}
  {%
   \pushQED{\qed}\begin{example/}}
  {\popQED\end{example/}}

\newenvironment{question}
  {%
   \pushQED{\qed}\begin{question/}}
  {\popQED\end{question/}}

\newenvironment{convention}
  {%
   \pushQED{\qed}\begin{convention/}}
  {\popQED\end{convention/}}

\newenvironment{criterion}
  {%
   \pushQED{\qed}\begin{criterion/}}
  {\popQED\end{criterion/}}

\theoremstyle{remark}
\newtheorem{remark/}[equation]{Remark}
\newtheorem{notation/}[equation]{Notation}
\newtheorem*{remark*}{Remark}

\newtheorem*{case*}{Case}

\newtheorem{claim}[equation]{Claim}
\newtheorem*{claim*}{Claim}
\newtheorem*{subclaim*}{Subclaim}

\newenvironment{remark}
  {%
   \pushQED{\qed}\begin{remark/}}
  {\popQED\end{remark/}}

\newenvironment{notation}
  {%
   \pushQED{\qed}\begin{notation/}}
  {\popQED\end{notation/}}

\usepackage{colonequals}

\usepackage{booktabs}
\DeclareMathOperator{\vol}{Vol}

\DeclareMathOperator{\SL}{SL}

\DeclareMathOperator{\abs}{ab}

\DeclareMathOperator{\Aut}{Aut}
\DeclareMathOperator{\End}{End}
\DeclareMathOperator{\PAut}{PAut}
\DeclareMathOperator{\Sym}{Sym}
\DeclareMathOperator{\Sp}{Sp}
\DeclareMathOperator{\RV}{RV}
\DeclareMathOperator{\R}{R}

\DeclareMathOperator{\genus}{genus}
\DeclareMathOperator{\sgn}{sgn}
\DeclareMathOperator{\col}{col}

\DeclareMathOperator{\Mat}{Mat}

\newcommand{\KZ}{\mathrm{KZ}}

\newcommand{\from}{\colon}

\renewcommand{\geq}{\geqslant}
\renewcommand{\leq}{\leqslant}
\newcommand{\cross}{\times}

\newcommand{\st}{\mid}
\newcommand{\calA}{\mathcal{A}}
\newcommand{\calB}{\mathcal{B}}
\newcommand{\calC}{\mathcal{C}}
\newcommand{\calD}{\mathcal{D}}
\newcommand{\calE}{\mathcal{E}}

\newcommand{\calH}{\mathcal{H}}
\newcommand{\calJ}{\mathcal{J}}
\newcommand{\calL}{\mathcal{L}}
\newcommand{\calM}{\mathcal{M}}
\newcommand{\calN}{\mathcal{N}}
\newcommand{\calO}{\mathcal{O}}

\newcommand{\calQ}{\mathcal{Q}}
\newcommand{\calR}{\mathcal{R}}
\newcommand{\calS}{\mathcal{S}}
\newcommand{\calT}{\mathcal{T}}

\newcommand{\calV}{\mathcal{V}}
\newcommand{\calW}{\mathcal{W}}
\newcommand{\CC}{\mathbb{C}}
\newcommand{\NN}{\mathbb{N}}

\newcommand{\RR}{\mathbb{R}}

\newcommand{\ZZ}{\mathbb{Z}}
\renewcommand{\Re}{\mathrm{Re}}
\renewcommand{\Im}{\mathrm{Im}}

\newcommand{\upB}{\mathrm{B}}
\newcommand{\upC}{\mathrm{C}}
\newcommand{\upD}{\mathrm{D}}
\newcommand{\upH}{\mathrm{H}}
\newcommand{\upP}{\mathrm{P}}

\newcommand{\upT}{\mathrm{T}}

\newcommand{\closure}[1]{\overline{#1}}
\newcommand{\close}[1]{\overline{#1}}

\newcommand{\cover}[1]{{\widetilde{#1}}}
\newcommand{\bdy}{\partial}
\newcommand{\Id}{\mathrm{Id}}

\newcommand{\diff}{\mathrm{d}}

\renewcommand{\root}{\mathrm{root}}
\newcommand{\biroot}{\mathrm{biroot}}

\newcommand{\flip}{\mathrm{flip}}

\newcommand{\lab}{\mathrm{lab}}
\newcommand{\av}{\mathrm{av}}
\newcommand{\orb}{\mathrm{orb}}

\newcommand{\irreg}{\mathrm{irreg}}
\newcommand{\reg}{\mathrm{reg}}
\newcommand{\zip}{\mathrm{zip}}
\newcommand{\hyp}{\mathrm{hyp}}
\newcommand{\nonhyp}{\mathrm{nonhyp}}

\newcommand{\Rtop}{\mathrm{t}}
\newcommand{\Rbot}{\mathrm{b}}
\newcommand{\shift}{\text{\textup{S}}}
\newcommand{\leb}{\mathrm{Leb}}

\newcommand{\Flow}{\mathrm{Flow}}
\newcommand{\SFlow}{\mathrm{SFlow}}

\newcommand{\isom}{\cong}

\renewcommand{\subset}{\subseteq}
\newcommand{\It}{I_{\text{t}}}
\newcommand{\Ib}{I_{\text{b}}}

\newcommand{\Mod}{\mathrm{Mod}}

\makeatletter
\g@addto@macro\bfseries{\boldmath}
\makeatother

\allowdisplaybreaks

\newcommand{\fakeenv}{} 
\newenvironment{restate}[2]  
{
 \renewcommand{\fakeenv}{#2} 
 \theoremstyle{plain}
 \newtheorem*{\fakeenv}{#1~\ref{#2}} 
 \begin{\fakeenv}
}
{
 \end{\fakeenv}
}
